\documentclass{ws-rv9x6}
\usepackage[onethmnum]{ws-rv-thm}
\usepackage[square]{ws-rv-van}

\usepackage[english]{babel}

\usepackage[normalem]{ulem}
\usepackage[all]{xy}
\usepackage{tikz}
\usetikzlibrary{calc}
\usepackage{mathtools,stmaryrd,xspace,bussproofs}

\makeindex

\usepackage{hyperref}

\begin{document}

{
\newcommand{\hcancel}[5]{%
  \tikz[baseline=(tocancel.base)]{
    \node[inner sep=0pt,outer sep=0pt] (tocancel) {#1};
    \draw[black, line width=0.2mm] ($(tocancel.south west)+(#2,#3)$) -- ($(tocancel.north east)+(#4,#5)$);
  }%
}
\newtheorem{exercise}{Exercise}
\let\Exercisefont\upshape
\def\Exerciseheadfont{\bfseries}

\newcommand{\CCC}{\mathcal{C}}
\newcommand{\OOO}{\mathcal{O}}
\newcommand{\NN}{\mathbb{N}}
\newcommand{\QQ}{\mathbb{Q}}
\newcommand{\RR}{\mathbb{R}}
\newcommand{\TT}{\mathbb{T}}
\newcommand{\ZZ}{\mathbb{Z}}
\newcommand{\CC}{\mathbb{C}}
\newcommand{\aaa}{\mathfrak{a}}
\newcommand{\bbb}{\mathfrak{b}}
\newcommand{\ppp}{\mathfrak{p}}
\newcommand{\qqq}{\mathfrak{q}}
\newcommand{\mmm}{\mathfrak{m}}
\newcommand{\fff}{\mathfrak{f}}
\newcommand{\defeq}{\vcentcolon=}
\newcommand{\defeqv}{\vcentcolon\equiv}
\newcommand{\Sh}{\mathrm{Sh}}
\newcommand{\ab}{\mathrm{ab}}
\newcommand{\op}{\mathrm{op}}
\newcommand{\Set}{\mathrm{Set}}
\newcommand{\Hom}{\mathrm{Hom}}
\newcommand{\Spec}{\mathrm{Spec}}
\newcommand{\Max}{\mathrm{Max}}
\newcommand{\Gal}{\mathrm{Gal}}
\newcommand{\Rad}{\mathrm{Rad}}
\newcommand{\Idl}{\mathrm{Idl}}
\newcommand{\+}{\mathpunct{.}}
\newcommand{\?}{\,{:}\,}
\newcommand{\seq}[1]{\mathrel{\vdash\!\!\!_{#1}}}
\newcommand{\pt}{\mathrm{pt}}
\newcommand{\Pt}{\mathrm{Pt}}
\newcommand{\Loc}{\mathrm{Loc}}
\newcommand{\Top}{\mathrm{Top}}
\newcommand{\ann}{\operatorname{ann}}
\newcommand{\BPIT}{\textsc{bpit}\xspace}
\newcommand{\notnot}{\emph{not not}\xspace}
\newcommand{\negg}{\neg\!\!\!\neg}
\newcommand{\bott}{\bot\!\!\!\!\bot}
\newcommand{\brak}[1]{{\llbracket{#1}\rrbracket}}

\newcommand{\bracketedrefcite}[1]{\citep{#1}}

\title{Proof and Computation}

\chapter*{Generalized spaces for constructive algebra}

\author[I. Blechschmidt]{Ingo Blechschmidt}

\address{Università di Verona \\
Department of Computer Science \\
Strada le Grazie 15 \\
37134 Verona, Italy}

\begin{abstract}
The purpose of this contribution is to give a coherent account of a particular
narrative which links locales, geometric theories, sheaf semantics and
constructive commutative algebra. We are hoping to convey a firm grasp of three
ideas:
(1)~Locales are a kind of space in which opens instead of points are
fundamental.
(2)~Sheaf semantics allows us to explore mathematical objects from custom-tailored
mathematical universes.
(3)~Without loss of generality, any reduced ring is a field.
\end{abstract}
\body

\tableofcontents

\subsection*{Introduction\footnote{
Unlike some other texts on locales or sheaves, this text is set in a
constructive (but impredicative) metatheory. We do not use the law of excluded
middle nor any version of the axiom of choice, and we also do not adopt any
nonclassical principles. We freely use powersets, but more for linguistic
convenience than by necessity.
In particular, the central idea presented in
Section~\ref{sect:appl}, the reduction technique by which we can pretend that any
given reduced ring is a field, is very robust with respect to the metatheory.
Even though we chose a semantic point of view for its derivation, the final
result is a purely syntactical transformation. As it stands, this text could be
formalized in the kind of type theory which is valid in
toposes~\cite{maietti:modular-correspondence} or
in~\textsc{izf}~\cite{crosilla:cst-izf}.}}

\addcontentsline{toc}{section}{Introduction}

\paragraph{Locales.}

The notion of \emph{space} is fundamental to large parts of mathematics. This notion
exists in various flavors, ranging from the basic cartesian spaces~$\RR^n$ to
the more general metric and topological spaces and also including several more
and slightly exotic flavors such as the diffeological spaces from differential topology.

Common to all of the mentioned flavors of space is that \emph{points are their
building blocks}: cartesian, metric, topological and diffeological spaces are,
first and foremost, sets of points. Their spatial structure -- a metric, a
topology or a diffeology -- is additional data to the underlying set:

\begin{quote}
\textbf{Definition.} A \emph{metric space} consists of a set~$X$ of points
together with a map~$d : X \times X \to \RR_{\geq 0}$ satisfying the metric
axioms.

\textbf{Definition.} A \emph{topological space} consists of a set~$X$ of points
together with a set~$\OOO(X) \subseteq P(X)$ of point sets which are deemed
\emph{open} such that (arbitrary, set-indexed) unions and finite intersections
of open sets are open.
\end{quote}

Locales are a further and particularly unique flavor of the notion of space
which turn this classical picture upside down: Locales embrace \emph{opens
instead of points as primitive building blocks}. With locales, points are a
derived concept. In particular, the opens of a locale are not sets of points;
in fact, they need not be sets of anything in peculiar.

Section~\ref{sect:locales} explores the basics of the theory of locales, with a
focus on examples and issues for constructive mathematics. Accessible
introductions to the theory include
Refs.~\bracketedrefcite{johnstone:art,johnstone:point}. The interplay between logic and
geometry is particularly stressed in
Refs.~\bracketedrefcite{vickers:continuity,vickers:locales-toposes}. We recommend
Ref.~\bracketedrefcite{belanger-marquis:pointless} for a historical guide and
Ref.~\bracketedrefcite{picado-pultr:frames-and-locales} for a comprehensive modern
textbook on the subject.

\paragraph{Geometric theories and sheaf models.} The most sensible solution to
the task
\begin{quote}
Determine the continuous complex-valued functions~$f$ with~$f(z)^2 = z$.
\end{quote}
is not the single set
\[ \{ f : \CC \to \CC \,|\, \text{$f$ is continuous and $f(z)^2 = z$ for all~$z \in \CC$} \}. \]
Rather, since the solvability of this functional equation varies with the
domain on which~$f$ should be defined (having no global solution defined on all
of~$\CC$ but two solutions on small disks not containing the origin), we should
consider the family
\[ (\{ f : U \to \CC \,|\, \text{$f$ is continuous and $f(z)^2 = z$ for all~$z
\in \CC$} \})_{U \in \OOO(\CC)} \]
of sets as the solution (or perhaps, since no solution was actually given, as
an apt description) of the task.

If we are building a prime ideal~$\ppp$ step by step, adding elements to it
when the need arises, the most sensible answer to the question
\begin{quote}
Is~$x$ contained in the prime ideal~$\ppp$?
\end{quote}
is not a single truth value, but a family of truth values, parametrized by the
possible stages of construction of~$\ppp$.

Somewhat surprisingly, these two examples can be treated by a single theory,
the \emph{theory of sheaves}. Sheaves help to organize such families in an
efficient manner and reify them as single coherent entities; the first family
is (part of) a \emph{sheaf of sets} and the second family is (part of) a
\emph{sheaf of truth values}. Sheaves will be the subject of
Sections~\ref{sect:sheaves}--\ref{sect:sheaf-semantics}.

Sheaves interact particularly well with \emph{geometric theories}, which we
review in Section~\ref{sect:geometric-theories}. Firstly, sheaves allow us to
generalize our notion of a \emph{model} of a geometric theory. For instance, a
set-based model of the geometric theory of rings is just an ordinary ring; but
now we can also consider sheaves of rings as further models.

Some geometric theories only really unfold in the world of sheaves. For
instance, the geometric theory of surjections~$\NN \twoheadrightarrow \RR$
reviewed on page~\pageref{item:theory-of-surjections} does not have any
set-based models, but this theory is consistent and it does have nontrivial sheaf
models. For the purposes of constructive algebra, sheaf models of some geometric
theories play a vital role, as we will discuss in Section~\ref{sect:appl}.

Secondly, geometric theories allow us to efficiently construct the spaces
(locales) on which we want to study sheaves. This is because a fundamental
feature of geometric theories (in contrast to, say, arbitrary first-order
theories) is that their models naturally organize to form a space, the
space (locale) of models reviewed in Section~\ref{sect:presenting-frames}.


\paragraph{Applications in constructive algebra.}
\label{par:appl-constr-alg}
Commutative algebra abounds with techniques to reduce given situations to
simpler ones, for instance passing to a quotient or passing to a
localization. These techniques facilitate short and elegant proofs, such as the
following. (Section~\ref{sect:algebraic-preliminaries} contains some algebraic
preliminaries.)

\begin{quote}
\textbf{Theorem.} Let~$M$ be an injective matrix with more columns than rows
over a reduced ring~$A$. Then~$1 = 0$ in~$A$.

\textbf{Proof.} \uwave{Assume not.} Then there is a \uwave{minimal
prime ideal} $\ppp \subseteq A$. Since localization is exact, the matrix~$M$ is also injective when
considered as a matrix over the stalk~$A_\ppp$. Since~$A_\ppp$ is a
\uwave{field}, this is a contradiction to basic linear algebra. \qed
\end{quote}

However, these reduction techniques typically require \emph{transfinite
methods}. The displayed proof appeals to the transfinite four times:
\begin{enumerate}
\item The proof proceeds by contradiction, hence requires the law of excluded
middle.  
\item The proof then requires a minimal prime ideal. Such an ideal can be
obtained in two steps: First, the Boolean Prime Ideal
Theorem~\BPIT is invoked to obtain some prime ideal. Then
Zorn's lemma fabricates a minimal one.\footnote{Standard textbooks prove the statement ``any nontrivial
ring has a prime ideal'' by using Zorn's lemma, which
in the presence of the law of excluded middle is equivalent to the full axiom
of choice; however, the statement is actually equivalent to the weaker Boolean
Prime Ideal Theorem~\cite{scott:bpit,banaschewski-harting:lattice-aspects}.
The combined statement ``any nontrivial ring has a minimal prime ideal'' is equivalent to
the full axiom of choice~\cite{savin:minimal-prime-ideals}.}
\item Finally, the proof exploits that the stalks of reduced rings at minimal
prime ideals are fields. This requires two further invocations
of~\BPIT.\footnote{By~\BPIT, the intersection of all prime ideals of~$A_\ppp$
is its nilradical. Since~$A_\ppp$ is reduced, the nilradical is the zero ideal.
Since the prime ideals of~$A_\ppp$ are in bijection with those prime ideals
of~$A$ which are contained in~$\ppp$, the ring~$A_\ppp$ has exactly one prime
ideal. By~\BPIT, any noninvertible element of~$A_\ppp$ is contained in some
prime ideal, hence in the zero ideal. Thus~$A_\ppp$ is a field.}
\end{enumerate}

This state of affairs is not satisfactory. A statement as simple as the
displayed example should admit an explicit calculational proof, presenting a
concrete method for transforming the given conditional equations expressing
injectivity into the equation~$1 = 0$ without any appeals to the transfinite.

The completeness theorem for coherent logic~\cite[Corollary~D1.5.10]{johnstone:elephant} even
gives an a priori reason why (for a given matrix size) a finitary proof has to
exist. However, since the proof of the completeness theorem itself uses~\BPIT
(and indeed is equivalent to it), it does not give any indication how such a
proof could be found. (Still, beautiful constructive proofs are known and
presented, for instance, in a celebrated short note by Fred Richman on nontrivial
uses of trivial rings~\cite{richman:trivial-rings} and in the recent
textbook~\cite{lombardi-quitte:constructive-algebra} by Henri Lombardi and
Claude Quitté on constructive commutative algebra.)

The key issue with the usual reduction techniques in commutative algebra is
their dependence on \emph{ideal objects} such as prime or maximal ideals. In
general, those objects can only be obtained by transfinite methods. However, in
practice many arguments do not actually require completed ideal objects:
Their computational core applies just as well to finite approximations of
these ideal objects.

It is then a challenging task in mathematical logic to devise efficient means to
extract this obscured constructive content in a mostly mechanical manner. We
envision to reinterpret any given classical proof employing ideal objects in a
constructive fashion, realizing ideal objects as convenient fictions.

Locales and sheaf models contribute to this program by providing new reduction
techniques for commutative algebra. These techniques have a similar effect as the
classical reduction techniques, but are powered by purely constructive
underpinnings. In a nutshell, the idea is as follows.

\begin{enumerate}
\item Instead of replacing a given ring with another (for instance a ring~$A$ with
one of its stalks~$A_\ppp$ at a minimal prime ideal), we replace a given ring
with a \emph{sheaf of rings}.
\item We then maintain the convenient illusion that we are working with a plain old ring
instead of a sheaf by employing the sheaf semantics.
\end{enumerate}

In some cases, these new techniques go even beyond their classical role models.
For instance, while a module need not be free ``on a dense open'' if all its
stalks at minimal prime ideals are, a module~$M$ is free in that sense if its
mirror image~$M^\sim$ is \notnot free. This observation is the basis for a
new proof of Grothendieck's generic freeness lemma, reviewed in
Section~\ref{sect:example-applications}, which is shorter and arguably more
perspicuous than the previously known classical proofs.

\paragraph{Where to go from here.} The themes which this contribution touches
on are embedded into a larger context, and we invite readers to follow those
interesting tangents.

Regarding spaces, the notion of locales can be modified in several ways.
Firstly, there are \emph{sites} and closely related
\emph{toposes}.~\cite{leinster:introduction,moerdijk-maclane:sheaves-logic}.
Given two opens~$U$ and~$V$ of a locale, there is only a truth value as to
whether~$U$ is contained in~$V$; with sites and toposes, there can be a
nontrivial \emph{set} of ways that~$U$ is contained in~$V$. In an orthogonal
direction, there is the notion of a \emph{formal
topology}~\cite{sambin:some-points,sambin:ifs}. These implement the idea of
pointfree topology in predicative settings. In this context also basics of
algebraic geometry have been
developed~\cite{schuster:formal-zariski,cls:spectral-schemes,cls:projective-spectrum}.

Regarding sheaf semantics, Section~\ref{sect:sheaf-semantics} only presents
the first-order case, but the theory can be extended to include unbounded
quantification~\cite{shulman:stack-semantics}, function types and power types
and even a fully-fledged type theory~\cite{maietti:modular-correspondence}. The
sheaf semantics of appropriate toposes has been used to give \emph{synthethic
accounts} of several fields, such as differential geometry~\cite{kock:sdg},
domain theory~\cite{hyland:synthetic-domain-theory} and computability
theory~\cite{bauer:synthetic-computability-theory}.

Regarding constructive algebra, there are by now several
textbooks~\cite{mines-richman-ruitenburg:constructive-algebra,lombardi-quitte:constructive-algebra,yengui:constructive-commutative-algebra}
on constructive algebra available and there is a vast literature on
the intriguing subject of constructivizing classical commutative algebra.
Starting points include
Refs.~\bracketedrefcite{coste-lombardi-roy:dynamical-method,coquand-lombardi:logical-approach,rinaldi-schuster-wessel:edde,cederquist-coquand:entrel}.

\paragraph{Notes for readers familiar with the themes of this contribution.}
We present a self-contained proof of the simple
version of Barr's theorem as Theorem~\ref{thm:baby-barr}. For reducing
intuitionistic provability to geometric provability, we employ (what amounts
to) the syntactic site, similar to the refined completeness result
by Thierry Coquand~\cite{coquand:site} but closer to the standard description
of the syntactic site. For reducing classical provability to intuitionistic
provability, we employ a variant of the double negation translation. We were
not able to track down a reference to this approach in the context of
infinitary logic, but the basic idea has long been known and is summarized, for
instance, in Ref.~\bracketedrefcite{aczel:russell-prawitz}.

In Section~\ref{sect:sheaves}, where we introduce sheaves, we should
probably also introduce sites. They would fit the narrative well and allow for
the construction of generic models of arbitrary geometric theories and not only
propositional ones. They would also allow us to cut down on the required metatheory,
sitting comfortably within predicative environments such as arithmetic
universes~\cite{maietti:au,vickers:sketches} or~\textsc{czf}.
We do not, however, mainly to not require
prerequisites in category theory.

The technique presented in Section~\ref{sect:appl} can be
regarded as a first-order (and, if desired, higher-order) coating of an
underlying geometric theory. We propose to take this coating seriously; in
particular to make good use of surprising nongeometric sequents validated by the
generic model; and, for ease of use by algebraists, to cast it in semantic terms,
even though it is all syntax behind the scenes.\footnote{By now it is an established idea in constructive algebra that we should
consider the theory of prime filters instead of the prime filters themselves,
hoping that, for instance, a given mathematical proof of the statement~``no
prime filter contains~$x$'' can also be cast in the
geometric language of the theory, which then yields the nilpotency of~$x$
constructively. If the given proof uses only geometric reasoning, this is
immediate; if it uses intuitionistic reasoning, it can be compiled down (for
instance precisely by the technique we use); and if it uses classical reasoning, a
version of Barr's theorem might be applicable. The elusive actual prime filters
of the ring are replaced with the concrete and purely syntactical generic prime
filter.

This approach misses that the generic model is interesting on its own and in
particular that it satisfies peculiar nongeometric sequents such as ``nonunits
are nilpotent'' and ``any ideal is \notnot finitely generated''. In this sense,
the generic model is better than the actual models it aims to replace, and we
believe that potential applications of this insight should be properly
explored.}

\section{Locales}
\label{sect:locales}

The formal definition of a locale will be given below as
Definition~\ref{defn:locale}, but we first review examples and comment on
the relevance of locales to constructive mathematics.

\subsection{Spaces without points}
\label{sect:examples-no-points}

A metric or topological space without any points is not very interesting: It is
empty, and up to isomorphism there is only one such space. In contrast, a
locale can be nontrivial even if it does not contain any points. This
phenomenon is an instance of a general guiding principle, namely that relinquishing
points increases flexibility.

\paragraph{The locale of surjections~$\NN \twoheadrightarrow \RR$.} As is
well-known, there are no surjections from~$\NN$ to~$\RR$.\footnote{More
precisely, it is a theorem of classical mathematics that the reals are
uncountable in the sense of admitting no surjection from the naturals. The
situation is more subtle in constructive mathematics. Firstly, in the absence
of countable choice, the reals bifurcate into several distinct flavors, hence
one needs to state which flavor of the reals one is referring to. Secondly, all
known proofs of the uncountability of the Cauchy and the Dedekind reals assume
either the law of excluded middle or the axiom of countable choice. In the
absence of either of these axioms, only the MacNeille reals are known to be
uncountable~\cite{blechschmidt-hutzler:macneille}.} Hence there is no
interesting topological space of those surjections. The set~$\{ f : \NN
\twoheadrightarrow \RR \,|\, \text{$f$ is a surjection} \}$ is empty, and there
is not much more to say about this state of affairs. However, there is a
well-defined and nontrivial locale~$X$ of those surjections.

The points of~$X$ are in canonical one-to-one correspondence with the
surjections~$\NN \twoheadrightarrow \RR$, hence~$X$ does not have any points.
But this locale does have uncountably many basic opens~$U_{nx}$, where~$n$
ranges over the naturals and~$x$ ranges over the reals. We picture~$U_{nx}$ as
the ``open of those surjections~$f$ for which~$f(n) = x$'' and we can compute with
these opens, consider functions on~$X$ and ponder sublocales of~$X$. For instance:
\begin{enumerate}
\item If~$x \neq y$, then the intersection of~$U_{nx}$ with~$U_{ny}$
is truly empty.
\item The union of the~$U_{nx}$, where~$x$ is a fixed real
number and~$n$ ranges over the naturals, is all of~$X$. No finite number
of these opens covers~$X$, hence~$X$ is not compact.
\item Assuming~$x = y \vee \neg(x = y)$ for any pair of reals, there is, for any real~$x$, a well-defined continuous function from~$X$ to
(the localic version of) the naturals which, on the level of points, would map
a surjection~$f : \NN \to \RR$ to the smallest number~$n$ such that~$f(n) = x$.
\item There is a well-defined continuous function from~$X$ to (the localic
version of) the reals which, on the level of points, would map a surjection~$f$
to the number~$\sum_{n=0}^\infty 2^{-n} \arctan(f(n))$. Unlike the previous
example, this function is not locally constant.
\item\label{item:intersection-sublocales} The locale~$X$ is a sublocale of the
locale~$Y$ of arbitrary functions~$\NN \to \RR$ (which can also be realized as
a topological space). It can be obtained as the intersection of the uncountably
many sublocales~$Y_x$, where~$Y_x$ is the sublocale of~$Y$ consisting of those
functions which hit the real~$x$.
\item A certain sublocale of~$X$, the locale~$X'$ of those functions~$\NN \to
\RR$ for which any real has infinitely many preimages, has a fractal nature: It
is covered by the (intersections with~$X'$ of the) opens~$U_{0x}$, where~$x$ ranges over the reals; the pairwise
intersection of these opens is (truly) empty; and they are each isomorphic
to~$X'$ -- on the level of points, by mapping a surjection~$f$ in~$U_{0x}$ to
the surjection~$(n \mapsto f(n+1))$.
\end{enumerate}

There is nothing special about the real numbers in this example;
in fact, the example works just as well with any set~$M$ in place of the reals.
For any set~$M$, there is a locale of surjections~$\NN \twoheadrightarrow M$,
and this locale is trivial (isomorphic to the empty locale) if and only if~$M$
is empty. Locales of these kind are used as an important reduction step in the
extension of Grothendieck's Galois theory by André Joyal and
Myles Tierney~\cite[Section~V.3]{joyal-tierney:galois-theory}.

\paragraph{Intersection of dense sublocales.} In ordinary topology, the
intersection of dense subspaces need not be dense. A simple example is the
intersection of~$\QQ$ with its complement in~$\RR$. In contrast, the
intersection of (even an arbitrary set of) dense sublocales is always again
dense -- even if the intersection might well have no points.

For instance, the locale-theoretic intersection of (the localic version
of)~$\QQ$ with its complement does not have any points and is dense in (the
localic version of) the reals. Intuitively, while these two sublocales do not
have any points in common, there still is nontrivial ``localic glue''.


Another example is given by item~(\ref{item:intersection-sublocales}) above:
Each of the sublocales~$Y_x$ is dense in~$Y$, hence their intersection~$X$ is
so as well.

\paragraph{The Banach--Tarski paradox.} The Banach--Tarski paradox is the
unintuitive statement that a three-dimensional solid ball in~$\RR^3$ of
radius~$r$ can be partitioned into five disjoint subsets in such a way that
rearranging those subsets using only Euclidean motions yields two disjoint
solid balls of radius~$r$ each. The axiom of choice is required to procure
these subsets, and the Banach--Tarski paradox is not in
contradiction with the basic properties of the Lebesgue measure in~$\RR^3$
because these intermediate subsets are not measurable.

The traditional way to avoid the Banach--Tarski paradox is to adopt the
\emph{axiom of determinacy} instead of the axiom of choice. Just as the axiom
of choice posits that a certain property of the finite domain also holds for the
infinite, the axiom of determinacy is a certain statement whose finitary
analogue is provable in unadorned Zermelo--Fraenkel set theory. It entails that
all subsets of~$\RR^n$ are measurable.

The Banach--Tarski paradox can also be avoided by adopting a localic point of
view: While the localic counterparts of the five pieces do not have any points
in common, the locale-theoretic pairwise intersections are still nontrivial~\cite{simpson:measure}.
Hence there is no paradox, as one would not expect a rearrangement of
overlapping sublocales to preserve volume.

\paragraph{Random sequences.} \emph{Cantor space} is the space~$2^\NN$ of
infinite binary sequences, equipped with the product topology. A subset~$U
\subseteq 2^\NN$ is open with respect to this topology if and only if for
every~$\alpha \in U$, there is a number~$n \in \NN$ such that
\[ \{ \beta \in 2^\NN \,|\, \forall i < n\+ \alpha(i) = \beta(i) \} \subseteq U. \]

The open subsets of~$2^\NN$ are precisely those subsets which encode
\emph{observable} (or \emph{semidecidable}) properties of binary sequences,
those properties for which a finite record of digits suffices to verify that a
sequence has it (but not necessarily to falsify it). For instance, the sets~$V
= \{ \alpha \in 2^\NN \,|\, \alpha(0) = \alpha(1) = 1 \}$ and~$W = \{
\alpha \in 2^\NN \,|\, \exists n\in\NN\+ \alpha(n) = 1 \}$ are open and the sets
\begin{align*}
  A &\defeq \{ \alpha \in 2^\NN \,|\, \forall n\in\NN\+ \alpha(n) = 0 \} \\
  B &\defeq \Bigl\{ \alpha \in 2^\NN \,\Big|\, \lim_{n \to \infty} \frac{1}{n} \sum_{i=0}^{n-1} \alpha(i) = \frac{1}{2} \Bigr\}
\end{align*}
are not. With respect to the standard Lebesgue measure on~$2^\NN$, the sets~$V$
and~$W$ have measure~$\frac{1}{4}$ and~$1$, respectively.

Unlike a regular sequence such as~$(0,0,0,\ldots)$, we intuitively expect that, if the
digits of a \emph{random sequence} are revealed step by step, we will
eventually verify that it is an element of~$W$; and moreover, that it has any
observable property~$U$ of Lebesgue measure~$1$.

Elevating this expectation to a definition, it seems prudent to declare that a
sequence is \emph{random} if and only if it is contained in the intersection of
all open subsets~$U \subseteq 2^\NN$ with Lebesgue measure~$1$. One could then
endeavor to setup a develop a theory of probability on this notion.

However, the resulting theory will be trivial for the plain reason that
according to this definition, there are no random sequences: For any
sequence~$\alpha$, the set~$\{ \beta \in 2^\NN \,|\, \exists n \in \NN\+
\alpha(n) \neq \beta(n) \}$ is open, has measure~$1$ and does not
contain~$\alpha$.

Despite this fundamental issue, one can still argue that this approach has
merit~\cite{simpson:measure}, and hence it is worthwhile to find a
mathematical foundation which is capable of formalizing it. Topological
spaces are too restrictive, but Alex Simpson observed that locales are
sufficiently flexible: There is a \emph{locale of random sequences}, defined as
the locale-theoretic intersection of all measure-$1$ opens of~$2^\NN$. This
locale is nontrivial despite having no points.

\subsection{Constructive concerns}

By relinquishing points, locales provide a more flexible notion of space.
Section~\ref{sect:examples-no-points} substantiates this observation
with several examples of nontrivial locales without any points. Constructive
mathematics gives a further, orthogonal motivation to study locales: There are
situations in which the relevant spaces do have enough points, but only if one
subscribes to the axiom of choice or similar non-constructive principles. In
these situations, the pointfree approach facilitated by locales helps to give
constructive versions of classical results. We will further discuss this theme
in Section~\ref{sect:presenting-frames}. Examples for this phenomenon include
the following.

\paragraph{Compactness of the unit interval.} The unit interval, when realized
as a topological space, can fail to be compact in constructive
mathematics. For instance, it fails in the Russian school because the Kleene
tree provides a computable open covering of~$[0,1]$ with no computable finite
subcovering (a self-contained exposition of this phenomenon is contained in
Ref.~\bracketedrefcite{bauer:kleene-tree}).

In contrast, the localic version of the unit interval is always compact. The
proof is by an explicit computation with its basic opens, exploiting the inductive
nature of derivations~\cite{cederquist-negri:heine-borel}.

Similarly, Tychonoff's theorem that the product of any number of compact
topological spaces is again compact is equivalent to the axiom of choice. Its
localic counterpart, the statement that the product of any number of compact
locales is again compact, can be verified without any nonconstructive
principles~\cite{vickers:tychonoff}.

\paragraph{Galois theory.} Let~$L|k$ be a Galois extension. The
fundamental theorem of (infinite) Galois theory states that there is a bijection between
the intermediate extensions~$L|E|k$ and the closed subgroups of the topological
Galois group~$\Gal(L|k)$. The bijection maps an intermediate extension~$L|E|k$
to the subgroup~$\Gal(L|E)$ and its inverse maps a closed subgroup~$H \subseteq
\Gal(L|k)$ to the fixed field~$L^H$.

Much of the proof of the fundamental theorem of Galois theory is constructive,
but some parts use the law of excluded middle and Zorn's lemma in order to
construct certain extensions of given field homomorphisms.\footnote{For
instance, the statement~$E \subseteq L^{\Gal(L|E)}$ is trivial. For the converse
inclusion, let~$x \in L^{\Gal(L|E)}$. Assume for the sake of contradiction
that~$x \not\in E$. Using Zorn's lemma and the law of excluded middle, we find
a homomorphism~$\sigma : L \to L$ with~$\sigma|_E = \operatorname{id}$
and~$\sigma(x) \neq x$. This is a contradiction to~$x \in L^{\Gal(L|E)}$.}
As a consequence, the fundamental theorem of Galois theory as stated is not
provable in constructive mathematics.

However, this failure is not for fundamental Galois-theoretic reasons, but because of an
unfortunate choice in the definitions. There is a notion of a localic group~\cite{wraith:localic-groups} (a
locale~$G$ equipped with continuous maps~$G \times G \xrightarrow{\circ} G$, $G
\xrightarrow{(\cdot)^{-1}} G$, $1 \xrightarrow{e} G$ satisfying the group
axioms), and the topological Galois group has a localic counterpart. The
fundamental theorem can be reformulated to refer to this localic Galois group,
and the proof of this reformulation is entirely constructive~\cite{wraith:galois-topos}.

If one is so inclined, then one can obtain the topological version of the
fundamental theorem as a corollary of the localic version; the required
non-constructive principles for this step are neatly packaged up in the study
of the relation of the localic Galois group with the topological one.

\begin{remark}Incidentally, the fundamental theorem of Galois theory also
showcases a related general phenomenon, namely that classical mathematics
allows to push back topological concerns for a bit longer,
whereas in constructive mathematics we have to embrace topology
(in a sufficiently pointfree form such as locales) from the beginning.

To be more specific, classically, the basic version of the fundamental theorem
(intermediate extensions correspond to subgroups) only holds for finite Galois
extensions. For infinite Galois extensions, we have to restrict to
\emph{closed} subgroups. Constructively, the basic version cannot even be shown
for finite extensions; we have to employ spatial language even for those, or
else settle for a discrete version of the fundamental theorem: For finite field
extensions, finite intermediate extensions correspond to finite
subgroups~\cite[Theorem~8.8]{mines-richman-ruitenburg:constructive-algebra}.
Classically, intermediate extensions of finite extensions and subgroups of
finite groups are automatically finite, but constructively this can
fail.\footnote{For instance, the field extension~$\QQ(\sqrt{2})|\QQ$ is finite.
However, if the intermediate extension~$E \defeq \{ x \in \QQ(\sqrt{2}) \,|\, x
\in \QQ \vee \varphi \}$ is finite, then~$\varphi \vee \neg\varphi$. Using the
sheaf semantics reviewed in Section~\ref{sect:sheaf-semantics}, this
intermediate extension can be turned into a topological
counterexample~\cite[page~65]{wraith:localic-groups}.}
\end{remark}

\paragraph{Gelfand duality.} The celebrated Gelfand correspondence states that
the mapping
\[ X \longmapsto \Hom(X, \CC) = \{ f : X \to \CC \,|\, \text{$f$ is continuous} \} \]
defines a contravariant equivalence between compact Hausdorff spaces and
commutative unital~$C^\star$-algebras. In the other direction,
a~$C^\star$-algebra~$A$ is mapped to its \emph{maximal spectrum}~$\Max(A)$, the
topological space of multiplicative linear functionals~$A \to \CC$.

Some parts of the proof require the axiom of choice, (only) for fabricating
appropriate elements of the maximal spectrum. But as with the fundamental
theorem of infinite Galois theory, the Gelfand correspondence has a
constructive core, and a pointfree reformulation unearths this constructive
content: There is a contravariant equivalence between compact completely
regular locales and commutative
unital~$C^\star$-algebras~\cite{banaschewski-mulvey:gelfand,coquand-spitters:gelfand,henry:gelfand}.

The constructive Gelfand correspondence has been used to construct phase spaces
for quantum-mechanical systems in the Bohr topos approach to quantum
mechanics~\cite{butterfield-hamilton-isham:bohr,heunen-landsman-spitters:bohr,henry:bohr}.
Briefly, a quantum mechanical system is described by a~$C^\star$-algebra~$A$
which is in all interesting cases noncommutative. As such the Gelfand correspondence cannot be
applied to~$A$. But there is a mirror image of~$A$ as a sheaf model on a
certain locale, the \emph{Bohr topos} of~$A$, which \emph{is} commutative.
Since the universe of sheaves supports intuitionistic reasoning, as we will
review in Section~\ref{sect:sheaf-semantics}, the constructive Gelfand
correspondence can then be applied to this mirror image.

\begin{remark}For a time, we only knew that~\textsc{zfc} proved that there is a
constructive proof of the localic version of Gelfand duality without directly
knowing such a constructive proof. This situation arose because the first proof
of the localic duality~\cite{banaschewski-mulvey:gelfand} employed \emph{Barr's
theorem}, to be reviewed on page~\pageref{par:barr}.\end{remark}

\subsection{The basics of the theory of locales}

The starting point of the theory of locales is the following
observation. The set~$\OOO(X)$ of open subsets of a topological
space~$X$ forms a partially ordered set which has
\[ \begin{array}{@{}ccc@{}}
\text{arbitrary joins (suprema)}
&\quad\text{and}\quad&
\text{finite meets (infima)}, \\
\bigvee && \wedge
\end{array} \]
and where finite meets distribute over arbitrary joins:
\[ U \wedge \bigvee_i V_i = \bigvee_i (U \wedge V_i). \]

The key idea of locales is to elevate this observation to a definition, in the
process dropping the requirement for the elements of the ordered set to be sets
of points:
\begin{definition}A \emph{frame} is a partially ordered set with (arbitrary,
set-indexed) joins and finite meets such that the distributive law
holds. A \emph{frame homomorphism~$\alpha : A \to A'$} is a monotone map~$A
\to A'$ which preserves arbitrary joins and finite meets.\end{definition}

The least element of a frame (the empty join) is denoted~``$\bot$'' and the
largest element (the empty meet)~``$\top$''. The notion of a frame is
(infinitarily) algebraic. To obtain a geometric notion, we ``reverse the
direction of the arrows'':

\begin{definition}\label{defn:locale}
A \emph{locale}~$X$ is given by a frame~$\OOO(X)$, the ``frame
of opens of~$X$''. A \emph{morphism~$f : X \to X'$ of locales} (or ``continuous
map of locales'') is a frame homomorphism~$\OOO(X') \to \OOO(X)$.\end{definition}

In place of the open sets of points, locales have arbitrary \emph{opens}, the
elements of their underlying frame. The opens of locales behave similar to the
open sets in topology in that arbitrary unions and finite intersections make
sense; but unlike before, they need not be sets of points. Occasionally we
abuse notation and denote the largest open of a locale~$X$ by~``$X$''.

Examples for locales include the following.
\begin{enumerate}
\item Any topological space~$Y$ induces a locale~$L(Y)$ by
setting~$\OOO(L(Y)) \defeq \OOO(Y)$. A continuous map~$f : Y \to Y'$ of topological
spaces induces the frame homomorphism~$\OOO(Y') \to \OOO(Y),\,U \mapsto f^{-1}[U]$
in the other direction and hence a morphism~$L(Y) \to L(Y')$ of locales in the
same direction.
\item The \emph{one-point locale}~$\pt$ is the locale induced
by the one-point topological space~$\{\star\}$. Its frame of opens is
the powerset of~$\{\star\}$, also known as the set~$\Omega$ of truth
values. Its least element is~$\bot = \emptyset$ and its largest element
is~$\top = \{\star\}$, and potentially not all elements of~$\Omega$ are equal
to one of these two.
\item The locale of surjections~$\NN \twoheadrightarrow \RR$ and the localic
version of the reals of Section~\ref{sect:examples-no-points} are best
constructed as \emph{classifying locales}, a notion to be introduced in
Section~\ref{sect:presenting-frames}.
\item The radical ideals of any ring~$A$ form a frame and hence give rise to a
locale. We will identify this locale in Theorem~\ref{thm:workhorse} to be the
\emph{spectrum} of~$A$ as studied in algebraic geometry.
\end{enumerate}

Several notions in topology only refer to open sets and not to points. Such
notions have an immediate counterpart in locale theory. For instance, a
locale~$X$ is \emph{compact} iff for any family~$(U_i)_{i \in I}$ of opens
of~$X$ such that~$\top = \bigvee_{i \in I} U_i$, there is a (Kuratowski-)finite
subset~$I' \subseteq I$ such that~$\top = \bigvee_{i \in I'} U_i$.\footnote{A
set~$M$ is \emph{Kuratowski-finite} iff there is a surjection~$[n]
\twoheadrightarrow M$, where~$[n] = \{1,2,\ldots,n\}$. In constructive
mathematics, we distinguish this notion from the stronger condition for a
set~$M$ to be \emph{Bishop-finite}, which requires that there is a
bijection~$[n] \to M$.}
With this definition, a topological space~$Y$ is compact iff its induced locale~$L(Y)$ is.

Points do not appear in the definition of a locale, but they can be defined as
a derived concept:
\begin{definition}A \emph{point} of a locale~$X$ is a locale morphism~$\pt \to
X$.\end{definition}

This definition is inspired from the situation with topological spaces, where
continuous maps from the one-point space to a topological space~$Y$ are in
canonical bijection with the points of~$Y$.

The underlying frame homomorphism~$\alpha : \OOO(X) \to \Omega$ of a point~$x$
of~$X$ can be pictured as mapping each open~$U$ of~$X$ to the truth value to which
extent~$x$ belongs to~$U$. Mnemonically, we write~$x \inplus U$ iff~$\alpha(U)
= \top$.

The point~$x$ is completely determined by the information to which opens it
belongs, that is by the set {$\{ U \in \OOO(X) \,|\, \alpha(U) = \top \}$}. This set is a \emph{completely
prime filter}, that is a subset~$\fff \subseteq \OOO(X)$ which is upward-closed,
closed under finite meets and for which~$\bigvee_i U_i \in \fff$ implies~$U_i
\in \fff$ for some index~$i$; and conversely, any such completely prime filter
gives rise to a point of~$X$ (Exercise~\ref{ex:points-as-filters}).

\begin{definition}A locale~$X$ is \emph{spatial} if and only if its points
suffice to detect the inclusion relation on the opens of~$X$, that is if for any opens~$U, V \in \OOO(X)$,
if~$x \inplus U \Rightarrow x \inplus V$ for all points~$x$ of~$X$, then~$U
\preceq V$.\end{definition}

An equivalent definition is: A locale~$X$ is \emph{spatial} if and only if its points
suffice to distinguish its opens, that is if for any opens~$U, V \in \OOO(X)$, if~$x \inplus U
\Leftrightarrow x \inplus V$ for all points~$x$ of~$X$, then~$U = V$.

For instance, any locale induced by a topological space is spatial. The locale
of surjections~$\NN \twoheadrightarrow \RR$ is a striking example of a locale which fails to be
spatial. The localic real line is spatial in classical mathematics and can
fail to be spatial in constructive mathematics
(Exercise~\ref{ex:spatiality-reals}). In this case, the locale induced by the
topological space of reals should not be confused with the true localic real
line. These locales have the same points, but the former might not be locally compact
while the latter always is.

\begin{remark}\label{rem:definition-sober}
The set of points of a locale~$X$ can be made into a topological space, giving
rise to a functor~$\Pt : \Loc \to \Top$. This functor is right adjoint
to the functor~$L : \Top \to \Loc$. A locale~$X$ is spatial iff the canonical
morphism~$L(\Pt(X)) \to X$ is an isomorphism, and a topological space~$Y$
is \emph{sober} iff the canonical morphism~$Y \to \Pt(L(Y))$ is a
homeomorphism. The space~$\Pt(L(Y))$ is the \emph{sobrification} of~$Y$; for
instance, the sobrification of any inhabited indiscrete space is the one-point
space.

The adjunction~$L \dashv \Pt$ restricts to an equivalence between sober
topological spaces and spatial locales. Assuming the law of excluded middle,
both Hausdorff spaces and the schemes from algebraic geometry are sober
(Exercise~\ref{ex:hausdorff-sober}), hence the spaces of a wide
range of the mathematical landscape can be faithfully studied as locales.\end{remark}

\begin{remark}\label{rem:sober-embed}
Since continuous maps between spaces induce locale morphisms, we
have a canonical map~$\Hom_\Top(Y,Y') \to \Hom_\Loc(L(Y),L(Y'))$. If~$Y'$ is
sober, this map is a bijection. In particular, in this case the set-theoretic points
of~$Y'$ and the locale-theoretic points of~$L(Y')$ are in canonical
one-to-one correspondence.
\end{remark}

\subsection*{Exercises}
\addcontentsline{toc}{subsection}{Exercises}

\begin{exercise}[The space of points of a locale]%
\label{ex:canonical-topology}%
Devise a canonical to\-po\-logy on the set of points of a locale.
\end{exercise}

\begin{exercise}[Sober topological spaces as locales]%
Verify the statements made in Remark~\ref{rem:sober-embed}.
\end{exercise}

\begin{exercise}[Hausdorff spaces are sober]%
\label{ex:hausdorff-sober}%
A \emph{meet-irreducible open subset} of a topological space~$Y$ is an
open subset~$W \subseteq Y$ such that~$(U \cap V \subseteq W) \Rightarrow (U
\subseteq W) \vee (V \subseteq W)$ for all open subsets~$U,V \subseteq Y$ and
such that~$Y \not\subseteq W$.
\begin{alphlist}[(c)]
\item Assuming the law of excluded middle, show that the completely prime
filters of a topological space are in canonical one-to-one correspondence with
the meet-irreducible open subsets.\smallskip

{\scriptsize\emph{Hint.} Given a completely prime filter~$\fff \subseteq
\OOO(Y)$, the set~$\bigcup\{ U \in \OOO(Y) \,|\, U \not\in \fff \}$ is a
meet-irreducible open subset.\par}
\item Assuming the law of excluded middle, show that a topological space is
sober in the sense of Remark~\ref{rem:definition-sober} if and only if for
every meet-irreducible open subset~$W$ there is a unique point~$y \in Y$ such
that~$W = Y \setminus \overline{\{y\}}$.
(This is a classical definition of sobriety.)

\item Assuming the law of excluded middle, show that Hausdorff topological
spaces validate the condition stated in part~(b) and hence are sober.
\end{alphlist}
{\scriptsize\emph{Note.} Without the law of excluded middle, Hausdorff spaces
and even metric spaces may fail to be sober. In fact, already the topological
space of rational numbers can fail to be
sober~\cite[Example~8.14(iv)]{fourman-scott:sheaves-and-logic}.\par}
\end{exercise}

\begin{exercise}[Complete metric spaces are sober]%
A \emph{metric space} is a set~$Y$ together with a relation~$d \subseteq Y
\times Y \times \QQ_{\geq0}$, written~``$d(x,y) \leq q$'', such that the
following axioms are satisfied.
\begin{align*}
  d(x,y) \leq q \wedge q \leq r &\Longrightarrow d(x,y) \leq r \\
  (\forall r \in \QQ_{> q}\+ d(x,y) \leq r) &\Longrightarrow d(x,y) \leq q \\
  &\mathrel{\phantom{\Longrightarrow}} \exists q \in \QQ_{\geq0}\+ d(x,y) \leq q \\
  &\mathrel{\phantom{\Longrightarrow}} d(x,x) \leq 0 \\
  d(x,y) \leq 0 &\Longrightarrow x = y \\
  d(x,y) \leq q &\Longrightarrow d(y,x) \leq q \\
  d(x,y) \leq q \wedge d(y,z) \leq r &\Longrightarrow d(x,z) \leq q+r
\end{align*}
A \emph{Cauchy process} (or \emph{multi-valued modulated Cauchy sequence}) in a metric
space~$Y$ is a map~$\alpha : \QQ_{>0} \to P(Y)$ such that all the sets~$\alpha(q)$
are inhabited and such that~$d(x,x') \leq q + q'$ for all~$q,q' \in \QQ_{>0}$
and~$x \in \alpha(q), x' \in \alpha(q')$. Such a process \emph{converges
to}~$x_0 \in Y$ iff~$d(x,x_0) \leq q$ for all~$q \in \QQ_{>0}$
and~$x \in \alpha(q)$. A metric space~$Y$ is \emph{complete} iff every Cauchy
process in~$Y$ converges.

Verify, without using the law of excluded middle, that any complete metric
space is sober. \smallskip

\noindent{\scriptsize\emph{Note.} In the presence of the countable axiom of choice, any
Cauchy process can be refined to a Cauchy sequence. In its absence, however, the
familiar completion construction using equivalence classes of Cauchy sequences
fails to yield a Cauchy-complete space. In contrast, any Cauchy sequence and
more generally any Cauchy process converges in the space of equivalence classes
of Cauchy processes.\par}
\end{exercise}

\begin{exercise}[Isomorphisms of frames]%
\label{ex:isomorphisms-of-frames}%
Show that a frame homomorphism~$\alpha : A \to A'$ is an isomorphism (that is,
admits an inverse frame homomorphism) if and only if~$\alpha$ is surjective and
reflects the ordering (that is,~$\alpha(U) \preceq \alpha(V) \Rightarrow U
\preceq V$).
\end{exercise}

\begin{exercise}[The covariant approach to locales]%
Let~$A$ be a frame.
\begin{alphlist}[(c)]
\item Show that~$A$ contains arbitrary (set-indexed) meets (though they are not
required to be preserved by frame homomorphisms), by the impredicative
construction
\[ \bigwedge_{i \in I} U_i = \bigvee\{ V \in A \,|\, \text{$V \preceq U_i$ for all~$i \in I$}
\}. \]
\item What are the arbitrary meets in the frame of open subsets of a
topological space?
\item Show that for any frame homomorphism~$\alpha : A \to A'$ there exists a
monotone map~$\beta : A' \to A$ such that~\emph{$\beta$ is right adjoint
to~$\alpha$}, that is such that
\[ \alpha(U) \preceq V \quad\text{iff}\quad U \preceq \beta(V) \]
for all~$U \in A$ and~$V \in A'$.
\end{alphlist}
{\scriptsize\emph{Note.} Based on the observation in part~(c), a
\emph{covariant approach} to locales can be developed, in which locale maps~$X
\to X'$ are defined to be certain kinds of maps~$\OOO(X) \to
\OOO(X')$~\cite{picado-pultr:covariant}.\par}
\end{exercise}

\begin{exercise}[Points as complete prime filters]%
\label{ex:points-as-filters}%
Verify that the points of a locale~$X$ are in canonical one-to-one correspondence
with the completely prime filters of~$\OOO(X)$.
\end{exercise}

\begin{exercise}[Dense opens]%
\label{ex:dense-opens}%
An open~$U$ of a locale is \emph{dense} iff for any open~$V$, $U \wedge V =
\bot$ implies~$V = \bot$.
\begin{alphlist}[(b)]
\item Let~$M$ be a set and endow it with the discrete topology. Show that an
open~$U \in \OOO(L(M))$ is dense iff for any element~$x \in M$, $\neg\neg(x \in
U)$.
\item Let~$\varphi$ be a truth value, hence an open of the one-point
locale~$\pt$. Show that~$\varphi$ is dense in~$\pt$ iff~$\neg\neg\varphi$.
\end{alphlist}
\end{exercise}

\begin{exercise}[Characterizing spatial locales]%
Let~$X$ be a locale. Show that the following conditions are equivalent.
\begin{alphlist}[(c)]
\item For any opens~$U, V \in \OOO(X)$: If~$x \inplus U \Rightarrow y \inplus V$
for all points~$x$ of~$X$, then~$U \preceq V$.
\item For any opens~$U, V \in \OOO(X)$: If~$x \inplus U \Leftrightarrow y \inplus V$
for all points~$x$ of~$X$, then~$U = V$.
\item There is a topological space~$Y$ and a surjective morphism~$L(Y) \to X$
of locales. (A morphism of locales is \emph{surjective} iff its underlying
frame homomorphism is injective.)
\item The canonical morphism~$L(\Pt(X)) \to X$ is an isomorphism.
\end{alphlist}
\end{exercise}

\begin{exercise}[The frame of radical ideals]%
\label{ex:frame-of-radical-ideals}%
Let~$A$ be a ring. We order the set~$\Idl(A)$ of ideals of~$A$ by inclusion.
(See Section~\ref{sect:algebraic-preliminaries} for some algebraic
preliminaries.)
\begin{alphlist}[(b)]
\item Verify that~$\Idl(A)$ has arbitrary meets and joins, given by
intersection and sum of ideals, but that it is in general not a
frame.\smallskip

{\scriptsize\emph{Hint.} The distributive law fails in the polynomial ring~$\QQ[X,Y]$.\par}
\item Show that the sub-poset~$\Rad(A)$ of radical ideals is a frame, with
the same finite meets but with joins given by the radical of the join computed
in~$\Idl(A)$.
\item\label{item:join-top} Let~$\varphi$ be a truth value. Let~$\{\top\,|\,\varphi\}$ be the
subsingleton subset of~$\Rad(A)$ which contains the top element~$\sqrt{(1)}$
iff~$\varphi$. Verify that
\[ \bigvee\{\top\,|\,\varphi\} = \{ f \in A \,|\, \text{$f$ is nilpotent
or~$\varphi$} \} \]
in~$\Rad(A)$.
\end{alphlist}
\end{exercise}

\begin{exercise}[Coproduct of locales]%
\begin{alphlist}[(c)]
\item Show that the category of frames has small products.
\item Deduce that the category of locales has small coproducts.\smallskip

{\scriptsize\emph{Note.} The category of locales also has small products;
Exercise~\ref{ex:products-of-locales} is devoted to this fact.\par}
\item Let~$M$ be a set and endow it with the discrete topology. Show
that~$L(M)$ is the coproduct~$\coprod_{x \in M} \pt$.
\end{alphlist}
\end{exercise}

\begin{exercise}[Local locales]%
\label{ex:local-locales}%
A locale~$X$ is \emph{local} iff for any open covering~$\top = \bigvee_i U_i$,
there is an index~$i$ such that~$\top = U_i$. (For instance, the spectrum of a
ring~$A$ is local iff~$A$ is local.)
Show that any local locale has a canonical point, its \emph{focal point}.
\end{exercise}

\begin{exercise}[Maps into discrete spaces]%
Let~$X$ be a locale. Let~$A$ be a set and endow it with the discrete topology. Show
that morphisms~$X \to L(A)$ are in canonical correspondence with
families~$(U_a)_{a \in A}$ of opens of~$X$ such that~$\top = \bigvee_{a \in A}
U_a$ and~$U_a \wedge U_b \preceq \bigvee\{ \top \,|\, a = b \}$.
\end{exercise}

\section{Geometric theories and sheaf models}

\subsection{Geometric theories}
\label{sect:geometric-theories}

\begin{definition}\label{defn:geometric-theory}
A \emph{geometric theory} consists of
\begin{enumerate}
  \item a set of sorts: $X$, $Y$, $Z$, \ldots
  \item a set of function symbols: $f : X \times Y \to Z$, \ldots
  \item a set of relation symbols: $R \hookrightarrow X \times Y \times Z$, \ldots
  \item a set of geometric sequents as axioms: $\varphi \vdash_{x:X, y:Y} \psi$, \ldots
\end{enumerate}
A \emph{geometric sequent} (in some context~$x_1\?X_1,\ldots,x_n\?X_n$) is a
formula built using only the ingredients ${=}\ {\top}\ {\wedge}\ {\bot}\
{\vee}\ {\bigvee}\ {\exists}$ and the relation symbols (but no ${\Rightarrow}\
{\forall}$). We follow the usual convention that negation is understood as
syntactic sugar for implication to~$\bot$. The symbol~``$\bigvee$'' refers to disjunction of arbitrary
set-indexed families of formulas. The theory \emph{proves} a sequent~$\sigma$ if and
only if there exists a \emph{derivation} of~$\sigma$, where the class of
derivations is inductively generated by the axioms and the rules displayed in
Table~\ref{table:geometric-logic}.\footnote{Because of the infinitary
disjunctions, some care is needed to not ``build in'' the axiom of choice in
the definition of derivations. More specifically, we allow inhabited sets of
subdeductions instead of requiring that subdeductions are given by a function.
Details on the implementation in~\textsc{czf} can be found in
Ref.~\cite[Definition~5.2]{rathjen:barr}.}
\end{definition}

We refer to Ref.~\cite[Section~D1.1]{johnstone:elephant} for a more detailed
version of Definition~\ref{defn:geometric-theory}.

{\makeatletter
\newcommand*{\shifttext}[1]{%
  \settowidth{\@tempdima}{#1}%
  \makebox[\@tempdima]{\hspace*{\@tempdima}#1}%
}
\makeatother
\begin{table}[t]
  \tbl{The rules of geometric logic}{\begin{minipage}{0.8\textwidth}
    \centering
    \medskip

    \textnormal{structural rules} \\\smallskip
    \phantom{a}\hfill
    \AxiomC{$\phantom{\seq{\vec x}}$}\UnaryInfC{$\varphi \seq{\vec x} \varphi$}\DisplayProof\hfill
    \AxiomC{$\varphi \seq{\vec x} \psi$}\AxiomC{$\psi \seq{\vec x}
    \chi$}\BinaryInfC{$\varphi \seq{\vec x} \chi$}\DisplayProof\hfill
    \AxiomC{$\varphi \seq{\vec x} \psi$}\UnaryInfC{$\varphi[\vec s/\vec x]
    \seq{\vec y} \psi[\vec s/\vec x]$}\DisplayProof
    \makebox[0pt]{\shifttext{\tiny ($\vec y$ not bound)}}  
    \phantom{a}\hfill
    \bigskip

    \textnormal{rules for conjunction} \\\smallskip
    \phantom{a}\hfill
    \AxiomC{$\phantom{\seq{\vec x}}$}\UnaryInfC{$\varphi \seq{\vec x} \top$}\DisplayProof\hfill
    \AxiomC{$\phantom{\seq{\vec x}}$}\UnaryInfC{$\varphi \wedge \psi \seq{\vec x} \varphi$}\DisplayProof\hfill
    \AxiomC{$\phantom{\seq{\vec x}}$}\UnaryInfC{$\varphi \wedge \psi \seq{\vec x} \psi$}\DisplayProof\hfill
    \AxiomC{$\varphi \seq{\vec x} \psi$}\AxiomC{$\varphi \seq{\vec x} \chi$}\BinaryInfC{$\varphi \seq{\vec x} \psi \wedge \chi$}\DisplayProof
    \phantom{a}\hfill
    \bigskip

    \textnormal{rules for finitary disjunction} \\\smallskip
    \phantom{a}\hfill
    \AxiomC{$\phantom{\seq{\vec x}}$}\UnaryInfC{$\bot \seq{\vec x} \varphi$}\DisplayProof\hfill
    \AxiomC{$\phantom{\seq{\vec x}}$}\UnaryInfC{$\varphi \seq{\vec x} \varphi \vee \psi$}\DisplayProof\hfill
    \AxiomC{$\phantom{\seq{\vec x}}$}\UnaryInfC{$\psi \seq{\vec x} \varphi \vee \psi$}\DisplayProof\hfill
    \AxiomC{$\varphi \seq{\vec x} \chi$}\AxiomC{$\psi \seq{\vec x} \chi$}\BinaryInfC{$\varphi \vee \psi \seq{\vec x} \chi$}\DisplayProof
    \phantom{a}\hfill
    \bigskip

    \textnormal{rules for infinitary disjunction} \\\smallskip
    \phantom{a}\hfill
    \AxiomC{$\phantom{\seq{\vec x}}$}\UnaryInfC{$\varphi_{i_0} \seq{\vec x} \bigvee_{i \in I} \varphi_i$}\DisplayProof\hfill
    \AxiomC{$\varphi_i \seq{\vec x} \chi \text{ for each~$i \in I$}$}\UnaryInfC{$\bigvee_{i \in I} \varphi_i \seq{\vec x} \chi$}\DisplayProof
    \phantom{a}\hfill
    \bigskip

    \textnormal{double rule for existential quantification} \\\smallskip
    \phantom{a}\hfill
    \Axiom$\varphi \ \fCenter\seq{\vec x, y} \psi$
    \doubleLine
    \UnaryInf$\exists y\?Y\+ \varphi\ \fCenter\seq{\vec x} \psi$
    \DisplayProof
    \makebox[0pt]{\shifttext{\tiny ($y$ not free in~$\psi$)}}
    \hfill\phantom{a}
    \bigskip

    \textnormal{mixed rules (unnecessary in presence of the rule for implication)} \\\smallskip
    \vspace{-0.5em}
    \phantom{a}\hfill
    \AxiomC{$\phantom{\seq{\vec x}}$}
    \UnaryInfC{$(\bigvee_i \varphi_i) \wedge \psi \seq{\vec x} \bigvee_i (\varphi_i \wedge \psi)$}
    \DisplayProof
    \hfill
    \AxiomC{$\phantom{\seq{\vec x}}$}
    \UnaryInfC{$(\exists y\?Y\+ \varphi) \wedge \psi \seq{\vec x} \exists y\?Y\+ (\varphi \wedge \psi)$}
    \DisplayProof
    \makebox[0pt]{\shifttext{\tiny ($y$ none of the~$\vec x$)}}
    \hfill\phantom{a}
    \bigskip

    \textnormal{rules for equality} \\\smallskip
    \vspace{-0.5em}
    \phantom{a}\hfill
    \AxiomC{$\phantom{\seq{\vec x}}$}
    \UnaryInfC{$\top \seq{x} x = x$}
    \DisplayProof
    \hfill
    \AxiomC{$\phantom{\seq{\vec x}}$}
    \UnaryInfC{$(\vec x = \vec y) \wedge \varphi \seq{\vec z} \varphi[\vec y/\vec x]$}
    \DisplayProof
    \hfill\phantom{a} \\[0.5em]
    (``$\vec x = \vec y\,$'' is an abbreviation for~``$x_1 = y_1 \wedge \cdots \wedge x_n =
    y_n$''.)%
  \end{minipage}}
  \label{table:geometric-logic}
\end{table}}

{\makeatletter
\newcommand*{\shifttext}[1]{%
  \settowidth{\@tempdima}{#1}%
  \makebox[\@tempdima]{\hspace*{\@tempdima}#1}%
}
\makeatother
\begin{table}[t]
  \tbl{The rules of infinitary intuitionistic first-order logic}{\begin{minipage}{0.8\textwidth}
    \centering
    \medskip

    (structural rules and rules for~${=}\ {\top}\ {\wedge}\ {\bot}\ {\vee}\ {\bigvee}\
    {\exists}$ as in Table~\ref{table:geometric-logic})
    \medskip

    \textnormal{rules for infinitary conjunction} \\\smallskip
    \phantom{a}\hfill
    \AxiomC{$\phantom{\seq{\vec x}}$}\UnaryInfC{$\bigwedge_{i \in I} \varphi_i \seq{\vec x} \varphi_{i_0}$}\DisplayProof\hfill
    \AxiomC{$\psi \seq{\vec x} \varphi_i \text{ for each~$i \in I$}$}\UnaryInfC{$\psi \seq{\vec x} \bigwedge_{i \in I} \varphi_i$}\DisplayProof
    \phantom{a}\hfill
    \bigskip

    \textnormal{double rule for implication} \\\smallskip
    \phantom{a}\hfill
    \Axiom$\varphi \wedge \psi\ \fCenter\seq{\vec x} \chi$
    \doubleLine
    \UnaryInf$\varphi\ \fCenter\seq{\vec x} \psi \Rightarrow \chi$
    \DisplayProof
    \phantom{a}\hfill
    \bigskip

    \textnormal{double rule for universal quantification} \\\smallskip
    \phantom{a}\hfill
    \Axiom$\varphi \ \fCenter\seq{\vec x, y} \psi$
    \doubleLine
    \UnaryInf$\varphi\ \fCenter\seq{\vec x} \forall y\?Y\+ \psi$
    \DisplayProof
    \makebox[0pt]{\shifttext{\tiny ($y$ not free in~$\varphi$)}}
    \hfill\phantom{a}
    \bigskip

  \end{minipage}}
  \label{table:first-order-logic}
\end{table}}

Even though superficially similar, geometric theories play a
substantially different role than formal systems such as Peano arithmetic or
Zermelo--Fraenkel set theory. There are a number of notable differences:
\begin{enumerate}
\item Formal systems are typically of foundational interest and can be
fruitfully employed as metatheories. Geometric theories are more interesting
for specific applications, for instance for presenting spaces. This aspect is discussed
in Section~\ref{sect:presenting-frames}.
\item Formal systems typically come with an intended model; geometric theories
do not. For instance, the geometric theory of groups has all groups as models;
when we are writing down the axioms of a group, we are not setting out to
capture the properties of any one specific group.
\item Formal systems are typically required to be recursively
axiomatizable. Many interesting geometric theories are not, and indeed contain
an uncountable number of sorts, function symbols, relation symbols or axioms. Geometric theories also allow for arbitrary
set-indexed disjunctions; for this to make sense, the study of geometric
theories can only be carried out in a sufficiently rich background theory.
\item For formal systems, typically a version of Gödel's completeness theorem
holds: A formula is derivable iff it holds in all models. In contrast,
geometric theories can be consistent yet do not admit any set-based models. An
example is the geometric theory of surjections~$\NN \twoheadrightarrow \RR$
reviewed below.
\item Perhaps the most important difference to formal systems is that geometric
theories, unlike formal systems, often depend on given mathematical objects. For instance, for each
ring~$A$, there is the geometric theory of prime filters of~$A$. This theory
will play a crucial role in Section~\ref{sect:appl}.
\end{enumerate}

\begin{definition}A \emph{set-based model}~$M$ (or ``Tarski model'') of a
geometric theory~$\TT$ consists of
\begin{enumerate}
  \item a set~$\brak{X}$ for each sort~$X$,
  \item a function~$\brak{f} : \brak{X_1} \times \cdots \times \brak{X_n} \to
  \brak{Y}$
  for each function symbol~$f : X_1 \times \cdots \times X_n \to Y$ and
  \item a relation~$\brak{R} \subseteq \brak{X_1} \times \cdots \times \brak{X_n}$
  for each relation symbol~$R \hookrightarrow X_1 \times \cdots \times X_n$
\end{enumerate}
such that~$M$ validates the axioms of~$\TT$.
\end{definition}

We explicitly allow empty carriers, that is we do not demand that the
sets~$\brak{X}$ are inhabited.
Examples for geometric theories include the following.
\begin{enumerate}
\item \emph{The geometric theory of rings.} This theory has one sort, $R$; five function symbols:
$0$ and $1$ (nullary), $-$ (unary), $+$ and $\cdot$ (binary); no relation
symbols; and the usual axioms, such as $\top \vdash_{x:R,y:R} x + y = y + x$.

A set-based model of the theory of rings is an ordinary ring.

By adding the axiom
\begin{align*}
  \top &\seq{x:R} (x = 0) \vee (\exists y\?R\+ xy = 1),
\intertext{we describe \emph{geometric fields}, and by additionally adding for
each natural number~$n \geq 1$ the axiom}
  \top &\seq{a_0:R,\ldots,a_{n-1}:R} \exists x\?R\+ x^n + a_{n-1}x^{n-1} +
  \cdots + a_1x + a_0 = 0,
\end{align*}
we describe \emph{algebraically closed geometric fields}. (This field
condition is called ``geometric'' because it can be put as a geometric sequent.
Other, constructively not equivalent conditions such as ``any nonzero element
is invertible'' or ``any nonunit is zero'' are not geometric sequents.)

In a similar vein, there is the geometric theory of monoids, of groups, and so
on. Some of these theories use the infinite disjunctions supported by geometric
logic: For instance, the theory of fields of positive characteristic or the
theory of torsion groups have among their axioms
\[ \top \vdash \bigvee_{n \in \NN_{\geq 1}} (\underbrace{1 + \cdots + 1}_{\text{$n$ summands}} = 0)
  \quad\text{or}\quad
  \top \seq{g:G} \bigvee_{n \in \NN} (g^n = e). \]

\item \emph{The geometric theory of objects.}\label{item:theory-of-objects}
This theory has one sort, $X$, and no function
symbols, relation symbols or axioms. A set-based model of this theory is just a
set.

A variant of this theory is the theory of inhabited objects, which has the
axiom~$\top \vdash \exists x\?X\+ \top$. Its set-based models are precisely the
inhabited sets.

\item\label{item:theory-of-surjections} \emph{The geometric theory of surjections~$f : \NN \twoheadrightarrow
\RR$.} This theory has no
sorts; no function symbols; a nullary relation symbol~$\varphi_{nx}$ for each
pair~$\langle n,x \rangle \in \NN \times \RR$ (to be read as~``$f$ maps~$n$
to~$x$''); and the following axioms:
\begin{enumerate}
\item For each~$n \in \NN$, the axiom~$\top \vdash \bigvee_{x \in \RR}
\varphi_{nx}$, expressing that~$f$ is total.
\item For each~$n \in \NN$ and each~$x,y \in \RR$, the axiom~$\varphi_{nx}
\wedge \varphi_{ny} \vdash \bigvee\{ \top \,|\, x = y \}$,
expressing that~$f$ is single-valued.\footnote{The disjunction is taken over
the set~$\{ \top \,|\, x = y \}$. This is a certain subsingleton set of
formulas; it is inhabited (by the formula~$\top$) iff~$x = y$. Hence, if~$x =
y$, this axiom reads~$\varphi_{nx}
\wedge \varphi_{ny} \vdash \top$ and could also be omitted; if~$x \neq y$,
this axiom reads~$\varphi_{nx} \wedge \varphi_{ny} \vdash \bot$. If we work in
a constructive metatheory, we cannot perform this case distinction.}
\item For each~$x \in \RR$, the axiom~$\top \vdash \bigvee_{n \in \NN}
\varphi_{nx}$, expressing that~$f$ is surjective.
\end{enumerate}

A model of this geometric theory consists of an~$(\NN \times \RR)$-indexed
family of truth values, that is a subset~$G$ of~$\NN \times \RR$, satisfying
precisely those axioms which render~$G$ the graph of a surjection~$\NN
\twoheadrightarrow \RR$.

\item \emph{The geometric theory of Dedekind cuts.}\label{par:dedekind-cuts}
This theory has no sorts; no
function symbols; nullary relation symbols~$\alpha_x$ and~$\beta_x$ for each
rational number~$x$ (to be read as~``$x$ is contained in the lower (upper)
cut''); and the following axioms, spelling out that the cut is bounded,
rounded, open and located:
\begin{align*}
  \top &\vdash \textstyle\bigvee_{x \in \QQ} \alpha_x \\
  \top &\vdash \textstyle\bigvee_{x \in \QQ} \beta_x \\
  \alpha_y &\vdash \alpha_x & \text{(for each~$x,y \in \QQ$ such that~$x < y$)} \\
  \beta_x &\vdash \beta_y & \text{(for each~$x,y \in \QQ$ such that~$x < y$)} \\
  \alpha_x &\vdash \textstyle\bigvee_{y > x} \alpha_y & \text{(for each~$x \in \QQ$)} \\
  \beta_y &\vdash \textstyle\bigvee_{x < y} \beta_x & \text{(for each~$y \in \QQ$)} \\
  \alpha_x \wedge \beta_y &\vdash \bot & \text{(for each~$x,y \in \QQ$ such that~$x \geq y$)} \\
  \top &\vdash \alpha_x \vee \beta_y & \text{(for each~$x,y \in \QQ$ such that~$x < y$)}
\end{align*}

A model of this geometric theory consists of two families of truth values, both
indexed by the rational numbers, hence two subsets~$L,U \subseteq \QQ$, in such
a way that~$\langle L,U \rangle$ is a Dedekind cut.

\item \emph{The geometric theory of prime ideals of a given ring~$A$.} In classical
commutative algebra, the notion of a \emph{prime ideal} of a ring~$A$ is
fundamental. Corresponding to the definition (recalled in
Section~\ref{sect:algebraic-preliminaries}), the geometric theory of prime
ideals has no sorts; no function
symbols; one relation symbol~``$V(x)$'' for each element~$x \in A$; and the
following axioms:
\begin{align*}
  \top &\vdash V(0) \\
  V(x) \wedge V(y) &\vdash V(x+y) & \text{(for each~$x,y \in A$)} \\
  V(x) &\vdash V(xy) & \text{(for each~$x,y \in A$)} \\
  V(1) &\vdash \bot \\
  V(xy) &\vdash V(x) \vee V(y) & \text{(for each~$x,y \in A$)}
\end{align*}

For many purposes, the theory of prime \emph{filters} is actually more
relevant.\footnote{This is not to say that all the issues with prime ideals in
constructive mathematics evaporate when turning to prime filters. For a time it
was believed that algebraic geometry could be developed constructively if only
the spectrum of a ring was defined as the topological space of its prime
filters~\cite[Section~3]{lawvere:icm-address}, but this hope turned out to be
too naive. André Joyal gave an explicit example of a nontrivial sheaf model of
the theory of rings without any prime filters~\cite[pp.~200f.]{tierney:spectrum}.}
Classically, a prime filter of a ring is simply the complement of a
prime ideal; constructively, it is prudent to turn prime ideals upside down and axiomatize this notion directly.
Hence the theory of prime filters of~$A$ has no sorts, no function symbols, one
relation symbol~``$D(x)$'' for each element~$x \in A$ and the following
axioms:
\begin{align*}
  D(0) &\vdash \bot \\
  D(x+y) &\vdash D(x) \vee D(y) & \text{(for each~$x,y \in A$)} \\
  D(xy) &\vdash D(x) & \text{(for each~$x,y \in A$)} \\
  \top &\vdash D(1) \\
  D(x) \wedge D(y) &\vdash D(xy) & \text{(for each~$x,y \in A$)}
\end{align*}

\item \emph{The geometric theory of automorphisms of a field.}\label{par:theory-automorphisms}
Let~$L|k$ be a field
extension. The theory of automorphisms~$L \to L$ which fix~$k$ has
one nullary relation symbol~``$\sigma_{xy}$'' for each pair~$\langle x,y
\rangle \in L$ and the following axioms:
\begin{align*}
  \top &\vdash \bigvee_y \sigma_{xy} & \text{(for each~$x \in L$)} \\
  \top &\vdash \bigvee_x \sigma_{xy} & \text{(for each~$y \in L$)} \\
  \sigma_{xy} \wedge \sigma_{xy'} &\vdash \bigvee\{ \top \,|\, y = y' \} & \text{(for each~$x,y,y' \in L$)} \\
  \sigma_{xy} \wedge \sigma_{x'y} &\vdash \bigvee\{ \top \,|\, x = x' \} & \text{(for each~$x,x',y \in L$)} \\
  \top &\vdash \sigma_{00} \\
  \top &\vdash \sigma_{11} \\
  \sigma_{xy} \wedge \sigma_{x'y'} &\vdash \sigma_{x+x',y+y'} & \text{(for each~$x,x',y,y' \in L$)} \\
  \sigma_{xy} \wedge \sigma_{x'y'} &\vdash \sigma_{xx',yy'} & \text{(for each~$x,x',y,y' \in L$)} \\
  \top &\vdash \sigma_{xx} & \text{(for each~$x \in k$)}
\end{align*}
A model of this theory (can be reorganized to be) an element of the classical
Galois group of~$L|k$.

\item \emph{The inconsistent geometric theory.} This theory has no sorts,
function symbols or relations, but the single axiom $\top \vdash \bot$.
This theory does not admit any set-based models.

\item \emph{The empty geometric theory.} This theory does not have any sorts,
function symbols, relations or axioms. There is exactly one set-based model of
this theory, the empty structure.
\end{enumerate}

\paragraph{Extracting constructive content from classical proofs.}\label{par:barr}
Geometric logic strikes a fine balance between being sufficiently expressive in order to
have rich applications and also being sufficiently restrictive so that \emph{Barr's
theorem} holds. This theorem is concerned with eliminating classical reasoning
from proofs of geometric sequents, and it comes in a simple yet still useful
and a more sophisticated version:
\begin{enumerate}
\item If infinitary classical logic (classical logic extended by set-indexed
conjunctions and disjunctions) proves a geometric sequent, then so does
geometric logic.
\item If a specific version of extensional type theory with the law of excluded
middle and the axiom of choice proves a geometric sequent, then so does
geometric logic.
\end{enumerate}
Barr's theorem has ramifications for constructive algebra, as pointed out by
Gavin Wraith~\cite{wraith:intuitionistic-algebra}.
For instance, the \emph{Nullstellensatz} is the statement that for any
system~$f_1,\ldots,f_n \in K[X_1,\ldots,X_m]$ of polynomials over an
algebraically closed field~$K$, either the polynomials~$f_i$ have a common zero
or there are polynomials~$g_i$ such that~$1 = \sum_i f_i g_i$ (such an
expression is an \emph{algebraic witness} of the nonexistence of a common
zero).

The Nullstellensatz is a theorem of classical mathematics whose usual proof employs the
axiom of choice. Since it can be formulated as a geometric sequent over the
geometric theory of algebraically closed geometric fields, Barr's theorem
implies that there is also a constructive proof, in fact even a purely
geometric proof.

The simpler version of Barr's theorem is constructive, and although we will not
use it in the remainder of this text, we use this opportunity to present a
self-contained proof. The proof describes an explicit method for eliminating
the law of excluded middle from proofs of geometric sequents.

In contrast, the full version is not constructive.\footnote{%
The textbook argument~\cite[Theorem~7.57]{johnstone:topos-theory} establishing
the full version of Barr's theorem requires the axiom of choice. But recently
Michael Rathjen was able to lower the required metatheory
to~\textsc{zf}~\cite[Remark~4.2]{rathjen:barr}, by working internally to a
relative version of Gödel's constructible universe~$L$ over a forcing
extension, where the Barwise completeness theorem applies. To our knowledge,
the precise strength of Barr's theorem relative to~\textsc{czf} or~\textsc{izf}
has not yet been precisely calibrated. In particular, it is an open
question~\cite{henry:question-barr} whether Barr's theorem implies the law of
excluded middle, though it would be a major surprise if it did not.}
This puts us in the curious situation that we are promised, unconstructively,
that there exist certain constructive proofs, without being given any indication how
these proofs could be found or what an upper bound on their length is.

However, knowing that these proofs exist platonically is still worthwhile, as
we can be assured that looking for these proofs is not bound to fail. Indeed,
for the example of the Nullstellensatz, a constructive proof is given
in Ref.~\cite[Theorem~9.7]{lombardi-quitte:constructive-algebra}.

Also, if one is interested in constructive proofs not because of their
algorithmic content or of the mathematical insights they communicate but more
pragmatically for their applicability in, for instance, sheaf models, not
having them explicitly available is entirely bearable.

For the precise statement of the simple version of Barr's theorem, we recall
that the rules of geometric logic have been defined in
Table~\ref{table:geometric-logic}; that the rules of infinitary intuitionistic
first-order logic have been defined in Table~\ref{table:first-order-logic}; and
that infinitary classical logic is infinitary intuitionistic logic extended by
the law of excluded middle.
\begin{theorem}[Barr's theorem, simple version]
\label{thm:baby-barr}
Let~$\TT$ be a geometric theory.
Let~$\sigma$ be a geometric sequent over the signature of~$\TT$. Then the
following statements are equivalent.
\begin{enumerate}
\item There is a~$\TT$-derivation of~$\sigma$ in geometric logic.
\item There is a~$\TT$-derivation of~$\sigma$ in infinitary intuitionistic
logic.
\item There is a~$\TT$-derivation of~$\sigma$ in infinitary classical logic.
\end{enumerate}
\end{theorem}

\begin{proof}It is immediate that~(1) implies~(3).

To show that~(3) implies~(2), we combine the double negation translation
with Friedman's ``nontrivial exit continuation'' trick. We give a sketch of the
main thrust of the argument,
omitting lengthy verifications. Let~$\sigma \equiv (\alpha \seq{\vec a} \beta)$.
For notational simplicity, we assume that the context~$\vec a$ is empty, though
all of the following can be adapted to the nonempty case if care is taken to
avoid variable capture.

We introduce the operator~$\nabla$ with~$\nabla\varphi \defeqv
((\varphi \Rightarrow \beta) \Rightarrow \beta)$. This operator is a
\emph{local operator}~\cite[Section~14.5]{goldblatt:topoi} in the sense that
for any formula~$\varphi$ in some context~$\vec x$, infinitary intuitionistic
logic proves the sequents
\[
\top \seq{\vec x} (\varphi \Rightarrow \nabla\varphi),\
\top \seq{\vec x} (\nabla\nabla\varphi \Rightarrow \nabla\varphi),\
\top \seq{\vec x} ((\nabla\varphi \wedge \nabla\psi) \Leftrightarrow \nabla(\varphi \wedge \psi)).
\]
We then set up the~\emph{$\nabla$-translation}~$\varphi \mapsto \varphi^\nabla$
for infinitary first-order formulas over the signature of~$\TT$ as described in
Table~\ref{table:nabla-translation}.

\begin{table}[ht]
  \tbl{The definition of the~$\nabla$-translation}{\begin{minipage}{\textwidth}\begin{align*}
    (x = y)^\nabla &\defeqv \nabla(x = y) \\
    R(x_1,\ldots,x_n)^\nabla &\defeqv \nabla(R(x_1,\ldots,x_n)) \\
    \top^\nabla &\defeqv \top \quad \\
    \bot^\nabla &\defeqv \nabla\bot \\
    (\varphi \wedge \psi)^\nabla &\defeqv \varphi^\nabla \wedge \psi^\nabla &
    \textstyle (\bigwedge_i \varphi_i)^\nabla &\defeqv \textstyle \bigwedge_i \varphi_i^\nabla \\
    (\varphi \vee \psi)^\nabla &\defeqv \nabla(\varphi^\nabla \vee \psi^\nabla) &
    \textstyle (\bigvee_i \varphi_i)^\nabla &\defeqv \textstyle \nabla(\bigvee_i \varphi_i^\nabla) \\
    (\varphi \Rightarrow \psi)^\nabla &\defeqv (\varphi^\nabla \Rightarrow \psi^\nabla) \\
    (\forall x\?X\+ \varphi)^\nabla &\defeqv (\forall x\?X\+ \varphi^\nabla) &
    (\exists x\?X\+ \varphi)^\nabla &\defeqv \nabla(\exists x\?X\+ \varphi^\nabla)
  \end{align*}\end{minipage}}
  \label{table:nabla-translation}
\end{table}

Lengthy but straightforward computations by induction on the structure of
formulas or derivations then establish the following properties.
\begin{alphlist}
\item For any infinitary first-order formula~$\varphi$ in some context~$\vec y$,
infinitary intuitionistic logic proves~$\nabla(\varphi^\nabla) \seq{\vec x} \varphi^\nabla$.
\item For any geometric formula~$\varphi$ in some context~$\vec x$, infinitary
intuitionistic logic proves~$\varphi^\nabla \mathrel{\dashv\vdash\!\!\!_{\vec x}} \nabla\varphi$.
\item For any derivation of a sequent~$\varphi \seq{\vec x} \psi$ in infinitary
classical logic, there is a derivation of~$\varphi^\nabla \seq{\vec x}
\psi^\nabla$ in infinitary intuitionistic logic. This property rests on the
fact that the~$\nabla$-translation of the law of excluded middle,
\[ \top \seq{\vec x} \nabla(\varphi^\nabla \vee (\varphi^\nabla \Rightarrow \nabla\bot)), \]
is intuitionistically provable.
\end{alphlist}
Combining these ingredients, a given derivation of~$\sigma \equiv (\alpha
\vdash \beta)$ in infinitary classical logic can be transformed to give a derivation of
\[
  \alpha \vdash
  \nabla\alpha \vdash
  \alpha^\nabla \vdash
  \beta^\nabla \vdash
  \nabla\beta \equiv
  ((\beta \Rightarrow \beta) \Rightarrow \beta) \vdash
  \beta
\]
in infinitary intuitionistic logic.

To verify that~(2) implies~(1), we define a semantics~``$\alpha \models
\varphi$'' for geometric formulas~$\alpha$ and infinitary first-order
formulas~$\varphi$ by the clauses in Table~\ref{table:classifying-topos}. For
notational simplicity, we restrict to the propositional case, that is assume that the set of
sorts of~$\TT$ is empty.
The following properties of the semantics can then be established by induction:
\begin{alphlist}[(d)]
\item If~$\TT$ proves~$\beta \vdash \alpha$ (in geometric logic), then~$\alpha \models \varphi$ implies~$\beta
\models \varphi$.
\item If~$\TT$ proves~$\alpha \dashv\vdash \bigvee_i \beta_i$ (geometrically) and if~$\beta_i \models \varphi$ for
all indices~$i$, then~$\alpha \models \varphi$.
\item\label{item:baby-geom} If~$\varphi$ is a geometric formula, then~$\alpha
\models \varphi$ iff~$\TT$ proves~$\alpha \vdash \varphi$ (geometrically).
\item\label{item:baby-sound} If a first-order sequent~$(\varphi \vdash \psi)$
is derivable from the axioms of~$\TT$ in infinitary intuitionistic logic,
then~$\top \models (\varphi \Rightarrow \psi)$.
\end{alphlist}
We can then conclude as follows. Assume that~$(\varphi \vdash \psi)$
is derivable in infinitary intuitionistic logic. By~(\ref{item:baby-sound}), we
have~$\top \models (\varphi \Rightarrow \psi)$. By the clause for implication
in Table~\ref{table:classifying-topos}, we have~$\varphi \models \psi$.
By~(\ref{item:baby-geom}), there is a~$\TT$-derivation of~$\varphi \vdash
\psi$ in geometric logic.
\end{proof}

\begin{remark}The semantics set up in the proof of Theorem~\ref{thm:baby-barr}
coincides with the sheaf semantics of the \emph{classifying locale} of~$\TT$
reviewed in Section~\ref{sect:presenting-frames} and
Section~\ref{sect:sheaf-semantics}. In the nonpropositional case, the semantics
coincides with the slightly more general topos semantics of the
\emph{classifying topos} of~$\TT$~\cite[Section~2]{caramello:tst}. The local
operator defined in the proof cuts out a certain Boolean subtopos of the
classifying topos. An alternative proof proceeds by cut
elimination~\cite[Section~6]{rathjen:barr}; see also
Ref.~\bracketedrefcite{dyckhoff-negri:geometrisation}, and
Ref.~\bracketedrefcite{rathjen-toppel:reduction} for background on proof-theoretical
reduction, of which Barr's theorem is an example.
\end{remark}

\begin{table}[ht]
  \tbl{The semantics of the \emph{classifying topos of~$\TT$} (propositional case)}{\begin{tabular}{@{}l@{\ \ }c@{\ \ }l@{}}
    $\alpha \models \beta$ &iff& $\TT$ proves~$\alpha \vdash \beta$ (in geometric logic) \qquad(for~$\beta$ a relation symbol) \\
    $\alpha \models \top$ &iff& true \\
    $\alpha \models \varphi \wedge \psi$ &iff& $\alpha \models \varphi$ and $\alpha \models \psi$ \\
    $\alpha \models \bigvee_i \varphi_i$ &iff& there is a family~$(\beta_k)_{k \in K}$ of geometric
    formulas such that \\ &&\qquad $\TT$ proves~$\alpha \dashv\vdash \bigvee_k \beta_k$ and such for each index~$k \in K$,
    \\ &&\qquad\qquad there is an index~$i \in I$ such that~$\beta_k \models \varphi_i$ \\
    $\alpha \models (\varphi \Rightarrow \psi)$ &iff& for any geometric
    formula~$\beta$ such that~$\TT$ proves~$\beta \vdash \alpha$, \\
    &&\qquad if~$\beta \models \varphi$, then~$\beta \models \psi$
  \end{tabular}}
  \label{table:classifying-topos}
\end{table}

\subsection{Presenting frames by theories}
\label{sect:presenting-frames}

\begin{definition}A geometric theory~$\TT$ is \emph{propositional} if and only
if its set of sorts is empty.\end{definition}

As a consequence, a propositional geometric theory consists just of a set of
nullary relation symbols and a set of axioms. Associated to any such
theory~$\TT$ is its \emph{Lindenbaum algebra}: This is the partially ordered
set of the geometric formulas over the signature of~$\TT$ modulo~$\TT$-provable
equivalence, equipped with the ordering given by~$[\varphi] \preceq [\psi]$ iff~$\TT$
proves~$\varphi \vdash \psi$.

The Lindenbaum algebra is a frame, with finite meets given by~$[\varphi] \wedge
[\psi] = [\varphi \wedge \psi]$ and set-indexed joins given by~$\bigvee_i
[\varphi_i] = [\bigvee_i \varphi_i]$;\footnote{This short description of the
set-indexed joins picks representatives from each equivalence class and is
hence only sensible in the presence of the axiom of choice. In a
constructive metatheory, we should rather write~$\bigvee_{i \in I} M_i =
[\bigvee_{\varphi \in \bigcup_i M_i} \varphi]$.}
we can regard the Lindenbaum algebra as the free
frame generated by the nullary relation symbols of~$\TT$ modulo the axioms
of~$\TT$. Our interest in the Lindenbaum algebra is because it gives rise to a
locale:

\begin{definition}The \emph{classifying locale}~$L(\TT)$ of a propositional
geometric theory is the locale which has the Lindenbaum algebra of~$\TT$ as its
underlying frame.\end{definition}

It is an instructive exercise to verify that the points of~$L(\TT)$ are in
canonical one-to-one correspondence with the set-based models of~$\TT$. This
observation also explains why we cannot hope to construct classifying
locales of arbitrary, nonpropositional, geometric theories: Any locale only has
a set of points, but arbitrary geometric theories can have a proper class of
set-based models. (There is, however, a \emph{classifying arithmetic
space}~\cite{vickers:sketches} of any coherent theory and a \emph{classifying
topos}~\cite[Section~2]{caramello:tst} of any geometric theory.)

Many locales are fruitfully described as the classifying locale of a certain
propositional geometric theory. For instance:
\begin{enumerate}
\item The \emph{localic real line} is the classifying locale of the theory of
Dedekind cuts. Its points are the models of that theory, that is, the Dedekind
reals, and the induced topology on the set of points coincides with the usual
Euclidean topology on the reals.

The localic open unit interval is the classifying locale of the
same theory but with the additional axiom~$\top \vdash \alpha_0 \wedge \beta_1$.
\item The \emph{locale of surjections~$\NN \twoheadrightarrow \RR$} is the
classifying locale of the theory of surjections~$\NN \twoheadrightarrow \RR$.
\item The \emph{spectrum} of a ring~$A$ is the classifying locale of the
theory of prime filters of~$A$. The topological space of points of this locale
is the Zariski spectrum as familiar from algebraic
geometry.\footnote{Classically, the spectrum is usually defined as the
topological space of prime ideals instead of prime filters. However,
classically there is a canonical one-to-one correspondence between prime ideals
and prime filters, hence one could just as well use filters instead of ideals
in the definition. The classifying locale of prime ideals of~$A$ also exists,
and has the prime ideals of~$A$ as its points, but the induced topology on its
topological space of points is the \emph{flat topology} or \emph{co-Zariski
topology}~\cite{tarizadeh:flat} instead of the ordinary
Zariski topology~\cite[Proposition~4.5]{johnstone:rings-fields-and-spectra}.}
\item The \emph{localic Galois group} of a field extension~$L|k$ is the
classifying locale of the theory of ring automorphisms~$L \to L$ which fix~$k$.
\item The \emph{empty locale} is the classifying locale of the inconsistent theory.
\item The \emph{one-point locale} is the classifying locale of the empty theory.
\end{enumerate}

Just as we often refer to a topological space only by its points (``the
topological space of prime filters of~$A$''), not mentioning its topology, it is
customary to abbreviate ``the classifying locale of the propositional geometric
theory of prime filters of~$A$'' as ``the locale of prime filters of~$A$''.

\begin{remark}\label{rem:traditional-spaces}
The traditional way of constructing the real line or the Zariski
spectrum as a topological space proceeds in three steps: (1) Write down the
axioms (of Dedekind cuts or of prime filters, respectively). (2) Using the powerset
axiom and separation, construct the set of Dedekind reals or the set
of prime filters, respectively. (3) Devise a useful topology on the resulting set.

Locales provide us with a more economic and also more conceptual way of
arriving at these topological spaces: They can be obtained as the topological
space of points of the corresponding classifying locale. In particular,
step~(3) of manually devising a topology is not necessary when pursuing the
localic route.
\end{remark}

In general, different theories can give rise to isomorphic classifying locales;
such theories are called \emph{Morita-equivalent}. Olivia Caramello has built a vast research program
on this observation~\cite{caramello:tst}.

Conversely, any locale~$X$ is the classifying locale of a certain theory, namely
the \emph{theory of points of~$X$}. This theory has one nullary relation
symbol~$\varphi_U$ for each open~$U \in \OOO(X)$ (read as ``the point belongs
to~$U$'') and the following axioms:
\begin{align*}
  \varphi_U &\vdash \varphi_V & \text{(for all~$U,V \in \OOO(X)$ such that~$U \preceq V$)} \\
  \top &\vdash \varphi_\top \\
  \varphi_U \wedge \varphi_V &\vdash \varphi_{U \wedge V} & \text{(for all~$U,V \in \OOO(X)$)} \\
  \varphi_{\bigvee_i U_i} &\vdash \bigvee_i \varphi_{U_i} & \text{(for each set-indexed family~$(U_i)_i$ of opens)}
\end{align*}
These are exactly the axioms for a completely prime filter of~$\OOO(X)$. The
isomorphism from~$\OOO(X)$ to the Lindenbaum algebra of this theory maps an
open~$U$ to~$[\varphi_U]$.

The question whether a given locale is spatial is intimately related with the
existence of appropriate \emph{ideal objects} whose existence typically hinges
on the axiom of choice or one of its variants. For instance, there are the
following results (where we allow ourselves a modicum of classical logic in
order to state the results in a familiar form).
\begin{enumerate}
\item The localic spectrum of a ring~$A$ is spatial iff every nontrivial ideal
(of a certain class of rings related with~$A$) is contained in a prime ideal.
\item The localic Galois group of a Galois extension~$L|k$ is spatial iff
any~$k$-homomorphism~$\sigma : E \to L$ defined on a finite intermediate
extension~$L|E|k$ admits an extension to a~$k$-homomorphism~$L \to L$.
\item The localic real line is spatial iff every open covering of the
metric space~$[0,1]$ of Dedekind reals has a Lebesgue number (that is, if the
metric space~$[0,1]$ is Heine--Borel compact).
\item The statement that the classifying locale of any propositional coherent
theory (that is, a propositional geometric theory in which only finite
disjunctions occur) is spatial is equivalent to the Boolean Prime Ideal
Theorem~\BPIT.
\end{enumerate}

In this picture, the opens of a locale can be regarded as finite approximations
to the ideal points which potentially belong to them. For instance, the
open~$[\sigma_{\sqrt{2},-\sqrt{2}} \wedge \sigma_{\sqrt{3},\sqrt{3}}]$
of the localic Galois group of~$\overline{\QQ}|\QQ$ can be regarded as an
approximation to a certain automorphism of~$\overline{\QQ}$ about
which~$\sqrt{2} \mapsto -\sqrt{2}$ and~$\sqrt{3} \mapsto \sqrt{3}$ is known. Over time, this
approximation might be refined to a smaller open, one which contains more
information about the purported automorphism, and eventually we might actually
obtain an automorphism which is defined on all of~$\overline{\QQ}$. \emph{The opens
make good constructive sense even if the points remain elusive.}

The results mentioned above explain why we can, in classical mathematics, where all of the stated
conditions are satisfied, blithely and fruitfully use the topological
spectrum of a ring, the topological Galois group or the topological space of
reals. In classical mathematics, employing their localic counterparts is merely
for convenience or for aesthetic reasons; in constructive mathematics,
localic replacement (or similar replacements in other pointfree approaches to
topology) is vital.

\begin{remark}The case of the locale of surjections~$\NN \twoheadrightarrow
\RR$ is particular in that there, even the strong existence axioms of
classical mathematics do not suffice to materialize a point. In fact, already
the law of excluded middle or the countable axiom of choice preclude the
existence of points.
\end{remark}

\subsection{Sheaves on locales}
\label{sect:sheaves}

\begin{definition}\label{defn:presheaf}
A \emph{presheaf}~$F$ on a locale~$X$ is a functor~$\OOO(X)^\op \to \Set$, that is
\begin{enumerate}
  \item a set~$F(U)$ for each open~$U \in \OOO(X)$ and
  \item a map~$(\cdot)|^U_V : F(U) \to F(V)$ for each pair of opens~$V \preceq U$
\end{enumerate}
such that $(\cdot)|^U_U = \operatorname{id}_{F(U)}$ for all~$U \in \OOO(X)$ and
$(\cdot)|^V_W \circ (\cdot)|^U_V = (\cdot)|^U_W$ for all~$W \preceq V
\preceq U$.\end{definition}

Definition~\ref{defn:presheaf}, as well as any other definition in this section,
also makes sense for topological spaces instead of locales, since only the
notion of opens and their inclusion relation is used. The elements of~$F(U)$ are
called \emph{sections of~$F$ over~$U$} and the maps~$(\cdot)|^U_V$
are called \emph{restriction maps}. Elements of~$F(\top)$ are called
\emph{global sections}.

When first learning about presheaves, one may feel intimidated by the vast
amount of data encoded in a single presheaf (one set for each open, of which there are often uncountably many). The
situation should be compared to the perhaps more familiar Kripke models for a first-order language~$L$. The
preorder~$W$ of nodes of such a model is like the frame of opens of a locale, and the
family~$(M_w)_{w \in W}$ of~$L$-structures, one for each node, is like a
presheaf. In fact, the notion of a locale can be generalized to that of a
\emph{site} so that both locales and Kripke frames induce sites and that
presheaves on sites can be defined; however, this shall not be pursued here,
a textbook reference is Ref.~\bracketedrefcite{moerdijk-maclane:sheaves-logic}.

The prototypical example of a presheaf (and also a sheaf) is the presheaf~$\CCC$ of
continuous real-valued functions on a topological space~$X$. For this presheaf, the
set~$\CCC(U)$ is the set of continuous real-valued maps~$U \to \RR$, and the
maps~$(\cdot)|^U_V$ are given by actual restriction of functions to smaller
domains, that is by
\[ \CCC(U) \longrightarrow \CCC(V),\ s \longmapsto s|_V. \]

Any set~$M$ gives rise to a presheaf~$\underline{M}^\mathrm{pre}$, by
setting~$\underline{M}^\mathrm{pre}(U) = M$ for every open~$U$ and employing
the identity map as restriction maps. In this way, the universe of sets embeds
into the universe of presheaves.

\begin{definition}
\begin{enumerate}
\item A \emph{compatible family} of a presheaf~$F$ with respect to an open
covering~$U = \bigvee_i U_i$ is a family~$(s_i)_i$ of sections~$s_i \in F(U_i)$
such that~$s_i|^{U_i}_{U_i \wedge U_j} = s_j|^{U_j}_{U_i \wedge U_j}$ for all
indices~$i,j$.
\item A presheaf~$F$ is a \emph{sheaf}
iff for any compatible family~$(s_i)_i$ with respect to any open covering~$U =
\bigvee_i U_i$ there is a unique section~$s \in F(U)$ such
that~$s|^U_{U_i} = s_i$ for all indices~$i$ (``compatible sections glue'').
\end{enumerate}
\end{definition}

For instance, the presheaf~$\CCC$ of continuous real-valued functions is a sheaf,
while its subpresheaf~$\CCC_\text{c}$ of constant real-valued functions is
usually not a sheaf. For instance, on~$X = \RR$, the two constant functions~$f
: (-1,0) \to \RR,\,x \mapsto -1$ and~$g : (0,1) \to \RR,\,x \mapsto 1$ agree on
the intersection of their domains (which is empty), but there is no constant
function~$h : (-1,0) \cup (0,1) \to \RR$ which restricts to~$f$ on~$(-1,0)$ and
to~$g$ on~$(0,1)$.

There is a general construction called \emph{sheafification} which turns
presheaves into sheaves. Applied to the presheaf~$\CCC_\text{c}$, this process yields the
sheaf of \emph{locally} constant real-valued functions. Similarly, the
sheafification of the presheaf of bounded functions is the sheaf of locally
bounded functions, and so on; this is a general principle.

\label{par:constant-sheaf}%
The presheaf~$\underline{M}^\mathrm{pre}$ associated to a set~$M$ is almost
never a sheaf. Its sheafification~$\underline{M}$ is the \emph{constant sheaf associated
to~$M$}\label{page:constant-sheaf}. On an open~$U$, its sections are the locale
morphisms~$U \to M$, where~$M$ is regarded as the locale induced by
the discrete topological space~$M$ and~$U$ is regarded as a sublocale of~$X$.
By the constant sheaf construction, the universe of sets maps to the universe
of sheaves on~$X$. We will not require much more details on this construction
and only note two special cases:
\begin{enumerate}
\item If the locale~$X$ is induced by a topological space~$Y$, then
\[ \underline{M}(U) = \{ f : U \to M \,|\, \text{$f$ is locally constant} \}. \]
\item If~$X$ is the one-point locale, then~$\underline{M}(U) = M^U$.
\end{enumerate}

In the other direction, given a point~$x$ of a locale~$X$, any sheaf~$F$ on~$X$ gives
rise to a set, namely the \emph{stalk of~$F$ at~$x$}. The stalk is the set of
\emph{germs of sections} of~$F$ at~$x$, that is the set of all sections of~$F$
on opens~$U$ such that~$x \inplus U$, where two such sections~$\langle U,s
\rangle$ and~$\langle U',s' \rangle$ are identified iff there is an open~$V
\preceq U \wedge U'$ such that~$x \inplus V$ and~$s|^U_V = s'|^{U'}_V$.

For instance, the stalk at the origin of the sheaf of holomorphic functions
on~$\CC$ is in canonical bijection with the set~$\CC\{z\}$ of convergent power
series, and the stalks of a constant sheaf~$\underline{M}$ are all in canonical
bijection with~$M$.

\begin{definition}A \emph{subsheaf} of a sheaf~$F$ on a locale~$X$ is a
sheaf~$F'$ with~$F'(U) \subseteq F(U)$ for all opens~$U$ and the inherited
restriction maps.\end{definition}

\begin{definition}A \emph{morphism of presheaves}~$\eta : F \to G$ on a
locale~$X$ is a family~$(\eta_U)_{U \in \OOO(X)}$ of maps~$\eta_U : F(U) \to
G(U)$ such that for all sections~$s \in F(U)$ and all opens~$V \preceq U$
of~$X$, $\eta_U(s)|^U_V = \eta_V(s|^U_V)$. A \emph{morphism of sheaves} is a
morphism of the underlying presheaves.
\end{definition}

\subsection{Sheaf semantics}
\label{sect:sheaf-semantics}

\paragraph{Motivating the sheaf semantics.} Let~$F$ be a sheaf on a locale~$X$.
What is a sensible reading of the statement~``$\forall s\?F\+ (\varphi(s)
\Rightarrow \psi(s))$''? A first answer is ``for any global section~$s \in
F(X)$ such that~$\varphi(s)$,~$\psi(s)$''. However, this proposal does not do
justice to the sheaf, whose interesting behavior may well only play out with
its sections on small opens. Indeed, many important sheaves have no or few
global sections. Hence a more sophisticated answer is ``for any open~$U \in
\OOO(X)$ and any section~$s \in F(U)$ with~$\varphi(s)$,~$\psi(s)$'',
where~``$\varphi(s)$'' and~``$\psi(s)$'' have to be interpreted in a similar
vein.

Similarly, given a section~$s \in F(U)$, what should~``$\neg\varphi(s)$'' mean?
Since a section may well validate additional properties when we zoom in to
smaller opens, the sheaf-theoretically sensible answer is ``except for the
empty open~$V = \bot$,~$s|^U_V$ does not have property~$\varphi$ for any
open~$V \preceq U$'', where again~``$\varphi(s|^U_V)$'' has to be understood in a
sensible way.

As a final motivating example, consider the statement~``$\exists s\?F\+
\varphi(s)$''. This should mean that there is an open covering~$X = \bigvee_i
U_i$ such that on each open~$U_i$, there is a section~$s \in F(U_i)$ such
that~$\varphi(s)$, properly interpreted. We should not require these individual
sections to agree on the intersections~$U_i \wedge U_j$ -- by the sheaf
property, this would amount to saying that there is a global section~$s \in
F(X)$ with~$\varphi(s)$.

The \emph{sheaf semantics} formalizes the informal language agreement just
outlined. A marvelous fact with wide-ranging applications is that this
semantics is sound with respect to first-order intuitionistic logic (and even
infinitary higher-order intuitionistic logic, and even a suitable version of
extensional type theory, if the semantics is appropriately extended).

That is, if we have established a property~$\varphi$ and if there is an
intuitionistic proof that~$\varphi$ entails a further property~$\psi$, then we
may conclude that also the interpretation of~$\psi$ with the sheaf semantics
holds.

Hence the sheaf semantics allows us to reason with sheaves in much the same way
as we reason with sets, in a simple element-based language, and to import much
of the already-existing results in constructive mathematics into the world of
sheaves. For instance, any intuitionistic theorem about rings yields a
corresponding theorem about sheaves of rings.

There is only one caveat to keep in mind: Irrespective of its status in the
metatheory, the interpretation of the law of excluded middle with the sheaf
semantics is, for most locales~$X$, false (Example~\ref{ex:negneg}).
Practitioners of other flavors of constructive mathematics should also observe
that the interpretation of any form of the axiom of choice, even countable
choice or weak countable choice, is false for most locales~$X$. (About the only
exception are locales induced from discrete topological spaces. For such
locales, the law of excluded middle and countable, dependent or full choice
pass from the metatheory to the universe of sheaves.)

Incidentally, that the sheaf semantics is not sound with respect to classical
logic allows us to illustrate the failure of certain classical theorems in
constructive mathematics by geometric counterexamples and thereby verify their
unprovability in intuitionistic logic.

We now begin the formal development of the sheaf semantics. More background and
examples from theoretical computer science, analysis and differential geometry
are contained in Ref.~\bracketedrefcite{blechschmidt:custom-tailored}, which should be
regarded as a companion paper to this contribution. We also recommend
Christopher Mulvey's tutorial~\cite{mulvey:repr}, which culminates in a proof
of the Serre--Swan theorem using the sheaf semantics. Standard references for
the sheaf semantics include Refs.
\cite[Chapter~VI]{moerdijk-maclane:sheaves-logic},
\cite[Chapter~14]{goldblatt:topoi},
\cite[Section~5]{caramello:preliminaries},
\cite[Sections~12--15]{streicher:ctcl},
\cite[Chapter~6]{borceux:handbook3},
\cite[Part~D]{johnstone:elephant} and~\bracketedrefcite{shulman:categorical-logic},
which also gives a detailed account of useful type theory.

\paragraph{Definition and basic properties of the sheaf semantics.}
Let~$X$ be a locale. A \emph{formula over an open~$U \in \OOO(X)$} is an
infinitary first-order formula (made up using~${=}\ {\top}\ {\wedge}\ {\bigwedge}\ {\bot}\ {\vee}\
{\bigvee}\ {\Rightarrow}\ {\forall}\ {\exists}$) over the signature which has one sort for
each sheaf~$F$, one constant symbol of sort~$F$ for each section~$s \in F(U)$,
one function symbol~$f : F \to G$ for each morphism of sheaves, and so on. As
usual, we define~$\neg\varphi$ to be an abbreviation of~$(\varphi \Rightarrow
\bot)$.

Table~\ref{table:sheaf-semantics} defines by recursion on the structure of
formulas what it means for a formula~$\varphi$ over an open~$U$ to be
\emph{forced by~$U$} or to be \emph{true over~$U$}, written~``$U \models \varphi$''. The
main differences with the perhaps more familiar Kripke semantics are marked.
The sheaf semantics is also called \emph{Kripke--Joyal semantics}.

For the special case that~$U$ is the largest open of~$X$ we also write~``$\Sh(X)
\models \varphi$'' or~``$X \models \varphi$''.

\begin{table}[ht]
  \tbl{The (first-order fragment of) the sheaf semantics of a locale~$X$}{\begin{tabular}{@{}l@{\ \ }c@{\ \ }l@{}}
    $U \models s =_F t$ &iff& $s = t$ when evaluated as elements of~$F(U)$ \\
    $U \models R(s_1,\ldots,s_n)$ &iff& $\langle s_1,\ldots,s_n \rangle \in
    R(U)$ \\
    && \quad\quad \text{(where $R$ is the relation symbol corresponding to} \\
    && \quad\quad\quad\quad \text{a subsheaf~$R \subseteq F_1 \times \cdots \times F_n$)} \\
    $U \models \top$ &iff& true \\
    $U \models \varphi \wedge \psi$ &iff& $U \models \varphi$ and $U \models \psi$ \\
    $U \models \bigwedge_k \varphi_k$ &iff& for every index~$k$, $U \models \varphi_k$ \\
    $U \models \bot$ &iff& \hcancel{false}{0pt}{3pt}{0pt}{-2pt}\ $U = \bot$ \\
    $U \models \varphi \vee \psi$ &iff& \hcancel{$U \models \varphi$ or $U \models \psi$}{0pt}{3pt}{0pt}{-2pt}\ there exists an open covering $U = \bigvee_i U_i$ \\
    && \quad\quad such that for all indices~$i$: $U_i \models \varphi$ or $U_i \models \psi$ \\
    $U \models \bigvee_k \varphi_k$ &iff& \hcancel{for some index~$k$, $U \models \varphi_k$}{0pt}{3pt}{0pt}{-2pt}\ there exists an open covering $U = \bigvee_i U_i$ \\
    && \quad\quad such that for each index~$i$, $U_i \models \varphi_k$ for some index~$k$ \\
    $U \models \varphi \Rightarrow \psi$ &iff& for all~$V \preceq U$: $V \models \varphi$ implies $V \models \psi$ \\
    $U \models \forall s \? F\+ \varphi(s)$ &iff& for all $V \preceq U$ and sections~$s_0 \in F(V)$: $V \models \varphi(s_0)$ \\
    $U \models \exists s \? F\+ \varphi(s)$ &iff& \hcancel{there exists $s_0 \in F(U)$ such that $U \models \varphi(s_0)$}{0pt}{3pt}{0pt}{-2pt} \\
    && there exists an open covering $U = \bigvee_i U_i$ such that for all~$i$: \\
    && \quad\quad there exists~$s_0 \in F(U_i)$ such that $U_i \models \varphi(s_0)$
  \end{tabular}}
  \label{table:sheaf-semantics}
\end{table}

Some clauses in Table~\ref{table:sheaf-semantics} contain slight abuses of
notation. Namely, if~$\varphi$ is a formula over an open~$U$ and if~$V$ is an
open such that~$V \preceq U$, then~``$V \models \varphi$'' is formally not
well-defined, as~$\varphi$ is not a formula over~$V$. When we still write~``$V
\models \varphi$'', we mean that any constant symbol~$s \in F(U)$ appearing
in~$\varphi$ should be interpreted as its restriction~$s|^U_V \in F(V)$, so
that the resulting formula is a formula over~$V$.

\begin{example}\label{ex:negneg}
For instance, let~$F$ be the sheaf
\[ U \longmapsto \{ f : U \to \CC \,|\, \text{$f$ is continuous and~$f(z)^2 =
z$ for all~$z \in U$} \} \]
on~$X = \CC$. The interpretation of the statement~``$\exists s\?F\+ \top$'' is
false, as on no open which contains the origin a continuous square root
function exists), but the interpretation of~``$\neg\neg(\exists s\?F\+ \top)$''
is true: Unrolling the definition, this statement expresses that\par
\begin{quote}
the only open~$U \preceq X$ which has the property that \\
${\qquad}$ the only open~$V \preceq U$ which has the property that \\
${\qquad\qquad}$ locally on~$V$, there exist continuous square root functions \\
${\qquad}$ is~$V = \bot$ \\
is~$U = \bot$.
\end{quote}
\end{example}

The principal properties of the sheaf semantics is that it is monotone, local
and sound with respect to (infinitary) intuitionistic first-order logic as
defined in Table~\ref{table:first-order-logic} on
page~\pageref{table:first-order-logic}:

\begin{theorem}\label{thm:basic-properties-sheaf-semantics}
Let~$X$ be a locale. Let~$U$ be an open of~$X$. Let~$\varphi$ be
a formula over~$U$.
\begin{enumerate}
\item If~$V$ is an open such that~$V \preceq U$, then~$U \models
\varphi$ implies~$V \models \varphi$.
\item Let~$U = \bigvee_i U_i$ be an open covering. If~$U_i \models
\varphi$ for all~$i$, then~$U \models \varphi$.
\item Let~$\psi$ be a further formula over~$U$. If~$U \models
\varphi$ and if~$\varphi$ entails~$\psi$ intuitionistically, then also~$U
\models \psi$.
\end{enumerate}
\end{theorem}

\begin{proof}The proof of~(1) is by a routine induction on the structure
of~$\varphi$. For the clauses pertaining disjunction and existential
quantification, the distributive law has to be used.

The proof of~(2) is also a routine induction proof. For the clause pertaining
equality, the sheaf condition has to be used.

To verify statement~(3), it is prudent to generalize the claim slightly: An
induction on the structure of derivations shows that if a sequent~$\varphi
\seq{x_1{:}X_1,\ldots,x_n{:}X_n} \psi$ is derivable in infinitary intuitionistic first-order
logic, then~$U \models \forall x_1\?X_1\+ \ldots \forall x_n\?X_n\+ (\varphi
\Rightarrow \psi)$.
\end{proof}

\begin{remark}The proof of Theorem~\ref{thm:basic-properties-sheaf-semantics}
is constructive. Hence, by unrolling the definition of the sheaf semantics, any
proof which uses the sheaf semantics as a simplifying device can mechanically
be unwound to a fully explicit proof which does not reference sheaves or the
sheaf semantics.

In general, the resulting proofs will be longer and less perspicuous
than the original proofs, in particular if the original proofs employed nested
implications (such as double negations). In these cases, there is a certain
``exotic flavor'' to the unwound proofs, as with tricks with continuations in
computer science. We display some such proofs which are still readable in
Section~\ref{sect:example-applications}.
\end{remark}

\paragraph{Sheaf models.}
\begin{definition}\label{defn:sheaf-model}
A \emph{sheaf model} of a geometric theory~$\TT$ over a locale~$X$ consists of
\begin{enumerate}
  \item a sheaf~$\brak{S}$ on~$X$ for each sort~$S$,
  \item a morphism~$\brak{f} : \brak{S_1} \times \cdots \times \brak{S_n} \to
  \brak{T}$ of sheaves for each function symbol~$f : S_1 \times \cdots \times S_n \to T$ and
  \item a subsheaf~$\brak{R} \subseteq \brak{S_1} \times \cdots \times \brak{S_n}$
  for each relation symbol~$R \hookrightarrow S_1 \times \cdots \times S_n$
\end{enumerate}
such that the axioms of~$\TT$ are validated under the sheaf semantics.
\end{definition}

\begin{example}The sheaf~$\CCC$ of continuous real-valued functions on any
locale~$X$ together with appropriate morphisms is a sheaf model of the theory
of rings. For instance, the interpretation of the axiom
\[ \forall f\?\CCC\+ \forall g\?\CCC\+ f + g = g + f \]
with the sheaf semantics amounts to the condition
\begin{quote}
For any open~$U$ and any continuous function~$f \in \CCC(U)$, \\
${\qquad}$ for any open~$V$ and any continuous function~$g \in \CCC(V)$, \\
${\qquad\qquad}$ $f|^U_V + g = g + f|^U_V$.
\end{quote}
and is readily seen to be fulfilled.
\end{example}

\begin{example}Any sheaf on any locale can be
regarded as a sheaf model of the theory of objects.\end{example}

\begin{example}\label{ex:sets-as-sheaf-models}
Any set-based model~$M$ of a geometric theory~$\TT$ gives rise to a
sheaf model over the one-point locale, by employing the constant sheaf
construction reviewed on page~\pageref{page:constant-sheaf}. An infinitary
first-order formula holds for~$M$ iff it holds for the induced sheaf model.
Hence sheaf models are strictly more general than set-based models.
\end{example}

In the special case of nullary function symbols or relations symbols, by
``$\brak{S_1} \times \cdots \times \brak{S_n}$'' in
Definition~\ref{defn:sheaf-model} we mean the empty product of sheaves, the
\emph{terminal sheaf}~$1$ with~$1(U) = \{\star\}$ for every open~$U$ of~$X$.
Morphisms~$1 \to F$ are in canonical bijection with global sections of~$F$.
Subsheaves~$F \subseteq 1$ are in canonical bijection with the opens of~$X$
(Exercise~\ref{ex:subsheaves-1}).

With this in mind, we can rephrase Definition~\ref{defn:sheaf-model} in the
special case of propositional theories in the following simpler form. A
\emph{sheaf model} of a propositional geometric theory~$\TT$ over a locale~$X$
consists of a family~$(U_\varphi)_\varphi$ of opens of~$X$, one open for each
nullary relation symbol~$\varphi$ of~$\TT$, in such a way that the axioms are
satisfied. For instance, that an axiom~``$\alpha \wedge \beta \vdash \bigvee_i
\gamma_i$'' is satisfied means that~$U_\alpha \wedge U_\beta \preceq \bigvee_i
U_{\gamma_i}$.

In other words, a model of a propositional geometric theory~$\TT$ over a
locale~$X$ is the same as a frame homomorphism~$\OOO(L(\TT)) \to \OOO(X)$, hence a
locale morphism~$X \to L(\TT)$.

For further reference we state a proposition relating truth in a sheaf model
with truth at all points.

\begin{proposition}\label{prop:at-points}
Let~$\TT$ be a geometric theory without any axioms. Let~$M$ be a sheaf model
of~$\TT$ on a locale~$X$ (hence just a \emph{sheaf structure}) and consider
the structures~$M_x$ obtained by computing the stalks at points~$x \in \Pt(X)$.
Let~$\sigma$ be a geometric sequent.
Then:
\begin{enumerate}
\item If~$M$ validates~$\sigma$, then so does~$M_x$ for any point~$x \in \Pt(X)$.
\item The converse holds if~$X$ is spatial.
\end{enumerate}
\end{proposition}

\begin{proof}See~\cite[Corollary~D1.2.14(ii)]{johnstone:elephant}.\end{proof}

In particular, the stalks of sheaf models of a geometric theory~$\TT$ are
set-based models of~$\TT$.

\paragraph{The generic model.}
A geometric theory can well have no set-based models, but might still be
consistent and allow for a diverse range of sheaf models on a host of different
locales. In fact, any propositional geometric theory always has at least one,
and very peculiar, sheaf model, namely its \emph{generic model}.

\begin{definition}The \emph{generic model}~$U_\TT$ of a propositional geometric
theory~$\TT$ is the sheaf model on~$L(\TT)$ given by~$U_\varphi \defeq
[\varphi]$ for all nullary relation symbols~$\varphi$.\end{definition}

Under the view of models as locale morphisms, the generic model~$U_\TT$
corresponds to the identity morphism~$L(\TT) \to L(\TT)$. If a~$\TT$-model~$M$ on
a locale~$X$ is given by a locale morphism~$f : X \to L(\TT)$, we also say
that~``$M$ arises from~$U_\TT$ by pullback along~$f$''.

For instance, there is the \emph{generic Dedekind cut}, the generic model of
the theory~$\TT$ of Dedekind cuts on the classifying locale of~$\TT$, the
localic real line. By Example~\ref{ex:sets-as-sheaf-models}, any specific Dedekind
real~$x$ (that is, any point~$x$ of~$L(\TT)$) gives rise to a model of~$\TT$ on
the one-point locale~$\pt$, corresponding to the locale morphism~$x : \pt \to
L(\TT)$.

For a geometric sequent~$\sigma$ over the signature of a geometric
theory~$\TT$, the following statements are equivalent:
\begin{enumerate}
\item The generic model of~$\TT$ validates~$\sigma$.
\item All sheaf models of~$\TT$ on arbitrary locales validate~$\sigma$.
\item The theory~$\TT$ proves~$\sigma$.
\end{enumerate}
Hence we may understand the phrase~``Let~$M$ be a model of~$\TT$.'' to
implicitly refer to the generic model. The first-order properties of the
generic model need not be shared by all sheaf models and need not be provable
by~$\TT$.

The story of generic models is at the same time tautologous and deep, and these
short paragraphs do not do justice to this intriguing aspect of the interplay
between topology (in the form of locales) and logic (in the form of
propositional geometric theories). We strongly
recommend the accessible account by Steve Vickers~\cite{vickers:continuity} for a fuller
picture.

\begin{remark}The traditional way of constructing a map of topological spaces~$f : Y \to
Y'$ proceeds in three steps: (1) Write down a map of the underlying sets. (2)
Identify the preimages~$f^{-1}[U]$. (3) Verify that these preimages are open.

Continuing Remark~\ref{rem:traditional-spaces}, locales provide more
efficient means for constructing morphisms: To write down a morphism of
locales~$X \to X'$, where~$X'$ is given as the classifying locale of some
theory~$\TT'$, we only need to specify a~$\TT'$-model~$M'$ on~$X$; no separate
verification of continuity is required.

If~$X$ classifies a theory~$\TT$, such a specification is usually carried out
by constructing~$M'$ in terms of the generic model~$U_\TT$.
For instance, if~$f : A \to B$ is a ring homomorphism, we obtain a locale
morphism~$\Spec(B) \to \Spec(A)$ by specifying the model~$([D(f(x))])_{x \in
A}$ of the theory of prime filters of~$A$ on~$\Spec(B)$. On the level of
points, this morphism induces the expected map~$\fff \mapsto f^{-1}[\fff]$.
\end{remark}

\subsection*{Exercises}
\addcontentsline{toc}{subsection}{Exercises}

\begin{exercise}[The theory of finite objects]%
\begin{alphlist}[(c)]
\item Extend the geometric theory of objects
(page~\pageref{item:theory-of-objects}) by relation symbols and axioms, but not by
constant symbols, to ensure that the set-based models are precisely the
two-element sets.
\item Write down a geometric theory whose set-based models are precisely the
(Kuratowski-)finite sets.\smallskip

{\scriptsize\emph{Hint.} Employ countably many relation symbols~$R_n
\hookrightarrow X \times \cdots \times X$, where~$R_n(x_1,\ldots,x_n)$ should
signify that the elements~$x_1,\ldots,x_n$ exhaust all elements of~$X$. A
solution to this exercise is given in Ref.~\cite[Example~D1.1.7(k)]{johnstone:elephant}.\par}
\end{alphlist}
\end{exercise}

\begin{exercise}[Issues of size]%
Working in~\textsc{izf}, show that the Lindenbaum algebra of a propositional
geometric theory~$\TT$ can be realized as a set, even though the class of
geometric formulas is typically a proper class.\smallskip

\noindent{\scriptsize\emph{Hint.} Use that any geometric formula is provably equivalent
to one of the form~$\bigvee_i \varphi_i$, where each formula~$\varphi_i$ is a
conjunction of nullary relation symbols.\par}
\end{exercise}

\begin{exercise}[Discrete locales]%
\begin{alphlist}[(b)]
\item Find a theory which is classified by the two-point locale (the locale
associated to the discrete two-point topological space).
\item Let~$M$ be a set and endow it with the discrete topology. Find a theory
which is classified by~$L(M)$.
\end{alphlist}
\end{exercise}

\begin{exercise}[The localic unit interval]%
Extend the theory of Dede\-kind cuts in such a way that its models are those
Dedekind reals which are elements of~$[0,1]$.
\end{exercise}

\begin{exercise}[More on exotic locales]%
\label{ex:exotic-locales}%
\begin{alphlist}[(d)]
\item Define the locale of partial surjections~$\NN \twoheadrightarrow \NN^\NN$.\smallskip

{\scriptsize\emph{Note.} In Russian constructivism, as for instance reigning in
the ef{}fective topos~\cite{hyland:effective-topos,phoa:effective,bauer:c2c}, the Baire space~$\NN^\NN$ is subcountable.
However, in this setting any function~$\NN \to \NN$ is computable.
The universe of sheaves over this locale was the first example of a setting in
which the Baire space is subcountable and still any external number-theoretic
function can be imported~\cite[Section~4.3]{scedrov:forcing}.\par}
\item Define the locale of injections~$\RR \hookrightarrow \NN$.\smallskip

{\scriptsize\emph{Note.} In classical mathematics, this locale does not have
any points. However there is a variant of the ef{}fective topos based on the
infinite-time Turing machines by Joel David Hamkins and Andy
Lewis~\cite{hamkins-lewis:ittm} in which an injection~$\RR \hookrightarrow \NN$
exists~\cite{bauer:injection}. See
Ref.~\cite[Section~2.2]{blechschmidt:custom-tailored} for an exposition of the
proof.\par}

\item Define the locale of ``open covers of~$[0,1]$ by intervals with rational
endpoints such that the union of any finite number of those intervals has total
Lebesgue measure less than~$\frac{1}{2}$''.\smallskip

{\scriptsize\emph{Note.} In the universe of sheaves over this locale, the set
of unit interval Dedekind real numbers fails to be (Heine--Borel) compact and
the localic real line fails to be
spatial~\cite[Example~D4.7.13]{johnstone:elephant}.\par}
\item An open~$U$ of a locale~$X$ is \emph{complemented} iff there is an
open~$V$ such that~$U \vee V = \top$ and~$U \wedge V = \bot$. A locale is
\emph{zero-dimensional} iff any open is a join of complemented opens. Assuming
that any two reals are equal or not, show that the locale of surjections~$\NN
\twoheadrightarrow \RR$ is zero-dimensional.
\end{alphlist}
\end{exercise}

\begin{exercise}[Models as points]%
Let~$\TT$ be a propositional geometric theory. Verify that the points
of~$L(\TT)$ are in canonical bijection with the set-based models of~$\TT$.
\end{exercise}

\begin{exercise}[The product of locales]\label{ex:products-of-locales}%
Let~$X$ and~$X'$ be locales, classifying propositional geometric
theories~$\TT$ and~$\TT'$ respectively. Let~$\TT + \TT'$ be the combined theory
(whose set of sorts is the disjoint union of the sets of sorts of~$\TT$
and~$\TT'$, and so on). Show that~$L(\TT + \TT')$ is the product of~$X$
and~$X'$ in the category of locales.
\end{exercise}

\begin{exercise}[The coproduct of locales, revisited]%
Let~$X$ and~$X'$ be locales, classifying some propositional geometric
theories~$\TT$ and~$\TT'$ respectively. Find a theory in terms of~$\TT$
and~$\TT'$ which is classified by~$X \amalg X'$.
\end{exercise}

\begin{exercise}[Spatiality of the localic reals]%
\label{ex:spatiality-reals}%
\newcommand{\fa}{\ensuremath{(\!\cdot}}%
\newcommand{\fb}{\ensuremath{\cdot\!)}}%
Let~$\RR$ be the localic real line, defined as the classifying locale of the
theory~$\TT$ of Dedekind cuts as introduced on
page~\pageref{par:dedekind-cuts}. For rational numbers~$x,y$, let~``$\fa x,y
\fb$'' denote the formula~$\alpha_x \wedge \beta_y$. Let~$(x,y)$ denote the
open interval~$\{ \xi \in \Pt(\RR) \,|\, x < \xi < y \}$ in the set of Dedekind reals.
\begin{alphlist}[(f)]
\item Extend the notation to allow for~$x,y=\pm\infty$.
\item Show that the canonical topology on~$\Pt(\RR)$, defined in
Exercise~\ref{ex:canonical-topology}, coincides with the usual Euclidean
topology.
\item Show that any open of~$\RR$ can be written in the form~$\bigvee_i
[\fa x_i,y_i \fb]$ with~$x_i,y_i \in \QQ \cup \{\pm\infty\}$.
\item Let rational numbers be given such that~$(x,y) \subseteq \bigcup_i
(x_i,y_i)$ in~$\Pt(\RR)$. Assume that this open covering has a Lebesgue number.
Show that~$\TT$ proves~$\fa x,y \fb \vdash \bigvee_i \fa x_i,y_i \fb$.
\item Assume that closed intervals in the set of Dedekind reals are
(Heine--Borel) compact. Then redo part~(d), but without the assumption on the
existence of a Lebesgue number.
\item Under the same assumption as in part~(e), show that the localic real line
is spatial.
\end{alphlist}
{\scriptsize\emph{Note.} Exercise~\ref{ex:exotic-locales} lists settings in
which the assumption in part~(e) is not satisfied. The converse in part~(f)
also holds. Incidentally, the spatiality of the localic real line is an example
for a statement which requires the full version of the Heine--Borel property,
not only Heine--Borel for countable covers as it is often studied in reverse
mathematics~\cite{normann-sanders:uncountable}.\par}
\end{exercise}

\begin{exercise}[The localic Galois group of a field extension]%
Let~$L|k$ be a field extension. The \emph{localic Galois group}~$\Gal(L|k)$
is the classifying locale of the theory of automorphisms~$L \to L$ which
fix~$k$, as described on page~\pageref{par:theory-automorphisms}.
\begin{alphlist}[(g)]
\item Explain how to obtain a morphism~$\Gal(L|k) \to \Gal(L|k)$ which on the
level of points should send an automorphism~$\sigma$ to its inverse~$\sigma^{-1}$.
\item Explain how to obtain a morphism~$\Gal(L|k) \times \Gal(L|k) \to \Gal(L|k)$ which on the
level of points should send a pair~$\langle \sigma, \tau \rangle$ to~$\sigma \circ \tau$.
\item Assume from now on that~$L$ is a \emph{geometric field} in the sense that any
element is zero or invertible. Show that the axiom expressing injectivity can
be dropped without changing the set of provable sequents.
\item Let~$f \in k[X]$ be a monic polynomial such that~$f(X) =
(X-x_1)\cdots(X-x_n)$ over~$L$. Let~$x$ be one of the zeros of~$f$. Show that
the theory proves~$\top \vdash \bigvee_{i=1}^n \sigma_{x,x_i}$.
\item Now furthermore assume that~$L$ is \emph{Galois over}~$k$ in the sense that
any element of~$L$ is the zero of a monic polynomial~$f \in k[X]$ which is
separable (that is,~$(f,f') = (1) \subseteq k[X]$) and splits over~$L$ into
linear factors. Show that the axioms expressing surjectivity can be dropped.
\item Show that the topological space of points of~$\Gal(L|k)$ is canonically
homeomorphic to the topological Galois group of~$L|k$.
\item Show that~$\Gal(L|k)$ is spatial in classical mathematics.
\end{alphlist}
\end{exercise}

\begin{exercise}[Coherent locales and coherent theories]%
\label{ex:coherent}%
An open~$U$ of a locale is \emph{compact} if and only if, whenever~$U
\preceq \bigvee_{i \in I} V_i$, there is a (Kuratowski-)finite subset~$I'
\subseteq I$ such that~$U \preceq \bigvee_{i \in I'} V_i$. A locale is
\emph{coherent} iff its compact opens are closed under finite meets and if
every open is a join of compact opens. A geometric formula is \emph{coherent} iff only
finite disjunctions occur in it. A geometric theory is \emph{coherent}
iff all of its axioms consist of coherent formulas.
\begin{alphlist}[(d)]
\item Show that any coherent locale classifies some coherent theory.
\item Let~$\TT$ be a propositional coherent theory. Show that~$L(\TT)$ is
coherent.\smallskip

{\scriptsize\emph{Hint.} Proceed similarly as in the proof of
Theorem~\ref{thm:baby-barr}. For coherent formulas~$\alpha$ and geometric
formulas~$\varphi$, define a semantics by the clauses below. Show by an induction on the
structure of geometric derivations that, if~$\TT$ proves a geometric
sequent~$\varphi \vdash \psi$, then for all coherent formulas~$\alpha$ such
that~$\alpha \models \varphi$, $\alpha \models \psi$. Show for coherent
formulas~$\alpha,\beta$ that~$\alpha \models \beta$ iff~$\TT$ proves~$\alpha
\vdash \beta$. Use these ingredients to show that, if~$\alpha$ is a coherent
formula and~$(\beta_i)_{i \in I}$ is a family of coherent formulas, if~$\TT$
proves~$\alpha \vdash \bigvee_{i \in I} \beta_i$ then there is a
(Kuratowski-)finite subset~$I' \subseteq I$ such that~$\TT$ proves~$\alpha
\vdash \bigvee_{i \in I'} \beta_i$.

\begin{center}\tablefont
\begin{tabular}{@{}l@{\ \ }c@{\ \ }l@{}}
  $\alpha \models \beta$ &iff& $\TT$ proves~$\alpha \vdash \beta$ \qquad(for~$\beta$ a nullary relation symbol) \\
  $\alpha \models \top$ &iff& true \\
  $\alpha \models \varphi \wedge \psi$ &iff& $\alpha \models \varphi$ and $\alpha \models \psi$ \\
  $\alpha \models \bigvee_i \varphi_i$ &iff& there are coherent
  formulas~$\alpha_1,\ldots,\alpha_n$ such that \\ &&\qquad $\TT$ proves~$\alpha \dashv\vdash
  \alpha_1 \vee \cdots \vee \alpha_n$ and such for each~$k \in \{1,\ldots,n\}$,
  \\ &&\qquad\qquad there is an index~$i \in I$ such that~$\alpha_k \models \varphi_i$
\end{tabular}
\end{center}}

\item Show that an arbitrary (set-indexed) product of coherent locales is
coherent.

\item \emph{Deligne's completeness theorem} states that coherent locales (and
even coherent \emph{toposes}) are spatial. Show that Deligne's
result implies that propositional coherent theories
without any set-based models are inconsistent. (This is usually phrased as
``consistent theories have models''.)
\end{alphlist}
\end{exercise}

\begin{exercise}[Subsheaves of the terminal sheaf]%
\label{ex:subsheaves-1}%
Let~$X$ be a locale.
\begin{alphlist}[(b)]
\item Show that any subsheaf~$F \subseteq 1$ of the terminal sheaf on~$X$ gives rise
to an open of~$X$ by gluing all its sections.
\item Show that the construction of part~(a) yields a canonical bijection of
the subsheaves of~$1$ and the opens of~$X$.
\end{alphlist}
\end{exercise}

\begin{exercise}[A geometric interpretation of double negation]%
\label{ex:negneg-interpretation}%
Let~$\varphi$ be a formula over an open~$U$ of a locale~$X$.
\begin{alphlist}[(c)]
\item Show that there is a largest open,
denoted~``$\llbracket\varphi\rrbracket$'', such that~$U \models \varphi$.\smallskip

{\scriptsize\emph{Hint.} Consider the join of all opens~$V$ such that~$V
\models \varphi$.\par}

\item Show that~$X \models \neg\neg\varphi$ iff~$\llbracket\varphi\rrbracket$
is dense in the sense of Exercise~\ref{ex:dense-opens}.

\item Give a condition on a locale or a topological space such that its sheaf
semantics validates the law of excluded middle (or equivalently the law of
double negation elimination).
\end{alphlist}
\end{exercise}

\begin{exercise}[Unique local existence is global existence]%
\label{ex:global-existence}%
Let~$F$ be a sheaf on a locale~$X$. Let~$\varphi$ be a formula over~$X$. Assume
that~$\Sh(X) \models \exists!s\?F\+ \varphi(s)$, that is
\[ \Sh(X) \models \exists s\?F\+ \varphi(s) \quad\text{and}\quad
  \Sh(X) \models \forall s\?F\+ \forall t\?F\+ (\varphi(s) \wedge \varphi(t)
  \Rightarrow s = t). \]
Show that for any open~$U \in \OOO(X)$, there is a unique section~$s \in F(U)$
such that~$U \models \varphi(s)$.
\end{exercise}

\begin{exercise}[The sheaf of real functions as a field]%
Let~$X$ be a topological space. Let~$\CCC$ be the sheaf of continuous real-valued
functions on~$X$.
\begin{alphlist}[(d)]
\item Let~$U$ be an open of~$X$. Let~$f \in \CCC(U)$. Show using the result
of Exercise~\ref{ex:global-existence} that~$U \models (\exists g\?\CCC\+ fg =_\CCC
1)$ iff there is a function~$g \in \CCC(U)$ such that~$fg = 1$.
\item Show that~$\CCC$ is a field in that
$X \models \forall f\?\CCC\+
  \bigl(\neg(\exists g\?\CCC\+ fg = 1)\bigr) \Rightarrow f = 0$.
\item Give an example of space~$X$ such that it is not the case that~$\CCC$ is a
field in the stronger sense that
$\Sh(X) \models \forall f\?\CCC\+ (\exists g\?\CCC\+ fg = 1) \vee f = 0$.
\item Let~$X = \CC$ and let~$\OOO$ be the sheaf of holomorphic functions on~$X$.
Show that, in classical mathematics, the sheaf~$\OOO$ is \emph{discrete} in that
$\Sh(\CC) \models \forall f\?\OOO\+ \forall g\?\OOO\+ f = g \vee \neg(f = g)$.
\end{alphlist}
\end{exercise}

\begin{exercise}[The failure of the fundamental theorem of algebra]%
\label{ex:fta}%
Let~$F$ be the sheaf of continuous complex-valued functions on~$\CC$. Give an
example of a function~$f : \CC \to \CC$ such that it is not the case that
\[ \Sh(\CC) \models \exists z\?F\+ z^2 = f. \]
{\scriptsize\emph{Note.} The sheaf~$F$ can be obtained by constructing, in the
universe of sheaves on~$\CC$, the set of pairs of Dedekind reals. Hence the
example demonstrates that the usual formulation of the fundamental theorem of
algebra does not admit a constructive proof. (However, there are replacements
available~\cite{richman:fta}, and the fundamental theorem does apply to the
Cauchy complex numbers~\cite{ruitenburg:roots} and under weak choice
principles~\cite{bridges-richman-schuster:wcc}).\par}
\end{exercise}

\begin{exercise}[The least number principle]%
\label{ex:least-number-principle}%
\begin{alphlist}[(e)]
\item A subset~$M \subseteq X$ of a set~$X$ is \emph{detachable} iff for any
element~$x \in X$, either~$x \in M$ or~$x \not\in M$. Show that any inhabited
detachable subset~$M \subseteq \NN$ contains a minimal element.
\item Show that for any inhabited set~$M$ of natural numbers, it is \notnot the
case that there is a number~$n \in M$ such that~$n \leq m$ for all~$m \in M$.
\item Show that for every finitely generated vector space~$V$ over a field in the
sense that nonunits are zero, it is \notnot the case that~$V$ has a basis.
\item Verify that the assumption that any inhabited set of natural numbers
contains a minimal element implies the law of excluded middle.\smallskip

{\scriptsize\emph{Hint.} Consider a set of the form~$\{ n \in \NN \,|\, n = 1
\vee \varphi \}$.\par}
\item Give an explicit example of a subsheaf~$F$ of the constant
sheaf~$\underline{\NN}$ on a topological space~$X$ such that
\[ \Sh(X) \models \exists n\?F\+ \top
  \quad\text{but}\quad
  \Sh(X) \not\models \exists n\?F\+ \forall m\?F\+ n \leq m. \]
\end{alphlist}
\end{exercise}


\section{Applications in constructive algebra}
\label{sect:appl}

We recall from the introduction one way how locales and sheaves yield new
reduction techniques for constructive algebra:
\begin{enumerate}
\item Instead of replacing a given ring with another ring (for instance a ring~$A$ with
one of its stalks~$A_\ppp$ at a minimal prime ideal), we replace a given ring
with a \emph{sheaf of rings}.
\item We then maintain the convenient illusion that we are working with a plain old ring
instead of a sheaf by employing the sheaf semantics.
\end{enumerate}

The goal of this section is to give a leisurely introduction to this procedure.
We recall in Section~\ref{sect:algebraic-preliminaries} algebraic
preliminaries. Section~\ref{sect:a-remarkable-sheaf} explores the sheaf to use.
There is an interesting origin story to this sheaf, which we tell in
Section~\ref{sect:origins}. For computations, it is vital to have an explicit
description of this sheaf available; such a description is presented in
Section~\ref{sect:understanding}. Finally, we illustrate how to use the technique in
a series of concrete examples in Section~\ref{sect:example-applications}.

\subsection{Algebraic preliminaries}
\label{sect:algebraic-preliminaries}

By \emph{ring}, we mean commutative ring with unit. We do not require the unit
to be distinct from the zero; indeed, trivial rings have intriguing
applications~\cite{richman:trivial-rings}.

A ring is \emph{reduced} iff zero is its only nilpotent element; that is, if~$x^n = 0$
implies~$x = 0$.

An \emph{ideal} is a subset~$\aaa \subseteq A$ such that~$0 \in \aaa$, $(x \in
\ppp) \wedge (y \in \ppp) \Rightarrow (x + y \in \ppp)$ and $(x \in
\ppp) \Rightarrow (xy \in \ppp)$ (for any~$y \in A$). A family~$(x_i)_i$ of
ring elements generates an ideal, also denoted~``$(x_i)_i$'', consisting of
(Kuratowski-finite) sums of multiples of the~$x_i$.

An ideal~$\ppp$ is \emph{prime} iff~$1 \not\in \ppp$ and~$(xy
\in \ppp) \Rightarrow (x \in \ppp) \vee (y \in \ppp)$.
A \emph{minimal prime ideal} is a prime ideal~$\ppp$ such that for any prime
ideal~$\qqq$ with~$\qqq \subseteq \ppp$, $\qqq = \ppp$.

A subset~$\fff \subseteq A$ is a \emph{prime filter} iff
$0 \not\in \fff$, $(x + y \in \fff) \Rightarrow (x \in \fff) \vee (y \in
\fff)$, $1 \in \fff$ and $(xy \in \fff) \Leftrightarrow (x \in \fff) \wedge (y
\in \fff)$.

An ideal~$\aaa \subseteq A$ is a \emph{radical ideal} iff~$x^n \in \aaa$
implies~$x \in \aaa$ for any~$x \in A$ and any natural number~$n$. The
\emph{radical}~$\sqrt{\aaa}$ of an arbitrary ideal~$\aaa$ is the radical
ideal~$\{ x \in A \,|\, \exists n \in \NN\+ x^n \in \aaa \}$.

A ring is \emph{local} iff a finite sum is invertible only if one of its
summands also is. More precisely, this means that~$1 \neq 0$ and that
whenever~$a + b$ is invertible, $a$ is invertible or~$b$ is invertible.
This definition is an elementary rephrasing of the textbook definition
according to which a ring is local iff it has exactly one maximal ideal. The
elementary definition works better in constructive contexts.

If~$S$ is a multiplicatively closed subset of a ring~$A$, the
\emph{localization}~$A[S^{-1}]$ is the ring of formal fractions~$\frac{a}{s}$ with~$a
\in A$ and~$s \in S$, where two such fractions~$\frac{a}{s}$ and~$\frac{b}{t}$ are deemed
equivalent iff there is an element~$u \in S$ such that~$uta = usb$. We do not
require that~$0 \not\in S$; instead we plainly observe that~$A[S^{-1}]$ is the
zero ring iff~$0 \in S$. We also remark that~$\frac{a}{s} = 0$ in~$A[S^{-1}]$
iff~$ua = 0$ for some element~$u \in S$.

An important special case is \emph{localizing away from an element}: If~$f$ is
an element of~$A$, then~$A[f^{-1}]$ is~$A[\{1,f,f^2,f^3,\ldots\}^{-1}]$. This
ring is the zero ring iff~$f$ is nilpotent.

The \emph{stalk at a prime ideal}~$\ppp \subseteq A$ is the
localization~$A_\ppp \defeq A[(A \setminus \ppp)^{-1}]$. Assuming de Morgan's
laws, such stalks are local rings. The \emph{stalk at a prime filter}~$\fff \subseteq A$ is
defined as~$A_\fff \defeq A[\fff^{-1}]$ and is a local ring even
constructively. Conversely, if a localization~$A[S^{-1}]$ is local then~$S$ can
be saturated to a prime filter.

A property of rings is \emph{localization-stable} iff it is inherited by any
localization. For instance, if~$A$ is a reduced ring, then so is any
localization~$A[S^{-1}]$.

\subsection{A remarkable sheaf}
\label{sect:a-remarkable-sheaf}

Let~$A$ be a ring. Then there is a certain locale~$X$ on which a
certain ``mirror image'' of~$A$ exists, a sheaf~$A^\sim$ of rings. On the one hand,
this mirror image is ``close to~$A$'', such that studying~$A^\sim$ pays off for
learning about~$A$. On the other hand, it has much better properties than~$A$
has.

\begin{center}
\vbox{$A^\sim$ is close to~$A$: \\[0.6em]
\framebox{\parbox{0.9\textwidth}{\vspace*{-0.7em}
\begin{Alphlist}[(C)]
\item\label{item:iso}
There is a canonical isomorphism~$A \to A^\sim(X)$ of rings.
\item\label{item:loc-stable}
$A^\sim$ inherits any property of~$A$ which is
localization-stable.$^\star$
\item\label{item:stalks}
A geometric sequent holds for~$A^\sim$ iff$^{\star\!\star}$ it holds for all
stalks~$A_\fff$.
\end{Alphlist}\vspace*{-0.7em}}}}
\hfill $^\star$precise formulation in
Proposition~{\ref{prop:immediate-consequences}}{\quad}\phantom{\,\,}\\
\hfill $^{\star\!\star}$assuming \BPIT for the ``if'' direction{\quad}
\bigskip

\vbox{$A^\sim$ has better properties than~$A$: \\[0.6em]
\framebox{\parbox{0.9\textwidth}{\vspace*{-0.7em}
\begin{Alphlist}[(D)]
\addtocounter{enumi}{3}
\item\label{item:local} $A^\sim$ is a local ring.
\end{Alphlist}

Assuming that~$A$ is reduced: \\[-1.8em]
\begin{Alphlist}[(G)]
\addtocounter{enumi}{4}
\item\label{item:field} $A^\sim$ is a field: $\forall x\?A^\sim\+ \Bigl((\neg(\exists y\?A^\sim\+ xy = 1)) \Rightarrow x = 0\bigr)$.
\item\label{item:negneg-stable} $A^\sim$ has $\neg\neg$-stable equality:
$\forall x,y\?A^\sim\+ \neg\neg(x = y) \Rightarrow x = y$.
\item\label{item:noetherian} $A^\sim$ is \emph{anonymously Noetherian}.
\end{Alphlist}\vspace*{-0.7em}}}}
\end{center}

Recall that when we ascribe some ring-theoretic property to~$A^\sim$, we
are implicitly using the sheaf semantics. For instance, when we say
that~$A^\sim$ is a local ring, we actually mean
\begin{multline*}
  \Sh(X) \models \neg(1 =_{A^\sim} 0) \quad\text{and} \\
  \Sh(X) \models \forall x,y\?A^\sim\+ \bigl((\exists z\?A^\sim\+ (x + y)z
  =_{A^\sim} 1) \Longrightarrow \\
  (\exists
  z\?A^\sim\+ xz =_{A^\sim} 1) \vee (\exists z\?A^\sim\+ yz =_{A^\sim} 1)\bigr).
\end{multline*}

The locale~$X$ on which this theater plays out is none other than the
\emph{spectrum of~$A$} from algebraic geometry, and the sheaf~$A^\sim$ is the
sheaf of regular functions, also denoted~``$\OOO_{\Spec(A)}$''. We will give details on its construction in
Section~\ref{sect:origins}, but first want to reflect on the displayed
statements.

\paragraph{A reification of all stalks.}
Statement~(\ref{item:stalks}) explains that, to a first approximation, the
sheaf~$A^\sim$ is a reification of all the stalks of~$A$ into a single coherent
entity: The sheaf~$A^\sim$ has exactly those properties which are shared by all
the stalks~$A_\fff$. For instance, the sheaf~$A^\sim$ is an integral domain iff
all the stalks are.

But crucially, this metatheorem only pertains to properties which can be
formulated as geometric sequents. Being an integral domain is such a property
($1 = 0 \vdash \bot$, $xy = 0 \seq{x:R,y:R} x = 0 \vee y = 0$), but many others
are not. The reductive power of harnessing~$A^\sim$ stems from the fact
that~$A^\sim$ enjoys properties which are not shared by~$A$, by its stalks or
indeed by any of its localizations or quotients.

\paragraph{Back and forth.}
Statement~(\ref{item:iso}) is a kind of conservativity result.
For instance, if we
want to show that a certain ring element~$x \in A$ is zero, we can just as well
show that its mirror image in~$A^\sim$ is zero.

For this step to be useful, we need to know what~$A^\sim$ looks like.
Statements~(\ref{item:stalks}) and~(\ref{item:loc-stable}) both allow us to
transfer our knowledge about~$A$ to~$A^\sim$; while
statement~(\ref{item:stalks}) requires a syntactic restriction on the property
under consideration, statement~(\ref{item:loc-stable}) requires a semantic one.

Not only does the ring~$A$ have a mirror image as a sheaf, also
any~$A$-module~$M$ does. Its mirror image is the \emph{quasicoherent
sheaf}~$M^\sim$ from algebraic geometry. As with~$A$ and~$A^\sim$, there is a
close connection between~$M$ and~$M^\sim$; for instance, the module~$M$ is
finitely generated over~$A$ iff~$M^\sim$ is finitely generated over~$A^\sim$
(Exercise~\ref{ex:bridging-modules}).

\paragraph{The field property.}
The fact that~$A^\sim$ is a field is where~$A^\sim$ derives its main usefulness as a reduction technique
from. While statements~(\ref{item:loc-stable}), (\ref{item:stalks}) and
(\ref{item:local}) are immediate consequences of the way~$A^\sim$ is
constructed (Proposition~\ref{prop:immediate-consequences}), the field property~(\ref{item:field})
requires an actual computation for its verification and was a surprising
discovery by Christopher Mulvey in the 1970s.

Myles Tierney commented around that time that the field property ``is surely important,
though its precise significance is still somewhat obscure---as is the case with
many such nongeometric formulas''~\cite[p.~209]{tierney:spectrum}. Even though
the property has been known for a long time, only recently was it put in a
wider context (recognizing it as a small shadow of algebraic geometry's
\emph{quasicoherence}) and its usefulness for constructive algebra
appreciated~\cite[Sections~3.3 and~11.4]{blechschmidt:phd}.

The field property is a unique feature of~$A^\sim$; it cannot be mimicked by
classical techniques in commutative algebra. Several constructions come close,
but fail in other ways:
\begin{enumerate}
\item The stalks~$A_\mmm$ at maximal ideals: These are not fields.
\item The stalks~$A_\ppp$ at minimal prime ideals: These are fields, but a ring
element~$x \in A$ which is zero in all these stalks~$A_\ppp$ is not necessarily
zero in~$A$. Also, properties of the stalks tend to only spread to
dense opens of the spectrum, if they do so at all, while properties of~$A^\sim$
are in close connection with properties of~$A$. For instance, that all
stalks~$M_\ppp$ of an~$A$-module~$M$ at minimal (or, for that matter, all) prime ideals are finitely
generated does not imply that~$M$ is finitely generated, while finite
generation of~$M^\sim$ does imply finite generation of~$M$.
\item The quotients~$A/\mmm$: These are fields, but an element~$x \in A$ which is
zero in all these quotients~$A/\mmm$ is not necessarily zero in~$A$, and
passing to~$A/\mmm$ is not exact (does not preserve injectivity of linear
maps).
\item The quotients~$A/\ppp$ modulo the prime ideals: These are merely integral
domains.
\end{enumerate}

The field condition displayed in statement~(\ref{item:field}) is not the only
field condition used in constructive algebra. Perhaps the most important such
condition is that any element is zero or invertible, not least because this
condition can be expressed as a geometric formula. It is stronger than the
condition that nonunits are zero and is satisfied by~$A^\sim$ iff~$A$ is reduced
and of Krull dimension~$\leq 0$~\cite[Proposition~2.13]{blechschmidt:phd}.

\paragraph{Double negation stability.}\label{par:double-negation-stability}
Generally in constructive mathematics,
two elements being \notnot equal does not imply that they are actually equal:
\begin{enumerate}
\item Sheaves give counterexamples: The interpretation of the statement that
two given sections of a sheaf are \notnot equal with the sheaf semantics is
that they agree on a dense open (Exercise~\ref{ex:negneg-interpretation}). This
does not generally imply that they agree on all of their domain.
\item Realizability gives counterexamples: Realizers of negated statements are
always uninformative; in the case that a statement~$\varphi$ has a realizer at
all, the statement~$\neg\neg\varphi$ is realized by any number whatsoever.
Only in certain cases, where a realizer can be algorithmically reconstructed
from the mere promise that a realizer exists, is~$\neg\neg\varphi \Rightarrow
\varphi$ realizable.
\end{enumerate}

Hence statement~(\ref{item:negneg-stable}) comes at a surprise. It is useful to
import some of classical logic into the constructive setting of sheaves
on~$\Spec(A)$. For instance, the law of excluded middle is constructively valid
in the form
\[ \neg\neg(\varphi \vee \neg\varphi). \]
Combined with the observation
\[ \bigl(\neg\neg\chi \wedge (\chi \Rightarrow \neg\neg\psi)\bigr) \Longrightarrow \neg\neg\psi \]
we can hence freely use the law of excluded middle in constructive arguments --
but at the price that the conclusion~$\neg\neg\psi$ is doubly negated, which is
why typically this observation is not of much use.

However, double negation stability of equality does allow us to conclude~$\psi$
from the a~priori much weaker premiss~$\neg\neg\psi$, if~$\psi$ is of the
form~$(s =_{A^\sim} t)$.

This stability is a consequence of the field property. In case that the
ring~$A$ is not reduced and hence~$A^\sim$ is not a field, a substitute is given
in Exercise~\ref{ex:local-global}. We stress that the proof of
statement~(\ref{item:negneg-stable}) does not require double negation stability
of~$A$ on the metalevel; in fact, it is fully constructive.

\begin{proposition}If~$A$ is reduced, then the sheaf of rings~$A^\sim$ has $\neg\neg$-stable equality:
\[ \Spec(A) \models \forall x,y\?A^\sim\+ \neg\neg(x = y) \Rightarrow x = y. \]
\end{proposition}

\begin{proof}We argue internally in the universe of sheaves, using the sheaf
semantics. Let~$x,y \? A^\sim$ be given such
that~$\neg\neg(x = y)$. We will verify that~$x - y$ is not invertible; by the
field condition, this will imply~$x - y = 0$, hence~$x = y$.

So assume that~$x - y$ is invertible. Then~$\neg(x - y = 0)$, since if~$x - y =
0$, then~$1 = (x-y) (x-y)^{-1} = 0$ in contradiction to~$1 \neq 0$. Hence we have a
contradiction to the assumption~$\neg\neg(x-y=0)$.\end{proof}

\paragraph{The Noetherian property.}

Noetherian conditions are notoriously tricky in constructive algebra, not
least because of their overuse in classical commutative algebra which render
attempts to extract constructive content more challenging.

There are several proposals for constructively sensible definition of
Noetherian rings in the literature on constructive algebra, each with unique
advantages and
disadvantages~\cite{richman:noetherian,mines-richman-ruitenburg:constructive-algebra,perdry:noetherian,perdry:lazy,perdry-schuster:noetherian,tennenbaum:hilbert}.
Insightful comments on why this is so can be found in the introduction and more
specifically on page~27 of the textbook by Henri Lombardi and Claude
Quitté~\cite{lombardi-quitte:constructive-algebra}. To this list of conditions,
we add the following.

\begin{definition}
\label{defn:anonymously-noetherian}
A ring is \emph{anonymously Noetherian} if and only if any of its ideals is
\notnot finitely generated.
\end{definition}

\begin{example}\label{ex:z-noeth}The assumption that any ideal of~$\ZZ$ is finitely generated implies
the law of excluded middle. But we can prove, constructively, that~$\ZZ$ is
anonymously Noetherian: Let~$\aaa$ be an ideal of~$\ZZ$ and assume that~$\aaa$
is not finitely generated. Then~$1 \not\in \aaa$, as else~$\aaa = (1)$ would be
finitely generated. Also~$2 \not\in \aaa$, as else (noting
that~$1\not\in\aaa$),~$\aaa = (2)$. Continuing in this manner, we may
conclude~$n \not\in \aaa$ for every number~$n \geq 1$. Hence~$\aaa = (0)$. But
this is a contradiction to the assumption that~$\aaa$ is not finitely
generated.
\end{example}

Definition~\ref{defn:anonymously-noetherian} is at odds with the idea that constructive mathematics should
be informative: It only expresses that it is impossible for no finite
generating family to exist without requiring that such a family should be given
-- it may remain \emph{anonymous}.\footnote{We borrowed the term ``anonymous''
from type theory, where it is used with a similar meaning (see for instance
Ref.~\bracketedrefcite{kraus-escardo-coquand-altenkirch:anonymous}). However, there is a
subtle difference: The conjunction of ``there is at most one element~$x$'' and
``there is anonymously an element~$x$ in the sense of type theory'' implies
that there actually exists an element~$x$. In contrast, \notnot existence does
not.} And while Hilbert's basis theorem, stating that if a ring~$B$ is
Noetherian then so is the polynomial ring~$B[X]$, does have a constructive
proof for the anonymous version of the Noetherian condition,\footnote{For
instance, a careful reading of the textbook proof given
in~\cite[Theorem~7.5]{atiyah-macdonald:commutative-algebra} shows that the
basis theorem holds intuitionistically as stated.} this result seems empty: For
it just transforms one unconcrete promise (generators of each ideal of~$B$
exist somewhere, platonically) into another (generators for ideals of~$B[X]$
exist somewhere, platonically).

Curiously, the anonymous version of the Noetherian condition is still useful in
constructive algebra when interpreted by the sheaf semantics: This is firstly
because it just so happens that~$A^\sim$ is anonymously Noetherian if~$A$ is reduced,
and secondly because the interpretation of double negation with the sheaf
semantics does have nontrivial content (as evidenced by
statement~(\ref{item:negneg-stable})). The Noetherian condition is also a
unique feature of the sheaf model, not shared by the ring~$A$ or its stalks.

\begin{proposition}\label{prop:anon-noeth}
If~$A$ is reduced, then the sheaf of rings~$A^\sim$ is anonymously
Noetherian.\end{proposition}

\begin{proof}We argue internally, under the sheaf semantics. Let~$\aaa \subseteq A^\sim$ be an ideal.
Assume that~$\aaa$ is not finitely generated. We will verify that then~$\aaa =
(0)$, hence that~$\aaa$ is finitely generated; contradiction.

Let~$x \in \aaa$. If~$x$ is invertible, then~$\aaa = (1)$ is finitely
generated. Hence~$x$ is not invertible. Thus~$x = 0$ because~$A^\sim$ is a
field.
\end{proof}

\begin{remark}The condition for a ring to be (anonymously) Noetherian is a
higher-order condition and hence out of scope of the first-order version of the
sheaf semantics presented in Section~\ref{sect:sheaf-semantics}. However, the
general version of the sheaf semantics is sound with respect to intuitionistic
higher-order logic, hence the proof of Proposition~\ref{prop:anon-noeth} is
valid.
\end{remark}

\subsection{An algebraic origin story}
\label{sect:origins}

The purpose of this section is to properly motivate the construction
leading to~$A^\sim$. Briefly, the story is as follows, with details given
below.
\begin{enumerate}
\item Given a ring~$A$, we set out to construct the \emph{free local ring}
over~$A$.
\item In the strict sense of the word, this endeavor will fail, but in a wider
sense the quest will succeed and yield the sheaf~$A^\sim$.
\item It is then a surprising observation that~$A^\sim$ is a field, even though
we only set out to construct a local ring.
\item The field property of~$A^\sim$ should be appropriately appreciated: While
there is a general machinery for free constructions of this generalized kind, this
machinery is \emph{not} applicable for constructing free fields.
\end{enumerate}

\paragraph{Free constructions.}
Let~$G$ be a group. Then there is a universal way of turning~$G$ into an abelian
group~$G^\ab$ equipped with a group homomorphism~$G \to G^\ab$, the \emph{free
abelian group over~$G$}. This abelianization has the universal property that
for any homomorphism~$G \to M$ into an abelian group, there is exactly one
homomorphism~$G^\ab \to M$ rendering the diagram
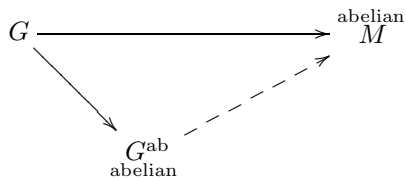
\[ \xymatrix{
  G \ar[rd] \ar[rrr] &&& {\substack{\text{abelian}\\\text{\normalsize$\!M$}\\\phantom{\text{abelian}}}} \\
  & {\substack{\text{\normalsize$G^\ab$}\\\text{abelian}}} \ar@{-->}_[@!29]{}[rru]
} \]
commutative. In textbooks, the abelianization is usually constructed as~$G^\ab
= G/[G,G]$, that is as the quotient by the subgroup generated by the
commutators. This explicit construction makes it appear that abelianization
is a specific result in algebra pertaining to groups.

We can also cast the construction in logical language. This change in perspective makes it evident
that abelianization is just a special case of a more general procedure.
The free abelian group over~$G$ can be obtained as the \emph{term model} of the
theory of abelian groups extended by a constant~$x_g$ for each element~$g \in
G$ and the axioms
\begin{align*}
  \top &\vdash x_{g \circ h} = x_g \circ x_h & \text{(for each~$g,h \in G$)} \\
  \top &\vdash x_{g^{-1}} = x_g^{-1} & \text{(for each~$g \in G$)} \\
  \top &\vdash x_e = e
\end{align*}
That is, $G^\ab$ contains exactly those elements and exactly those
identifications between elements such that~$G^\ab$ is an abelian group and such
that~$G$ can be interpreted in it (that is, that there is a homomorphism~$G \to
G^\ab$), and nothing more.

In exactly the same manner, we can construct the free group over a monoid, the
free~$\ZZ/(2)$-algebra over a ring, the free ring over a set or the free
Heyting algebra over a partially ordered set. This is all well-understood; a
particularly perspicuous account in the general setting of partial Horn
theories is due to Erik Palmgren and Steve Vickers~\cite[Section~5]{palmgren-vickers:partial-horn}.

However, this procedure cannot be carried out for any pair of theories. For
instance, there are no \emph{free fields}. The free field on a set~$M$ would be
a field~$K$ together with a map~$M \to K$ such that any map~$M \to L$ into a field~$L$ factors
over the given map~$M \to K$ by a unique field homomorphism~$K \to L$.

For instance, in the case~$M = \{ t \}$, the field~$\QQ(t)$ of rational
functions comes close, but fails because there are no field homomorphisms~$\QQ(t)
\to L$ which could map~$t$ to an algebraic element. We could also try to
concoct a term model, but the set of terms in the language of a field extended
by a constant~$t$ modulo provable equivalence is only the ring~$\ZZ[t]$, not a
field.

The non-existence of free fields is put into perspective by the observation that the
theory of fields is, for any of its variants, not a partial Horn theory.

\paragraph{Free local rings.} The theory of local rings, that is the theory of
rings extended by the two coherent axioms
\begin{align*}
  1 = 0 &\vdash \bot \\
  (\exists z\?R\+ (x+y)z = 1) &\seq{x:R,y:R} (\exists z\?R\+ xz = 1) \vee
  (\exists z\?R\+ yz = 1),
\end{align*}
is also not a partial Horn theory. Hence there is no reason to expect that free
local rings exist, and indeed, in general they do not. However, despite this
negative outlook, we want to analyze the situation in more detail.

The correct notion of a morphism between local rings is that of a \emph{local
ring homomorphism}, a ring homomorphism~$f : R \to S$ which \emph{reflects
invertibility}, that is for which~$f(x)$ being invertible implies that~$x$ is
invertible.

Hence a \emph{free local ring} over a ring~$A$ is to be a local ring~$A'$
together with a ring homomorphism~$A \to A'$ such that the following universal
property holds: Any ring homomorphism~$A \to B$ into a local ring uniquely
factors via a local ring homomorphism over~$A \to A'$ as indicated in the diagram.
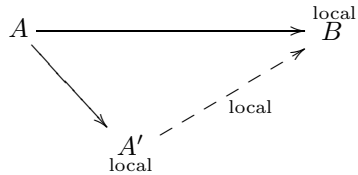
\[ \xymatrix{
  A \ar[rd] \ar[rrr] &&& {\substack{\text{local}\\\text{\normalsize$\!\!\!B$}\\\phantom{\text{local}}}} \\
  & {\substack{\text{\normalsize$\!\!A'$}\\\text{local}}} \ar@{-->}_[@!31]{\text{local}}[rru]
} \]

To develop intuition for what this universal property requests, let a
homomorphism~$f : A \to B$ into a local ring be given. The subset~$B^\times
\subseteq B$ is a \emph{prime filter}, due to the locality of~$B$. Its
preimage~$\fff \defeq f^{-1}[B^\times]$ is a prime filter of~$A$. Hence the
localization~$A_\fff = A[\fff^{-1}]$ is a local ring and~$f$ factors via the
well-defined map~$f' : A_\fff \to B,\,\frac{x}{s} \mapsto f(s)^{-1}f(x)$ over
the localization morphism~$A \to A_\fff,\,x \mapsto \frac{x}{1}$. Furthermore,
this map is local.

Hence the stalks~$A_\fff$, where~$\fff$ ranges over the prime filters of~$A$,
can be regarded as approximations to the hypothetical true local ring over~$A$.
Each of these stalks is local, but might not validate the universal property
for all homomorphisms~$A \to B$ into local rings. We are thus led to the
following conclusions.

\begin{enumerate}
\item A free local ring~$A \to A'$ over~$A$ exists if~$A$ has exactly one prime
filter. (The converse also holds.) In classical mathematics, this condition is
satisfied if and only if~$A$ is local ring of Krull dimension zero. In this
case, the unique prime filter is~$A^\times$ and the free local ring over~$A$ is given
by~$A$ itself. Hence in general, the problem of constructing the free local
ring over a given ring is not solvable.
\item However, we could solve the problem in full generality if we would have a
special prime filter~$\fff_0$ which could somehow shift shape, that is, turn into any
specific prime filter~$\fff$ on demand.
\end{enumerate}

The existence of such a shape-shifting prime filter~$\fff_0$ is not a pipe dream.
While it does not exist as a set, it can be realized as a \emph{sheaf}.

\begin{definition}\label{defn:generic-prime-filter}
The \emph{generic prime filter}~$\fff_0$ of a ring~$A$ is the
generic model of the theory of prime filters of~$A$.\end{definition}

According to this definition, the generic prime filter is a particular
\emph{sheaf model} over the classifying locale of prime filters of~$A$, the
\emph{spectrum}~$\Spec(A)$ of~$A$. By localizing at the generic prime filter, we solve
the problem of constructing the free local ring in a generalized sense: \emph{The
free local ring over~$A$ exists as a certain sheaf of rings.} It has a
universal property not only with respect to all ordinary rings, but also with respect to all
sheaves of rings on arbitrary locales.

\paragraph{Constructing the sheaf~$A^\sim$.} We recall from Section~\ref{sect:sheaves} that the given ring~$A$
has a simple counterpart in the universe of sheaves over~$\Spec(A)$, namely the
constant sheaf~$\underline{A}$. We construct the sheaf~$A^\sim$ as
\[ A^\sim \defeq \underline{A}[\fff_0^{-1}], \]
that is as the localization of~$\underline{A}$ at its prime
filter~$\fff_0$.\footnote{This description contains a slight abuse of notation.
According to Definition~\ref{defn:generic-prime-filter}, the generic prime
filter $\fff_0$ is a model of the theory of prime filters, hence a
family~$(U_x)_{x \in A}$ of opens of~$\Spec(A)$, one for each nullary relation
symbol of the theory. In fact, by the construction of the generic model, the
open~$U_x$ is simply the element~$[D(x)]$ of the Lindenbaum algebra. We
turn this family into the subsheaf of~$\underline{A}$ given by $U \mapsto \{ g
\in \Hom(U,A) \,|\, \forall x \in A\+ g^{-1}[\{a\}] \preceq [D(a)] \}$. Abusing
notation, we denote this subsheaf also as~``$\fff_0$''. The
statement~``$\fff_0$ is a prime filter'' is then validated by the sheaf semantics.}

With this definition, a number of statements from
Section~\ref{sect:a-remarkable-sheaf} are immediate.

\begin{proposition}\label{prop:immediate-consequences}
Let~$A$ be a ring. Then the statements~(\ref{item:loc-stable}), (\ref{item:stalks}) and
(\ref{item:local}) displayed above hold:
\begin{Alphlist}[(D)]
\item[(\ref{item:loc-stable})]
The ring~$A^\sim$ has every property~$\varphi$ that~$A$ has, assuming
that~$\varphi$ is a property of rings which can be put in the formal
higher-order language of toposes; that there is an intuitionistic proof
that~$\varphi$ is localization-stable; and that every element of~$A$ is
nilpotent or not.
\item[(\ref{item:stalks})]
If a geometric sequent holds for~$A^\sim$, then it also holds for all
stalks~$A_\fff$ at prime filters. The converse holds if~$\Spec(A)$ is spatial,
for instance if~\BPIT is available.
\item[(\ref{item:local})] The ring~$A^\sim$ is local.
\end{Alphlist}
\end{proposition}

\begin{proof}
\begin{Alphlist}[(D)]
\item[(\ref{item:local})] The ring~$A^\sim$ is local as it is the
localization at a prime filter.
\item[(\ref{item:loc-stable})] The assumption on~$A$ implies that~$\Spec(A)$ is
\emph{overt}. For overt locales, it is a standard fact that any (infinitary)
first-order property of sets passes to the induced constant sheaves. Hence the
claim follows by observing that~$A^\sim$ is a localization of~$\underline{A}$.
\item[(\ref{item:stalks})] We recall that the points of~$\Spec(A)$ are
precisely the prime filters of~$A$. The sheaf-theoretic stalk of the constant
sheaf~$\underline{A}$ at any point is precisely~$A$, and by construction, the
stalk of the generic prime filter~$\fff_0$ (considered as a subsheaf
of~$\underline{A}$) at a prime filter~$\fff$ is~$\fff$. Since localization is a
geometric construction, it is preserved by the operation of computing
(sheaf-theoretic) stalks. Hence the stalk of~$A^\sim$ at a point~$\fff$ is
$A[\fff^{-1}]$, that is, the ring-theoretic stalk of~$A$ at~$\fff$. The claim
therefore follows from Proposition~\ref{prop:at-points}.\qedhere
\end{Alphlist}
\end{proof}

Furthermore, the sheaf~$A^\sim$ is the free local ring over~$A$ in the
following sense. If~$R$ is a sheaf of rings on a locale~$X$ and~$S$ is a sheaf
of rings on a locale~$Y$, then by a morphism~$R \to S$ we mean a pair
consisting of a morphism~$f : Y \to X$ of locales together with a
morphism~$f^\sharp : f^{-1}R \to S$ of sheaves of rings on~$Y$. (The
sheaf~$f^{-1}R$ is the \emph{pullback sheaf} of~$R$ along~$f$.)
In this way a category of sheaves of rings on arbitrary locales is set up, into
which the category of ordinary rings embeds.

For~$A^\sim$ to be the local ring over~$A$, there has to be a morphism~$A \to
A^\sim$. For this we take the pair~$\langle g, g^\sharp \rangle$ where~$g :
\Spec(A) \to \pt$ is the unique morphism into the one-point locale
and~$g^\sharp : g^{-1}A = \underline{A} \to A^\sim$ is the localization
morphism.

If~$\langle f,f^\sharp \rangle : A \to B$ is a morphism into a local sheaf of
rings on an arbitrary locale~$X$, we can construct, in the universe of sheaves
over~$X$, the factorization~$\underline{A} \to
\underline{A}[\fff^{-1}] \to B$ where~$\fff \defeq
(f^\sharp)^{-1}[B^\times]$. Since~$\fff$ is a model of the theory of prime
filters of~$A$, there is a unique morphism~$g : X \to \Spec(A)$ of locales such
that~$g^{-1}\fff_0 = \fff$. Since~$g^{-1}A^\sim = \underline{A}[\fff^{-1}]$,
the localization just constructed gives rise to a morphism~$A^\sim \to B$ of
local sheaves of rings.

\begin{remark}There is general topos-theoretic machinery for generalized free
constructions of this kind~\cite{coste:sheaf-representation,cole:spectra}, and
this machinery can be used to construct several kinds of spectra, including the
ordinary spectrum of a ring. However, this machinery does have limitations, and
in particular it cannot be used to construct free fields.
\end{remark}

\begin{remark}Given a ring~$A$, there is the geometric theory of \emph{local
localizations} of~$A$. This is the theory of local rings extended by constants~$e_a$
for each element~$a \in A$ and the following axioms:
\begin{align*}
  \top &\vdash e_0 = 0 \\
  \top &\vdash e_1 = 1 \\
  \top &\vdash e_{a + b} = e_a + e_b & \text{(for each~$a,b \in A$)} \\
  \top &\vdash e_{ab} = e_a e_b & \text{(for each~$a,b \in A$)} \\
  \top &\vdash e_{-a} = -e_a & \text{(for each~$a \in A$)} \\
  \top &\vdash_{x:R} \bigvee_{a \in A} \bigvee_{s \in A} \bigl((\exists z\?R\+
  e_s z = 1) \wedge e_s x = e_a\bigr) \\
  e_a = 0 &\vdash \bigvee_{\substack{s \in A\\sa = 0}} (\exists z\?R\+ e_s z = 1) & \text{(for each~$a \in A$)}
\end{align*}
By the first five listed entries, a model~$R$ of this theory comes equipped
with a ring homomorphism~$f : A \to R$. The remaining two entries ensure that~$R$ is
a localization of~$A$, more precisely that it is the localization of~$A$
at~$f^{-1}R^\times$.

Since this theory is not propositional, Section~\ref{sect:presenting-frames}
does not apply. There is, however, still a classifying \emph{topos} of this theory, and it turns
out that it coincides with the topos of sheaves over~$\Spec(A)$ -- mainly
because the theory of local localizations of~$A$ is equivalent to the theory of
prime filters of~$A$. Under this identification, the generic model of the local
localization is precisely the sheaf~$A^\sim$, and it validates first-order
sequents not expected from a local localization; as repeatedly stressed
(statement~(\ref{item:field})), it is a field, even though the theory of local
localizations does not prove the field condition.
\end{remark}

\subsection{Understanding the sheaf model}
\label{sect:understanding}

In Section~\ref{sect:origins}, we constructed the sheaf~$A^\sim$ as the free
local ring over~$A$, more specifically as the localization of the constant
sheaf~$\underline{A}$ on~$\Spec(A)$ at the generic prime filter~$\fff_0$
of~$A$. Several of the properties mentioned in
Section~\ref{sect:a-remarkable-sheaf} are immediate consequences of this
abstract description.

However, for a more detailed identification of~$A^\sim$ -- especially for
verifying its field property -- we require a more concrete description. In
particular, we require a concrete description of the underlying frame
of~$\Spec(A)$. Since this frame is defined as the Lindenbaum algebra of the
theory of prime filters of~$A$, we need a thorough grasp of the set of sequents proved by
this theory.

Such an understanding is imparted by the following theorem. It is the
central workhorse of our approach.

\begin{theorem}\label{thm:workhorse}
Let~$A$ be a ring. The Lindenbaum algebra of the theory~$\TT$ of prime filters
of~$A$ is canonically isomorphic to the frame of radical ideals in~$A$, via an
isomorphism which maps
\[ [D(f)] \longmapsto \sqrt{(f)}. \]
In particular:
\[
  \textnormal{$\TT$ proves~$D(f) \vdash \bigvee_{i \in I} D(g_i)$} \quad\text{iff}\quad
    f \in \sqrt{(g_i)_i}. \]
\end{theorem}

\begin{proof}The set of radical ideals~$\Rad(A)$ is ordered by inclusion. Its
least element is~$\sqrt{(0)}$, the meet~$\aaa \wedge \bbb$ of two radical
ideals is their intersection and the join~$\bigvee_i \aaa_i$ of radical ideals
is~$\sqrt{\sum_i \aaa_i}$. In this way the radical ideals form a frame.

The Lindenbaum algebra of~$\TT$ is the free frame generated by the symbols~$D(x)$,
where~$x$ ranges over the elements of~$A$, modulo the axioms for a prime
filter. Hence to give a frame homomorphism~$\OOO(\Spec(A)) \to \Rad(A)$, we need
to specify images in~$\Rad(A)$ for each generator in such a way that the
axioms for a prime filter are satisfied. This is achieved by declaring~$D(f)
\mapsto \sqrt{(f)}$:
\begin{align*}
  D(0) &\vdash \bot & \sqrt{(0)} &\subseteq \sqrt{(0)} \\
  D(f+g) &\vdash D(f) \vee D(g) & \sqrt{(f+g)} &\subseteq\textstyle \sqrt{\sqrt{(f)} + \sqrt{(g)}} \\
  \top &\vdash D(1) & \sqrt{(1)} &\subseteq \sqrt{(1)} \\
  D(f) \wedge D(g) \dashv\!&\vdash D(fg) & \sqrt{(f)} \cap \sqrt{(g)} &= \sqrt{(fg)}
\end{align*}
The most interesting of these radical inclusions is probably~$\sqrt{(f)} \cap
\sqrt{(g)} \subseteq \sqrt{(fg)}$, as this inclusion illustrates why passing to the
radical is vital; in general, we do not have~$(f) \cap (g) \subseteq (fg)$.

The frame homomorphism obtained in this manner is surjective, since a preimage
for a radical ideal~$\aaa$ is given by~$\bigvee_{f \in \aaa} [D(f)]$. By
Exercise~\ref{ex:isomorphisms-of-frames}, it
remains to show that the homomorphism reflects the ordering.

To this end, let an element~$f \in A$ and a family~$(g_i)_{i \in I}$ of elements be
given and assume~$\sqrt{(f)} \subseteq \sqrt{\sum_i \sqrt{(g_i)}}$. Then there
are values~$n,m \in \NN$, $i_1,\ldots,i_m \in I$ and~$u_1,\ldots,u_m \in
A$ such that~$f^n = u_1 g_{i_1} + \cdots + u_m g_{i_m}$. We then have the
following chain of entailments.
\[
  D(f) \vdash
  D(f^n) \vdash
  \bigvee_{k=1}^m D(u_k g_{i_k}) \vdash
  \bigvee_{k=1}^m D(g_{i_k}) \vdash
  \bigvee_i D(g_i). \qedhere
\]
\end{proof}

Theorem~\ref{thm:workhorse} shows that any derivation in the theory of prime
filters of~$A$ can be put into a \emph{normal form}. An expression of the
form~``$f^n = u_1 g_{i_1} + \cdots + u_m g_{i_m}$'' can be regarded as an
\emph{algebraic certificate} of the entailment~$D(f) \vdash \bigvee_i D(g_i)$.

\begin{remark}An immediate corollary of Theorem~\ref{thm:workhorse} is
that~$\Spec(A)$ is compact, since if~$f \in \sqrt{(g_i)_i}$ then also~$f \in
\sqrt{(g_{i_1},\ldots,g_{i_m})}$ for suitable indices~$i_1,\ldots,i_m$.
However, the stronger result that~$\Spec(A)$ is
even a \emph{coherent locale} can also be deduced from a general theorem
since the theory of prime filters is a coherent theory
(Exercise~\ref{ex:coherent}).\end{remark}

\begin{remark}\label{rem:consistency}
A further corollary of Theorem~\ref{thm:workhorse} is that the theory of prime
filters of~$A$ is inconsistent (that it, that it proves the sequent~$\top
\vdash \bot$, which can also be written as~$D(1) \vdash D(0)$) if and only
if~$A$ is the trivial ring. Contrapositively, if~$A$ is not the trivial ring,
then the theory is consistent. This observation can be regarded as a finitary
substitute of the statement (equivalent to~\BPIT in general) ``nontrivial rings contain prime filters''.
\end{remark}

\begin{remark}Given the importance of Theorem~\ref{thm:workhorse}, it is
natural to wonder where the proof stems from. Our proof proceeded in an ad hoc
fashion, verifying that the frame of radical ideals has the required properties
-- but where does the idea to use radical ideals come from? In general, determining the set
of provable sequents of a geometric theory is a challenging problem. Our most
efficient tool is probably the theory of \emph{entailment relations}. We refer
to Refs.~\bracketedrefcite{rinaldi-schuster-wessel:edde,cederquist-coquand:entrel} and
the references therein and also strongly recommend the literature on dynamical
methods in commutative
algebra~\bracketedrefcite{coste-lombardi-roy:dynamical-method,coquand-lombardi:logical-approach,coquand:site}.
\end{remark}

Having identified the underlying frame of~$\Spec(A)$, the second workhorse is
an explicit description of the sheaf~$A^\sim$.

\begin{proposition}\label{prop:identification}
Let~$A$ be a ring. For any~$f \in A$, there is a canonical isomorphism
\[ A[f^{-1}] \longrightarrow A^\sim([D(f)]). \]
In particular, for~$f = 1$, we obtain statement~(\ref{item:iso}) from
Section~\ref{sect:a-remarkable-sheaf}.
\end{proposition}

\begin{proof}For~$f \in A$, we set~$S_f \defeq \{ g \in A \,|\, \exists n \in
\NN\+ \exists u \in A\+ f^n = ug \} \subseteq A$. This multiplicative system is
the \emph{saturation} of~$\{ f^0, f^1, \ldots \}$. The localization~$A[S_f^{-1}]$
is canonically isomorphic to~$A[f^{-1}]$, but has the advantage that for
elements~$f,g \in A$ such that~$\sqrt{(f)} = \sqrt{(g)}$, the
rings~$A[S_f^{-1}]$ and~$A[S_g^{-1}]$ are actually \emph{the same}
while~$A[f^{-1}]$ and~$A[g^{-1}]$ are merely canonically isomorphic. This
property eases formalization in set theory; it is otherwise not important.

On the basis of~$\Spec(A)$ given by the opens of the form~$[D(f)]$, we define a partial
presheaf by
\[ [D(f)] \longmapsto A[S_f^{-1}]. \]
This partial presheaf is, on the basis where it is defined, a sheaf. Unraveling
the definitions, this claim boils down to the following basic result in
commutative algebra:\par
\begin{quote}
Let~$B$ be a ring. Let~$1 = f_1 + \cdots + f_m \in B$ be a partition of unity. Let
elements~$s_i \in B[f_i^{-1}]$ be given. Assume that~$s_j = s_k$
in~$B[(f_jf_k)^{-1}]$ for all pairs of indices. Then there is exactly one
element~$s \in B$ such that, for all indices~$i$,~$s = s_i$ in~$B[f_i^{-1}]$.
\end{quote}
The remainder of the argument is by sheaf-theoretic generalities.
\end{proof}

More generally, for any~$A$-module~$M$, we can construct the
localization~$M^\sim \defeq \underline{M}[\fff_0^{-1}]$. In the special case
that~$M$ is~$A$, this definition coincides with our original definition
of~$A^\sim$. The proof of Proposition~\ref{prop:identification} carries over to
show that~$M^\sim([D(f)]) \cong M[f^{-1}]$.

\paragraph{An elementary reformulation.}
As a corollary of Theorem~\ref{thm:workhorse} and
Proposition~\ref{prop:identification}, we can unwind the definitions to recast
the sheaf semantics in entirely explicit algebraic terms, with no locales or
sheaves in sight. Notwithstanding the impredicative nature of locale theory,
the resulting formulation will even make sense in metatheories without a
powerset operation such as~\textsc{czf}~\cite{crosilla:predicativity,aczel-rathjen:cst} or arithmetic
universes~\cite{maietti:au,vickers:sketches}. In
fact, it is then a purely syntactic translation procedure which can be carried
out within~\textsc{pra}.

The approach using locales and sheaves provides two ingredients which go beyond
mere syntax: Firstly, they give motivation for setting up the semantics in
the way we do (Section~\ref{sect:origins}). Secondly, they
conceptualize~$A^\sim$ as a single entity, just as rings like~$A/\mmm$
or~$A_\ppp$ are single entities; nevertheless, after unrolling the definitions,
any statement about~$A^\sim$ is merely a (logically more complex) statement
about~$A$ and its localizations~$A[f^{-1}]$.

By a \emph{formula over a ring element}~$f \in A$, we mean a
first-order formula over the signature which has
\begin{enumerate}
\item a sort~``$A^\sim$'' and function symbols for the structure of a ring,
\item for each~$A$-module~$M$, a sort~``$M^\sim$'' and function symbols for the
structure of an~$A^\sim$-module,
\item for each linear map~$M \to N$, a function symbol~$M^\sim \to N^\sim$,
\item a constant of sort~$A^\sim$ for each element of~$A[f^{-1}]$ and
\item for any~$A$-module~$M$, a constant of sort~$M^\sim$ for each element of~$M[f^{-1}]$.
\end{enumerate}

\begin{corollary}\label{cor:algebraic-reformulation}
Let~$A$ be a ring. Let~$\varphi$ be a formula over~$f \in A$.
Then~$[D(f)] \models \varphi$ in the sense of Table~\ref{table:sheaf-semantics} if and only
if~$f \models \varphi$ in the sense of Table~\ref{table:algebraic-kripke-joyal}.
\end{corollary}

\begin{proof}Induction on the structure of~$\varphi$, harnessing the explicit
descriptions provided by Theorem~\ref{thm:workhorse} and
Proposition~\ref{prop:identification}.\end{proof}

\begin{table}[ht]
  \tbl{A purely algebraic presentation of the sheaf semantics of~$\Spec(A)$}{\begin{tabular}{@{}l@{\ \ }c@{\ \ }l@{}}
    $f \models \top$ &iff& true \\
    $f \models \bot$ &iff& $f$ is nilpotent \\
    $f \models x =_{M^\sim} y$ &iff& $x = y \in M[f^{-1}]$ \\
    $f \models \varphi \wedge \psi$ &iff&
      $f \models \varphi$ and $f \models \psi$ \\
    $f \models \varphi \vee \psi$ &iff&
      there exists a partition~$f^n = fg_1 + \cdots + fg_m$ such that, \\
    &&\quad for each~$i$, $fg_i \models \varphi$ or $fg_i \models \psi$ \\
    $f \models \varphi \Rightarrow \psi$ &iff&
      for all~$g \in A$, $fg \models \varphi$ implies $fg \models \psi$ \\
    $f \models \forall x\?A^\sim\+ \varphi(x)$ &iff&
      for all~$g \in A$ and all $x_0 \in A[(fg)^{-1}]$, $fg \models \varphi(x_0)$ \\
    $f \models \exists x\?A^\sim\+ \varphi(x)$ &iff&
      there exists a partition~$f^n = fg_1 + \cdots + fg_m$ such that, \\
    &&\quad for each~$i$, $fg_i \models \varphi(x_0)$ for some~$x_0 \in A[(fg_i)^{-1}]$
  \end{tabular}}
  \label{table:algebraic-kripke-joyal}
\end{table}

The analogue of Theorem~\ref{thm:basic-properties-sheaf-semantics} for the new
formulation of the sheaf semantics reads as follows.

\begin{theorem}\label{thm:basic-properties-sheaf-semantics-algebraic}
Let~$A$ be a ring. Let~$f \in A$. Let~$\varphi$ be a formula over~$f$.
\begin{enumerate}
\item Let~$g \in A$. If~$f \models \varphi$, then~$fg \models \varphi$.
\item Let~$f^n = fg_1 + \cdots + fg_m$. If~$fg_i \models \varphi$ for all
indices~$i$, then~$f \models \varphi$.
\item Let~$\psi$ be a further formula over~$f$. If~$f \models
\varphi$ and if~$\varphi$ entails~$\psi$ intuitionistically, then also~$f
\models \psi$.
\end{enumerate}
\end{theorem}

\begin{proof}Combination of Theorem~\ref{thm:basic-properties-sheaf-semantics}
and Corollary~\ref{cor:algebraic-reformulation}, or alternatively direct proof
by induction without using any preliminaries from earlier sections of this
contribution.\end{proof}

\paragraph{The field property.} We are now in a position to verify
statement~(\ref{item:field}) of Section~\ref{sect:a-remarkable-sheaf}.

\begin{proposition}\label{prop:field-property}
Let~$A$ be a reduced ring. Then~$A^\sim$ is a field in that
\[ 1 \models \forall x\?A^\sim\+ \bigl((\neg(\exists y\?A^\sim\+ xy = 1)) \Rightarrow
x = 0\bigr). \]
\end{proposition}

\begin{proof}By Table~\ref{table:algebraic-kripke-joyal}, we have to verify the
following claim.
\begin{quote}
For any element~$f \in A$ and any element~$x \in A[f^{-1}]$, \\
${\qquad}$ for any element~$g \in A$, \\
${\qquad\qquad}$ if for any element~$h \in A$, \\
${\qquad\qquad\qquad}$ if there is~$(fgh)^n = \sum_{i=1}^m fghp_i$
such that for each~$i$, \\
${\qquad\qquad\qquad\qquad}$ there exists~$y \in
A[(fghp_i)^{-1}]$ such that \\
${\qquad\qquad\qquad\qquad\qquad}$ $xy = 1$ in~$A[(fghp_i)^{-1}]$, \\
${\qquad\qquad\qquad}$ then~$fgh$ is nilpotent, \\
${\qquad\qquad}$ then~$x = 0$ in~$A[(fg)^{-1}]$.
\end{quote}
Hence let elements~$f \in A$, $x \in A[f^{-1}]$ and~$g \in A$ be given and
assume that for any~$h \in A$, the displayed condition holds. We are to show
that~$x = 0$ in~$A[(fg)^{-1}]$.

Write~$x = \frac{x'}{f}$. For~$h \defeq x'$, we have the
partition~$(fgh)^1 = fgh \cdot 1$ with~$m = 1$ and~$p_1 = 1$. For~$y \defeq
\frac{f}{x'} = \frac{f^2gp_1}{fghp_1}$, we have~$xy = 1$ in~$A[(fghp_1)^{-1}]$.
By assumption, the element~$fgh$ is nilpotent, hence zero. Thus~$x' = 0$
in~$A[(fg)^{-1}]$. This implies that~$x = 0$ in~$A[(fg)^{-1}]$.
\end{proof}

The proof of Proposition~\ref{prop:field-property} visibly demonstrates that
unwinding the definition of the sheaf semantics quickly results in formulas
which are quite convoluted. This behavior is a basic consequence of the fact
that, with the exception of the clauses for~``$\top$'' and~``$\wedge$'', all
clauses in Table~\ref{table:algebraic-kripke-joyal} introduce additional
quantifiers.

This increase in complexity should be properly appreciated: The sheaf semantics
allows us to harness convoluted properties in simple language. The unrolled
monstrosity in the proof of Proposition~\ref{prop:field-property} is not
memorable at all and unlikely to find its way into human-written proofs.
However, in the form~``$A^\sim$ is a field'', it becomes easily accessible.
This state of affairs becomes the more pronounced the more negations and especially
double negations occur in the proof.


\begin{example}\label{ex:negneg-interpretation-unrolled}
For future reference, and also for giving a different
perspective on statement~(\ref{item:negneg-stable}), we unroll here the
interpretation of~$f \models \neg\neg\varphi$. It is:
\begin{quote}
For any element~$g \in A$, \\
${\qquad}$ if for any element~$h \in A$, \\
${\qquad\qquad}$ if~$fgh \models \varphi$, \\
${\qquad\qquad}$ then~$fgh$ is nilpotent, \\
${\qquad}$ then~$fg$ is nilpotent.
\end{quote}
In the case that~$A$ is reduced and~$f$ is the unit element of~$A$, this
condition implies (by setting~$g \defeq 1$): ``If~$h = 0$ is the only element of~$A$
such that~$h \models \varphi$, then~$1 = 0$ in~$A$.''
\end{example}

\subsection{Evaluating test cases}
\label{sect:example-applications}

\paragraph{Size of injective and surjective matrices.}

On page~\pageref{par:appl-constr-alg}, we presented a classical proof of the
following theorem.\par
\begin{quote}
\textbf{Theorem.} Let~$M$ be an injective matrix with more columns than rows
over a reduced ring~$A$. Then~$1 = 0$ in~$A$.
\end{quote}
We are now in a position to show how the reduction technique developed in
Section~\ref{sect:appl} can be used to recast the classical proof so that it
becomes constructive:

\begin{quote}
\textbf{Proof.} Since localization is exact, the matrix~$M$ is also injective when
considered as a matrix over~$A^\sim$ (statement (\ref{item:loc-stable})).
Since~$A^\sim$ is a field, this is a contradiction to basic (intuitionistic)
linear algebra.\footnote{We are referencing the following theorem: ``Let~$M$ be
an injective matrix with more columns than rows over a field~$k$ in the sense
that any element of~$k$ is either zero or invertible. Then~$\bot$.'' This
theorem is proven by repeatedly applying elementary row and column
transformations until the matrix is of the
form~$\left(\begin{smallmatrix}I&0\\0&0\end{smallmatrix}\right)$. By assumption
on the size of~$M$, there is at least one zero column. This contradicts
injectivity.

However, since the ring~$A^\sim$ is only a field in the weaker sense that
nonunits are zero, this theorem does not apply directly. We only have that any
element of~$A^\sim$ is \notnot zero or invertible. We hence employ the
trick described in the section on double negation stability
(page~\pageref{par:double-negation-stability}).}
In other words, $1 \models \bot$. By
Remark~\ref{rem:consistency} or Table~\ref{table:algebraic-kripke-joyal}, this amounts to~$1 = 0$ in~$A$.\qed
\end{quote}

If desired, the proof can be mechanically unwound to a fully explicit proof
which does not refer to the sheaf semantics. In this case, unwinding yields
exactly the proof by Fred Richman~\cite[Theorem~2]{richman:trivial-rings} (only
that he needs one additional paragraph to deal with the possibility that the
ring is not reduced, which we for conciseness simply assumed):

\begin{quote}
\textbf{Proof.} Over the localized
ring~$A[M_{ij}^{-1}]$, the matrix element~$M_{ij}$ is invertible and hence~$M$ can be
put into the form~$\left(\begin{smallmatrix}1&0\\0&M'\end{smallmatrix}\right)$
by elementary row and column transformations. The matrix~$M'$ is injective and
has more columns than rows, hence by induction, applied to this submatrix
over~$A[M_{ij}^{-1}]$, we have~$1 = 0$ in~$A[M_{ij}^{-1}]$. Since~$A$ is
reduced, this amounts to~$M_{ij} = 0$ in~$A$.

Hence~$M$ is the zero matrix. Thus~$M(1,0,\ldots,0)^t = 0$ and by
injectivity~$1 = 0$ in~$A$. \qed
\end{quote}

In exactly the same fashion, the dual statement about surjective matrices can
be proven.

Concluding, one can well ask what the computational content of the equation~$1
= 0$ is. The answer depends on the ring in question. In rings where we know a
priori that~$1 \neq 0$, the two statements are most useful in their
contrapositive form. But in rings of the form~$R[f^{-1}]$ or~$R/(f)$, learning
that~$1 = 0$ can indeed be useful: In the former case, this tells us that~$f$
is nilpotent (and a computational witness of~$1 = 0$ can be transformed into a
number~$n$ such that~$f^n = 0$), and in the latter case, this tells us that~$f$
is invertible.

\paragraph{Grothendieck's generic freeness lemma.} This theorem is a basic
theorem in algebraic geometry, used for instance for developing the theory of
moduli spaces. When phrased for the affine situation in a constructively
sensible way, a simple version of this theorem reads as follows. We recall
from Section~\ref{sect:algebraic-preliminaries} that
the notation~$M[f^{-1}]$ makes sense for any element~$f \in A$; in case~$f$ is
zero or nilpotent, the localized module~$M[f^{-1}]$ will be the trivial module.\par
\begin{quote}
\textbf{Theorem.}
Let~$A$ be a reduced ring. Let~$M$ be a finitely generated~$A$-module.
If~$f = 0$ is the only element of~$A$ such that~$M[f^{-1}]$ is a
free~$A[f^{-1}]$-module, then~$1 = 0$ in~$A$.
\end{quote}
This theorem is particularly interesting for the program of constructivizing
classical commutative algebra, as it is an example where the constructive proof
is much shorter and arguably more perspiciuous than classical proofs (for which
see for instance Refs.~\cite[Lemme~6.9.2]{ega-4-2},
\cite[Thm.~24.1]{matsumura:commutative-ring-theory} or
\cite[Thm.~14.4]{eisenbud:commutative-algebra}, which only cover the case of
Noetherian integral domains, or Refs.~\bracketedrefcite{staats:generic-freeness},
\cite[Tag~051Q]{stacks-project}
which proceed in a series of intermediate steps, reducing to that case).\par
\begin{quote}
\textbf{Proof.} By Example~\ref{ex:negneg-interpretation-unrolled} and
Exercise~\ref{ex:bridging-modules}(\ref{item:free}), the claim amounts to the statement that~$M^\sim$ is \notnot
free. Since~$A^\sim$ is a field, this statement follows from
basic (intuitionistic) linear algebra.\footnote{A standard result of linear
algebra is: ``Any finitely generated vector space over a field is finite free
(that is, has a finite basis).'' The proof proceeds by considering a given
generating family and using the law of excluded middle to determine whether one
of the generators can be expressed using the others or not. In the latter case,
the family is linearly independent and hence a basis; in the former case, the
proof removes the redundant generator and continues by induction.

Intuitionistically, we only have the double negation of the law of excluded
middle available; this is why intuitionistically we only have that finitely
generated vector spaces are \notnot finite free. The field condition required
by the proof is exactly the one satisfied by~$A^\sim$.}\qed\end{quote}
Again, the proof can be unwound if
desired~\cite[Proposition~3]{blechschmidt:generic-freeness}:

\begin{quote}\textbf{Proof.}
We proceed by induction on the length of a given generating family
of~$M$. Let~$M$ be generated by~$(v_1,\ldots,v_m)$.

We show that the family~$(v_1,\ldots,v_m)$ is linearly independent. Let~$\sum_i
a_i v_i = 0$. Over~$A[a_i^{-1}]$, the vector~$v_i \in M[a_i^{-1}]$ is a linear
combination of the other generators. Thus~$M[a_i^{-1}]$ can be generated as
an~$A[a_i^{-1}]$-module by fewer than~$m$ generators. The induction hypothesis,
applied to this module over~$A[a_i^{-1}]$, yields that~$1 = 0$ in~$A[a_i^{-1}]$. Since~$A$ is
reduced, this amounts to~$a_i = 0$.

We finish by using the assumption for~$f = 1$. \qed
\end{quote}
The full version of Grothendieck's generic freeness lemma, and a constructive
proof of it, read as follows.\par
\begin{quote}
\textbf{Theorem.}
Let~$A$ be a reduced ring. Let~$B$ be an~$A$-algebra of finite type. Let~$M$ be
a finitely generated~$B$-module.
If~$f = 0$ is the only element of~$A$ such that
\begin{enumerate}
\item $B[f^{-1}]$ and $M[f^{-1}]$ are free modules over $A[f^{-1}]$,
\item $A[f^{-1}] \to B[f^{-1}]$ is of finite presentation and
\item $M[f^{-1}]$ is finitely presented as a module over~$B[f^{-1}]$,
\end{enumerate}
then~$1 = 0$ in~$A$.

\textbf{Proof.} One can check that the claim is the interpretation by the sheaf
semantics of the following statement about the sheaves: It is \notnot the case that
\begin{enumerate}
  \item $B^\sim$ and $M^\sim$ are free modules over $A^\sim$,
  \item $A^\sim \to B^\sim$ is of finite presentation, and
  \item $M^\sim$ is finitely presented as a module over $B^\sim$,
\end{enumerate}

Item~(1) is by basic linear algebra as in the proof of
simple version of Grothendieck's generic freeness lemma, but a bit more involved as we now need to
deal with countable generating
families~\cite[Theorem~11.16]{blechschmidt:phd}. Item~(2) is because finitely
generated algebras over anonymously Noetherian rings are \notnot finitely
presented (in a description of the form~$R[X_1,\ldots,X_n]/\aaa$, the
ideal~$\aaa$ is \notnot finitely generated).
Item~(3) is because finitely generated modules over anonymously Noetherian are
\notnot finitely presented.\qed
\end{quote}
A variant of this proof where the reference to the Noetherian condition has
been replaced by an explicit computation has been unwound in
Ref.~\bracketedrefcite{blechschmidt:generic-freeness}.

\paragraph{McCoy's theorem.}\label{par:mccoy}
It is a familiar fact from linear algebra that an injective~$(n \times
m)$-matrix~$M$ over a geometric field has at least one invertible~$m$-minor.
Like many proofs in linear algebra, the proof proceeds by first applying
elementary row and column transformations to put~$M$ into the form~$N \defeq
\left(\begin{smallmatrix}I&0\\0&M'\end{smallmatrix}\right)$ where all entries
in the submatrix~$M'$ are zero. This step exploits
that any element is zero or invertible. Since~$M$ is injective, the number of columns of~$M'$ is zero.
Hence~$N = \left(\begin{smallmatrix}I\\0\end{smallmatrix}\right)$ has an
invertible~$m$-minor, and thus~$M$ has well.

A version of \emph{McCoy's theorem} states that this fact
generalizes to arbitrary rings:

\begin{quote}
\textbf{Theorem.} Let~$M \in A^{n \times m}$ be an injective matrix.
Then~$(\Lambda^m M)$, the ideal of~$m$-minors of~$M$, is regular.
\end{quote}

The sheaf-theoretic reduction technique allows us to give a constructive proof
of McCoy's theorem which is close to the linear algebra proof. In case that the
ring~$A$ is reduced, we can even copy the original proof almost word for word;
for the general case, we require the following lemma.

\begin{quote}
\textbf{Lemma.} Let~$M \in A^{n \times m}$ be an injective matrix with~$m > 0$. Assume that all entries
of~$M$ are nilpotent. Then~$1 = 0$ in~$A$.

\textbf{Proof of the lemma.} Let~$(M)$ be the ideal generated by the entries
of~$M$. The assumption implies that there is a exponent~$K \in \NN$ such that~$(M)^K = (0)$.
If~$K \geq 1$, then also~$(M)^{K-1} = (0)$, since for~$x \in (M)^{K-1}$ the
entries of~$M (x,0,\ldots,0)^t$ are elements of~$(M)^K$, hence zero,
and~$M$ is injective. We may thus assume~$K = 0$. So~$1 \in (M)^0 = (0)$. \qed
\end{quote}
Because we want to work in the generality of
a ring which might not be reduced, we cannot use that~$A^\sim$ is a field.
Instead we will employ the generalization established in
Exercise~\ref{ex:gen-field-property}: Any nonunit of~$A^\sim$ is nilpotent. The
proof then proceeds as follows.\par
\begin{quote}
\textbf{Proof of the theorem.} Let~$x \in A$ be such that~$x (\Lambda^m M) = (0)$. We are to
show~$x = 0$. By passing to the factor ring~$A/\ann(x)$, we may
assume~$(\Lambda^m M) = (0)$ and~$x = 1$; the new goal is to verify~$1 = 0$. We
do this by validating~$\bot$ in the sheaf semantics.

We may employ the law of excluded middle a finite number of times because we
are to verify~$\bot$. By repeatedly exploiting that any matrix entry is
invertible or not, we may use row and column transformations to put~$M$ into
the form~$N \defeq \left(\begin{smallmatrix}I&0\\0&M'\end{smallmatrix}\right)$ where all
entries of~$M'$ are not invertible. Since nonunits of~$A^\sim$ are nilpotent, the entries of~$M'$ are
nilpotent, and since~$M'$ is injective and since~$1 \neq 0$, the preceding lemma
shows that the number of columns of~$M'$ is zero.

Hence~$N = \left(\begin{smallmatrix}I\\0\end{smallmatrix}\right)$ and~$(1) =
(\Lambda^m N) = (\Lambda^m M) = (0)$, a contradiction. \qed
\end{quote}

\begin{remark}Not of pragmatic, but of logical interest is that the proof can
also be recast as follows. Instead of reducing to the factor ring~$A/\ann(x)$,
we continue considering~$A$. Working under the sheaf semantics, we then
introduce a variant of the usual negation~$\neg\varphi \equiv (\varphi
\Rightarrow \bot)$, namely~$\negg\varphi \equiv (\varphi \Rightarrow x = 0)$.
We may still use the law of excluded middle for this new negation and hence
exploit that any element~$f$ of~$A^\sim$ is invertible or ``not'' invertible. By
Exercise~\ref{ex:gen-field-property}, this amounts to saying that some power
of~$f$ annihilates~$x$. The remainder of the proof carries over with minor
modifications.
\end{remark}

As with the other examples, the proof can mechanically be unwound to obtain a proof
which is fully explicit:

\begin{quote}
\textbf{Proof of the theorem.} Let~$x \in A$ be such that~$x (\Lambda^m M) = (0)$. We are to
show~$x = 0$. By passing to the factor ring~$A/\ann(x)$, we may
assume~$(\Lambda^m M) = (0)$ and~$x = 1$; the new goal is to verify~$1 = 0$.

If~$m = 0$, the claim is immediate because in this case~$(1) = (\Lambda^m M) =
(0)$. So assume~$m > 0$.

Over the localized ring~$A[M_{ij}^{-1}]$, the matrix entry~$M_{ij}$ is
invertible and hence~$M$ can be put into the
form~$\left(\begin{smallmatrix}1&0\\0&M'\end{smallmatrix}\right)$ by elementary
row and column transformations. The matrix~$M'$ is still injective,
and~$(\Lambda^{m-1} M') = (\Lambda^m M)$, hence we may assume by induction
that~$1 = 0$ in~$A[M_{ij}^{-1}]$.

Hence all entries of~$M$ are nilpotent and the claim follows from the preceding
lemma. \qed
\end{quote}

Other short and explicit proofs of McCoy's theorem are given in
Refs.~\bracketedrefcite{blocki:mccoy},
\cite[Theorem~5.22]{lombardi-quitte:constructive-algebra} and
\cite[Exercise~5.23A(3)]{lam:exercises}. The proof in
Ref.~\cite[Proposition~III.8.2.3]{bourbaki:alg} employs the law of excluded
middle. We were not able to discern that these proofs are just different
presentations of the same idea that ours is based on; these proofs all employ
computations with Cramer-like formulas, while our proof stresses the similarity
to the field case.

It is on purpose that the (sheaf-theoretic version of) our proof contains case
distinctions, to closely mimic the proof in the field case. It is also
possible to do without, as already our unrolled proof demonstrates. Thierry
Coquand and Claude Quitté gave a beautiful proof without case distinctions by
applying a local--global
principle~\cite[Theorem~2.4]{coquand-quitte:constructive-finite-free-resolutions};
if desired, their proof can also be cast in the sheaf-theoretic framework
(Exercise~\ref{ex:local-global}).

\subsection*{Exercises}
\addcontentsline{toc}{subsection}{Exercises}

\begin{exercise}[Noetherian properties of the integers]%
\begin{enumerate}
\item Show that the law of excluded middle holds if the ring~$\ZZ$ of integers
is Noetherian in the sense that any ideal is finitely generated.\smallskip

{\scriptsize\emph{Hint.} For a truth value~$\varphi$, consider the ideal~$\{ x
\in \ZZ \,|\, x = 0 \vee \varphi \}$.\par}
\item Give an example of a space~$X$ and a subsheaf~$F$ of the constant
sheaf~$\underline{\ZZ}$ such that, under the sheaf semantics, the subsheaf~$F$
is an ideal, and such that it is not the case that, again under the sheaf
semantics, the ideal~$F$ is finitely generated.
\item Give a proof that~$\ZZ$ is anonymously Noetherian which, unlike the proof
given in Example~\ref{ex:z-noeth}, uses the anonymous version of the least
number principle of Exercise~\ref{ex:least-number-principle}.
\end{enumerate}
\end{exercise}

\begin{exercise}[Free fields]%
\begin{alphlist}[(b)]
\item A \emph{free field} over a set~$M$ would be
a field~$K$ together with a map~$M \to K$ such that any map~$M \to L$ into a
field~$L$ factors
over the given map~$M \to K$ by a unique field homomorphism~$K \to L$. Prove
that free fields exist for no set~$M$.
\item Let~$k$ be a field. Devise a notion of
\emph{free~$k$-fields} (where a~$k$-field is a field together with a
homomorphism from~$k$) and show that the only set~$M$ such that there is a
free~$k$-field over~$M$ is the empty set.
\end{alphlist}
\end{exercise}

\begin{exercise}[Nontriviality of~$\boldsymbol{A^\sim}$]%
Let~$A$ be a ring. In Proposition~\ref{prop:immediate-consequences}, we gave an
abstract proof that~$A^\sim$ is (local and hence) nontrivial in the sense that~$1
\neq 0$. This proof did not need to assume that~$A$ itself is nontrivial. Give a new proof of this
fact, by verifying directly~$1 \models \neg(1 =_{A^\sim} 0)$ using the
semantics given in Table~\ref{table:algebraic-kripke-joyal}.
\end{exercise}

\begin{exercise}[A generalization of the field property]%
\label{ex:gen-field-property}%
Let~$A$ be a ring. We extend the sheaf semantics to infinitary first-order
formulas by adding the following clauses to
Table~\ref{table:algebraic-kripke-joyal}:
\begin{center}\tablefont
\begin{tabular}{@{}l@{\ \ }c@{\ \ }l@{}}
  $f \models \bigwedge_{k \in K} \varphi_k$ &iff&
    for all indices~$k \in K$, $f \models \varphi_k$ \\
  $f \models \bigvee_{k \in K} \varphi_k$ &iff&
    there exists a partition~$f^n = fg_1 + \cdots + fg_m$ such that, \\
  &&\quad for each~$i$, there is an index~$k \in K$ such that~$fg_i \models \varphi_k$
\end{tabular}
\end{center}
Theorem~\ref{thm:basic-properties-sheaf-semantics-algebraic}
also applies to this extension, and moreover, this extension in sound with
respect to infinitary intuitionistic logic.
\begin{alphlist}[(c)]
\item Prove that:
\[ 1 \models \forall f\?A^\sim\+ \forall x\?A^\sim\+
  \bigl((\exists z\?A^\sim\+ fz = 1) \Rightarrow x = 0\bigr) \Longrightarrow \bigvee_{n \in \NN} (f^n x = 0). \]
\item Deduce that~$A^\sim$ is ``almost a field'' in the sense that any nonunit
is nilpotent.
\item Show that~$A$ is reduced iff~$A^\sim$ is.
\end{alphlist}
\end{exercise}

\begin{exercise}[A local--global principle]%
\label{ex:local-global}%
Let~$A$ be a ring. We introduce an operator~$\nabla$ by
\[ \nabla\varphi \defeqv \bigl(\forall s\?A^\sim\+ ((\varphi \Rightarrow s = 0)
\Rightarrow s = 0\bigr).\]
\begin{alphlist}[(f)]
\item Verify that~$\nabla$ is a \emph{local operator}, that is that
\[ (\varphi \Rightarrow \nabla\varphi), \quad
(\nabla\nabla\varphi \Rightarrow \nabla\varphi) \quad\text{and}\quad
((\nabla\varphi \wedge \nabla\psi) \Leftrightarrow \nabla(\varphi \wedge \psi))
\]
hold under the sheaf semantics.
\item Show that equality is~$\nabla$-stable in that
\[ 1 \models \forall x\?A^\sim\+ \forall y\?A^\sim\+ \nabla(x = y) \Rightarrow x = y. \]

{\scriptsize\emph{Note.} If~$A$ is reduced, then~$\nabla$ coincides
with~$\neg\neg$ and the claim reduces to statement~(\ref{item:negneg-stable})
on page~\pageref{item:negneg-stable}. In general, the operator~$\nabla$ is
related to the notion of \emph{scheme-theoretic dense open
subsets}~\cite[Lemma~9.11]{blechschmidt:phd}.\par}
\item Verify that an element of~$A^\sim$ is regular iff it
is~$\nabla$-invertible, that is:
\[ 1 \models \forall f\?A^\sim\+
  \bigl(
    (\forall s\?A^\sim\+ fs = 0 \Rightarrow s = 0) \Longleftrightarrow
    \nabla(\exists g\?A^\sim\+ fg = 1)
  \bigr)
\]
\item Verify the analogous statement for finitely generated ideals:
\begin{multline*}1 \models \forall f_1\?A^\sim\+ \ldots \forall f_n\?A^\sim\+ \\
  \bigl(
    (\forall s\?A^\sim\+ (\textstyle\bigwedge_{i=1}^n f_is = 0) \Rightarrow s = 0) \Longleftrightarrow
    \nabla(1 \in (f_1,\ldots,f_n))
  \bigr) \end{multline*}
\item Explain how the local--global principle ``if~$(f_1,\ldots,f_n)$ is a
regular ideal and if an ideal~$(g_1,\ldots,g_m)$ is regular when considered
over each localization~$A[f_i^{-1}]$, then~$(g_1,\ldots,g_m)$ is regular
over~$A$'' can be viewed as an instance of modus ponens for local operators
($(\nabla\varphi \wedge (\varphi \Rightarrow \nabla\psi)) \Longrightarrow
\nabla\psi$).
\item Fill in the details, unwind the following proof of McCoy's theorem
(page~\pageref{par:mccoy}) and compare the result with the proof in
Ref.~\cite[Theorem~2.4]{coquand-quitte:constructive-finite-free-resolutions}.

\begin{quote}
\textbf{Theorem.} Let~$M \in A^{n \times m}$ be an injective matrix.
Then~$(\Lambda^m M)$, the ideal of~$m$-minors of~$M$, is regular.

\textbf{Proof.} We verify that, under~$\nabla$, there exists an
invertible~$m$-minor, that is we verify~$\nabla(\bigvee_f (\exists g\?A^\sim\+
fg=1))$, where the disjunction is over the~$m$-minors of~$M$.

Since~$M$ is injective, the ideal generated by the first column is regular,
hence under~$\nabla$ the unit ideal. Since our goal is~$\nabla$-stable, we may assume
that it actually is the unit ideal. Hence one of the entries in the first
column is invertible. Applying elementary row and column transformations, we
can put~$M$ into the
form~$\left(\begin{smallmatrix}1&0\\0&M'\end{smallmatrix}\right)$. Continuing
in this fashion, we obtain the
form~$\left(\begin{smallmatrix}I\\0\end{smallmatrix}\right)$. This matrix
obviously has an invertible~$m$-minor, and hence the original one has one as
well.\qed
\end{quote}
\end{alphlist}
\end{exercise}

\begin{exercise}[Bridging modules and their induced sheaves]%
\label{ex:bridging-modules}%
Let~$M$ be a module over a ring~$A$.
\begin{alphlist}[(c)]
\item Show that~$M$ is finitely generated iff~$M^\sim$ is.
\item\label{item:free} Show that~$M$ is finite locally free (that is, that there is a
partition~$1 = f_1 + \cdots + f_n$ such that each module~$M[f_i^{-1}]$ is
finite free over~$A[f_i^{-1}]$) iff~$M^\sim$ is finite free.
\item Show that~$M$ is finitely presented iff~$M^\sim$ is.
\end{alphlist}
\end{exercise}

\begin{exercise}[Local injectivity and surjectivity]%
Let~$M$ and~$N$ be modules over a ring~$A$. Let~$f : M \to N$ be a linear map.
\begin{alphlist}[(b)]
\item Show that~$f$ is injective iff~$f^\sim : M^\sim \to N^\sim$ is injective.
\item Show that~$f$ is surjective iff~$f^\sim : M^\sim \to N^\sim$ is surjective.
\end{alphlist}
\end{exercise}

\begin{exercise}[Units of the sheaf model]%
Let~$A$ be a ring. Let~$f \in A$. Let~$x \in A[f^{-1}]$. Show using
Exercise~\ref{ex:global-existence} that~$f \models \exists y\?A^\sim\+ xy
= 1$ if and only if~$x$ is invertible in~$A[f^{-1}]$.
\end{exercise}

\begin{exercise}[Decidability of invertibility]%
Show that it is not the case that any element of~$\ZZ^\sim$ is invertible or
not invertible.\smallskip

{\noindent\scriptsize\emph{Note.} Since~$\ZZ^\sim$ is a (sheaf) model of the
theory of local rings, this observation shows that the theory of
local rings does not prove that invertibility is decidable. Using much the same
technique, this result has also been proven by Thierry
Coquand~\cite{coquand:local-rings}.

More generally, one can show: A ring~$A$ is of Krull dimension~$\leq n$
iff~$A^\sim$ is; and a local ring (such as~$A^\sim$) is of Krull
dimension~$\leq 0$ iff any element is invertible or nilpotent. Hence
invertibility of elements of~$A^\sim$ is decidable iff~$A$ is of Krull
dimension~$\leq 0$. In this case, up to a technical condition, the universe of
sheaves over~$\Spec(A)$ is classical, hence any statement whatsoever is
decidable. \emph{The ring controls the logic:} Assuming that the metatheory is
classical and assuming the technical condition, the universe of sheaves
over~$\Spec(A)$ is governed by classical logic iff~$A^\sim$ is of Krull
dimension~$\leq 0$.

The for these purposes correct constructive notion of Krull dimension can be
found in Refs.~\bracketedrefcite{coquand-ducos-lombardi-quitte:krull-integral,coquand-lombardi-roy:char-krull}.
It has been used to give a short and constructive proof that~$\dim
k[X_1,\ldots,X_n] = n$, where~$k$ is a geometric
field~\cite{coquand-lombardi:krull-dim-polynomial-ring}. Proofs of the stated
relation between the dimension of~$A$ and the dimension of~$A^\sim$ are recorded in
Ref.~\cite[Section~3.4]{blechschmidt:phd}.\par}
\end{exercise}

\begin{exercise}[Triviality of the spectrum]%
Let~$A$ be a ring. Verify, without assuming~\BPIT, that~$\Spec(A)$ is
isomorphic to the one-point locale iff any element of~$A$ is either
invertible or nilpotent.\smallskip

{\noindent\scriptsize\emph{Hint.} Use the description of the unique map~$\OOO(\pt) \to
\OOO(\Spec(A)),\,\varphi \mapsto \bigvee\{\top\,|\,\varphi\}$ computed in
part~(\ref{item:join-top}) of Exercise~\ref{ex:frame-of-radical-ideals}.\par}
\end{exercise}

\begin{exercise}[Locality of the spectrum]%
Let~$A$ be a ring. Show that the spectrum of~$A$ is local in the sense of
Exercise~\ref{ex:local-locales} if and only if~$A$ is local as a ring.
\end{exercise}

\begin{exercise}[Spatiality of the spectrum]%
Let~$A$ be a ring.
\begin{alphlist}[(c)]
\item Verify that~$\Spec(A)$ is spatial if and only if for any element~$x \in
A$ and any ideal~$\aaa \subseteq A$,
\[ x \in \sqrt{\aaa} \quad\textnormal{if}\quad
\text{for all prime filters~$\fff \subseteq A$ with~$x \in \fff$, $\fff \between \aaa$}, \]
where the symbol~``$\between$'' denotes that the two sets have an element in common.
\item Assuming the law of excluded middle, show that~$\Spec(A)$ is spatial if
and only if for any ideal~$\aaa \subseteq A$ the familiar identity
\[ \sqrt{\aaa} = \bigcap_{\ppp \supseteq \aaa} \ppp \]
holds.
\item Conclude that the spectrum of any ring is spatial iff~\BPIT holds.
\end{alphlist}
\end{exercise}

\begin{exercise}[Prüfer domains as valuation domains]%
An \emph{integral domain} is a ring such that~$1 \neq 0$ and such that~$xy = 0$
implies~$x = 0$ or~$y = 0$. A \emph{valuation domain} is an integral domain
such that for any two elements, one divides the other. A \emph{Prüfer domain}
is an integral domain such that any finitely generated ideal~$\aaa$ is locally a
principal ideal (in the sense that there exists a partition~$1 = f_1 + \cdots +
f_n$ such that, for each index~$i$, the ideal~$\aaa[f_i^{-1}]$ is a principal
ideal in~$A[f_i^{-1}]$).
\begin{alphlist}[(d)]
\item Let~$A$ be a ring. Show that~$A^\sim$ is an integral domain if~$A$ is.
Does the converse hold?
\item Let~$A$ be a valuation domain. Show that any matrix over~$A$ can be put
into diagonal form by elementary row and column operations.
\item Let~$A$ be an integral domain. Show that~$A$ is a Prüfer domain if and
only if~$A^\sim$ is a valuation domain.
\item Let~$A$ be a Prüfer domain. Show that any matrix over~$A$ can locally be
put into diagonal form by elementary row and column operations, by applying the
result of part~(b) to~$A^\sim$.
\end{alphlist}
\end{exercise}

\begin{exercise}[A basic version of Kaplansky's theorem]%
\begin{alphlist}[(c)]
\item Let~$A$ be a local ring. Let~$\aaa \subseteq A$ be a finitely generated
ideal such that~$\aaa^2 = \aaa$. Show that~$\aaa = (0)$ or~$\aaa = (1)$.\smallskip

{\scriptsize\emph{Hint.} Nakayama's lemma.\par}
\item Let~$A$ be a local ring. Let~$M \in A^{n \times n}$ be an idempotent
matrix. Verify that~$M$ is similar to a diagonal matrix with entries zero and
one by applying part~(a) to the ideals of~$k$-minors of~$M$. Deduce that the
cokernel of~$M$ is finite free.
\item Let~$A$ be an arbitrary ring. Let~$M \in A^{n \times n}$ be an idempotent
matrix. Show that the cokernel of~$M$ is finite locally free, by applying the
result of part~(b) to~$A^\sim$.
\item Verify, without using the law of excluded middle or that~$A$ is
Noetherian, that an~$A$-module~$M$ is finitely generated and projective if and
only if it is finite locally free.\smallskip

{\scriptsize\emph{Note.} A self-contained solution is given in
Ref.~\bracketedrefcite{blechschmidt:kaplansky}.\par}
\end{alphlist}
\end{exercise}

\section*{Acknowledgments}

We are grateful to Peter Arndt, Andrej Bauer,
Martin Brandenburg, Thierry Coquand, Martín Escardó, Simon Henry, Matthias
Hutzler, Milly Maietti, Marc Nieper-Wißkirchen, Alexander Oldenziel, Anja
Petković, Peter Schuster, Helmut Schwichtenberg, Steve Vickers and Daniel
Wessel for invaluable discussions shaping this work. We are also grateful to
the anonymous referee, whose comments suggested several improvements to the
presentation.

We thank the
organizers Klaus Mainzer, Peter Schuster and Helmut Schwich\-ten\-berg of
\emph{Proof and Computation 2019}, where this work was presented, for their
kind invitation and for creating a beautiful and stimulating meeting. We also
express our thanks to the local coordinator Chuangjie Xu and to all
participants of that meeting.

}

\bibliographystyle{ws-rv-van}
\bibliography{generalized-spaces}

\begin{thebibliography}{107}
\providecommand{\natexlab}[1]{#1}
\providecommand{\url}[1]{\texttt{#1}}
\expandafter\ifx\csname urlstyle\endcsname\relax
  \providecommand{\doi}[1]{doi: #1}\else
  \providecommand{\doi}{doi: \begingroup \urlstyle{rm}\Url}\fi

\bibitem{maietti:modular-correspondence}
M.~Maietti, Modular correspondence between dependent type theories and
  categories including pretopoi and topoi, \emph{Math. Structures Comp. Sci.}
  {\bf 15}\penalty0 (6), \penalty0 1089--1149  (2005).

\bibitem{crosilla:cst-izf}
L.~Crosilla.
\newblock Set theory: constructive and intuitionistic {ZF}.
\newblock In ed. E.~Zalta, \emph{The Stanford Encyclopedia of Philosophy}.
  Metaphysics Research Lab, Stanford University  (2015).
\newblock URL
  \url{https://plato.stanford.edu/archives/sum2015/entries/set-theory-constructive/}.

\bibitem{johnstone:art}
P.~T. Johnstone.
\newblock The art of pointless thinking: a student's guide to the category of
  locales.
\newblock In \emph{Category theory at work (Bremen, 1990)}, number~18 in Res.
  Exp. Math., pp. 85--107, Heldermann  (1991).

\bibitem{johnstone:point}
P.~T. Johnstone, The point of pointless topology, \emph{Bull. Amer. Math. Soc.}
  {\bf 8}\penalty0 (1), \penalty0 41--53  (1983).

\bibitem{vickers:continuity}
S.~Vickers, Continuity and geometric logic, \emph{J. Appl. Log.} {\bf
  12}\penalty0 (1), \penalty0 14--27  (2014).
\newblock URL \url{https://www.cs.bham.ac.uk/~sjv/GeoAspects.pdf}.

\bibitem{vickers:locales-toposes}
S.~Vickers, \emph{Locales and Toposes as Spaces}, In eds. M.~Aiello,
  I.~Pratt-Hartmann, and J.~van Benthem, \emph{Handbook of Spatial Logics}, pp.
  429--496.
\newblock Springer  (2007).
\newblock URL \url{https://www.cs.bham.ac.uk/~sjv/LocTopSpaces.pdf}.

\bibitem{belanger-marquis:pointless}
M.~Bélanger and J.-P. Marquis, Menger and nöbeling on pointless topology,
  \emph{Logic Log. Philos.} {\bf 22}\penalty0 (2)  (2013).

\bibitem{picado-pultr:frames-and-locales}
J.~Picado and A.~Pultr, \emph{Frames and Locales. Topology without points}.
  Front. Math., Birkhäuser  (2012).

\bibitem{scott:bpit}
D.~Scott, Prime ideal theorems for rings, lattices, and boolean algebras,
  \emph{Bull. Amer. Math. Soc}. {\bf 60}\penalty0 (4), \penalty0 390  (1954).

\bibitem{banaschewski-harting:lattice-aspects}
B.~Banaschewski and R.~Harting, Lattice aspects of radical ideals and choice
  principles, \emph{Proc. Lond. Math. Soc.} {\bf s3-50}, \penalty0 385--404
  (1985).

\bibitem{savin:minimal-prime-ideals}
W.~Savin.
\newblock Minimal prime ideals and the axiom of choice (answer on
  mathoverflow).
\newblock URL \url{https://mathoverflow.net/a/98734}  (2012).

\bibitem{johnstone:elephant}
P.~T. Johnstone, \emph{Sketches of an Elephant: A Topos Theory Compendium}.
  Oxford University Press  (2002).

\bibitem{richman:trivial-rings}
F.~Richman, Nontrivial uses of trivial rings, \emph{Proc. Amer. Math. Soc.}
  {\bf 103}, \penalty0 1012--1014  (1988).

\bibitem{lombardi-quitte:constructive-algebra}
H.~Lombardi and C.~Quitté, \emph{Commutative Algebra: Constructive Methods}.
  Springer  (2015).

\bibitem{leinster:introduction}
T.~Leinster, An informal introduction to topos theory, \emph{Publications of
  the nLab}. {\bf 1}\penalty0 (1)  (2011).

\bibitem{moerdijk-maclane:sheaves-logic}
S.~Mac~Lane and I.~Moerdijk, \emph{Sheaves in Geometry and Logic: a First
  Introduction to Topos Theory}. Universitext, Springer  (1992).

\bibitem{sambin:some-points}
G.~Sambin, Some points in formal topology, \emph{Theoret. Comput. Sci.} {\bf
  305}, \penalty0 347--408  (2003).

\bibitem{sambin:ifs}
G.~Sambin.
\newblock Intuitionistic formal spaces -- a first communication.
\newblock In ed. D.~Skordev, \emph{Mathematical Logic and its Applications,
  Proc. Adv. Internat. Summer School Conf., Druzhba, Bulgaria, 1986}, pp.
  187--204. Plenum  (1987).

\bibitem{schuster:formal-zariski}
P.~Schuster, {Formal Zariski topology: positivity and points}, \emph{Ann. Pure
  Appl. Logic}. {\bf 137}\penalty0 (1), \penalty0 317--359  (2006).

\bibitem{cls:spectral-schemes}
T.~Coquand, H.~Lombardi, and P.~Schuster, Spectral schemes as ringed lattices,
  \emph{Ann. Math. Artif. Intell.} {\bf 56}, \penalty0 339--360  (2009).

\bibitem{cls:projective-spectrum}
T.~Coquand, H.~Lombardi, and P.~Schuster, The projective spectrum as a
  distributive lattice, \emph{Cah. Topol. Géom. Différ. Catég.} {\bf
  48}\penalty0 (3), \penalty0 220--228  (2007).

\bibitem{shulman:stack-semantics}
M.~Shulman.
\newblock Stack semantics and the comparison of material and structural set
  theories.
\newblock URL \url{https://arxiv.org/abs/1004.3802}  (2010).

\bibitem{kock:sdg}
A.~Kock, \emph{Synthetic Differential Geometry}, 2 edn. Number 333 in London
  Math Soc. Lecture Note Ser., Cambridge University Press  (2006).

\bibitem{hyland:synthetic-domain-theory}
J.~M.~E. Hyland.
\newblock First steps in synthetic domain theory.
\newblock In eds. A.~Carbonia, M.~Pedicchio, and G.~Rosolini, \emph{Proc. of
  the International Conference held in Como, Italy, 1990}, vol. 1488,
  \emph{Lecture Notes in Math.}, pp. 131--156, Springer  (1991).

\bibitem{bauer:synthetic-computability-theory}
A.~Bauer.
\newblock First steps in synthetic computability theory.
\newblock In eds. M.~Escardó, A.~Jung, and M.~Mislove, \emph{Proc. of the 21st
  Annual Conference on Mathematical Foundations of Programming Semantics}, vol.
  155, \emph{Electron. Notes Theor. Comput. Sci.}, pp. 5--31, Elsevier B.V.
  (2006).

\bibitem{mines-richman-ruitenburg:constructive-algebra}
R.~Mines, F.~Richman, and W.~Ruitenburg, \emph{A Course in Constructive
  Algebra}. Universitext, Springer  (1988).

\bibitem{yengui:constructive-commutative-algebra}
I.~Yengui, \emph{Constructive Commutative Algebra. Projective Modules Over
  Polynomial Rings and Dynamical Gröbner Bases}. Lecture Notes in Math.,
  Springer  (2015).

\bibitem{coste-lombardi-roy:dynamical-method}
M.~Coste, H.~Lombardi, and M.-F. Roy, {Dynamical method in algebra: ef{}fective
  Nullstellensätze}, \emph{Ann. Pure Appl. Logic}. {\bf 111}\penalty0 (3),
  \penalty0 203--256  (2001).

\bibitem{coquand-lombardi:logical-approach}
T.~Coquand and H.~Lombardi, A logical approach to abstract algebra, \emph{Math.
  Structures Comput. Sci}. {\bf 16}\penalty0 (5), \penalty0 885--900  (2006).

\bibitem{rinaldi-schuster-wessel:edde}
D.~Rinaldi, P.~Schuster, and D.~Wessel, Eliminating disjunctions by disjunction
  elimination, \emph{Indag. Math. (N.S.)}. {\bf 29}\penalty0 (1), \penalty0
  226--259  (2018).
\newblock Virtual Special Issue -- L.E.J. Brouwer, fifty years later.

\bibitem{cederquist-coquand:entrel}
J.~Cederquist and T.~Coquand.
\newblock Entailment relations and distributive lattices.
\newblock In eds. S.~Buss, P.~Hájek, and P.~Pudlák, \emph{Logic Colloquium
  '98. Proceedings of the Annual European Summer Meeting of the Association for
  Symbolic Logic, Prague, Czech Republic, August 9--15, 1998}, vol.~13,
  \emph{Lect. Notes Log.}, pp. 127--139. A. K. Peters  (2000).

\bibitem{coquand:site}
T.~Coquand.
\newblock A completeness proof for geometrical logic.
\newblock In eds. P.~Hájek, L.~Valdés-Villanueva, and D.~Westerståhl,
  \emph{Logic, Methodology and Philosophy of Science. Proceedings of the
  Twelfth International Congress}, pp. 79--90, King's College Publications
  (2005).

\bibitem{aczel:russell-prawitz}
P.~Aczel, {The Russell--Prawitz modality}, \emph{Math. Structures Comput. Sci}.
  {\bf 11}\penalty0 (4), \penalty0 541--554  (2001).

\bibitem{maietti:au}
M.~Maietti, Joyal's arithmetic universes as list-arithmetic pretoposes,
  \emph{Theory Appl. Categ.} {\bf 23}\penalty0 (3), \penalty0 39--83  (2010).

\bibitem{vickers:sketches}
S.~Vickers.
\newblock Sketches for arithmetic universes.
\newblock URL \url{https://arxiv.org/abs/1608.01559}  (2016).

\bibitem{blechschmidt-hutzler:macneille}
I.~Blechschmidt and M.~Hutzler.
\newblock {A constructive Knaster--Tarski proof of the uncountability of the
  reals}.
\newblock URL \url{https://arxiv.org/abs/1902.07366}  (2019).

\bibitem{joyal-tierney:galois-theory}
A.~Joyal and M.~Tierney, \emph{An Extension of the Galois Theory of
  Grothendieck}. vol. 309, \emph{Mem. Amer. Math. Soc.}, American Mathematical
  Society  (1984).

\bibitem{simpson:measure}
A.~Simpson, Measure, randomness and sublocales, \emph{Ann. Pure Appl. Logic}.
  {\bf 163}\penalty0 (11), \penalty0 1642--1659  (2012).

\bibitem{bauer:kleene-tree}
A.~Bauer.
\newblock {König's lemma and the Kleene tree}.
\newblock URL
  \url{http://math.andrej.com/wp-content/uploads/2006/05/kleene-tree.pdf}
  (2006).

\bibitem{cederquist-negri:heine-borel}
J.~Cederquist and S.~Negri.
\newblock {A constructive proof of the Heine--Borel covering theorem for formal
  reals}.
\newblock In eds. S.~Berardi and M.~Coppo, \emph{Types for Proofs and
  Programs}, pp. 62--75, Springer  (1996).

\bibitem{vickers:tychonoff}
S.~Vickers, \emph{Some constructive roads to Tychonoff}, In eds. L.~Crosilla
  and P.~Schuster, \emph{From Sets and Types to Analysis and Topology: Towards
  Practicable Foundations for Constructive Mathematics}, pp. 223--238.
\newblock Oxford Univ. Press  (2005).

\bibitem{wraith:localic-groups}
G.~Wraith, Localic groups, \emph{Cah. Topol. Géom. Différ. Catég.} {\bf 22},
  \penalty0 61--66  (1981).

\bibitem{wraith:galois-topos}
G.~Wraith, Galois theory in a topos, \emph{J. Pure Appl. Algebra}. {\bf 19},
  \penalty0 401--410  (1980).

\bibitem{banaschewski-mulvey:gelfand}
B.~Banaschewski and C.~Mulvey, {A globalisation of the Gelfand duality
  theorem}, \emph{Ann. Pure Appl. Logic}. {\bf 137}, \penalty0 62--103  (2006).

\bibitem{coquand-spitters:gelfand}
T.~Coquand and B.~Spitters, {Constructive Gelfand duality for
  $C^\star$-algebras}, \emph{Math. Proc. Cambridge Philos. Soc.} {\bf 147}
  (2009).

\bibitem{henry:gelfand}
S.~Henry.
\newblock {Constructive Gelfand duality for non-unital commutative
  $C^\star$-algebras}.
\newblock URL \url{https://arxiv.org/abs/1412.2009}  (2014).

\bibitem{butterfield-hamilton-isham:bohr}
J.~Butterfield, J.~Hamilton, and C.~Isham, A topos perspective on the
  kochen--specker theorem, i. quantum states as generalized valuations,
  \emph{Internat. J. Theoret. Phys.} {\bf 37}\penalty0 (11), \penalty0
  2669--2733  (1998).

\bibitem{heunen-landsman-spitters:bohr}
C.~Heunen, N.~Landsman, and B.~Spitters, A topos for algebraic quantum theory,
  \emph{Comm. Math. Phys.} {\bf 291}\penalty0 (1), \penalty0 63--110  (2009).

\bibitem{henry:bohr}
S.~Henry.
\newblock A geometric bohr topos.
\newblock URL \url{https://arxiv.org/abs/1502.01896}  (2015).

\bibitem{fourman-scott:sheaves-and-logic}
M.~Fourman and D.~Scott.
\newblock Sheaves and logic.
\newblock In eds. M.~Fourman, C.~Mulvey, and D.~Scott, \emph{Applications of
  sheaves}, vol. 753, \emph{Lecture Notes in Math.}, pp. 302--401, Springer
  (1979).

\bibitem{picado-pultr:covariant}
J.~Picado and A.~Pultr, \emph{Locales mostly treated in a covariant way}.
  vol.~41, \emph{Textos Mat. Sér. B}, Universidade de Coimbra, Departamento de
  Matemática  (2008).

\bibitem{rathjen:barr}
M.~Rathjen.
\newblock {Remarks on Barr's theorem: proofs in geometric theories}.
\newblock In eds. D.~Probst and P.~Schuster, \emph{Concepts of Proof in
  Mathematics, Philosophy, and Computer Science}, vol.~6, \emph{Ontos Math.
  Log.}, pp. 347--374. De Gruyter  (2016).

\bibitem{lawvere:icm-address}
F.~Lawvere, Quantifiers and sheaves, \emph{Actes Congrès Intern. Math.} {\bf
  1}, \penalty0 329--334  (1970).

\bibitem{tierney:spectrum}
M.~Tierney.
\newblock On the spectrum of a ringed topos.
\newblock In eds. A.~Heller and M.~Tierney, \emph{Algebra, Topology, and
  Category Theory. A Collection of Papers in Honor of Samuel Eilenberg}, pp.
  189--210. Academic Press  (1976).

\bibitem{wraith:intuitionistic-algebra}
G.~Wraith.
\newblock Intuitionistic algebra: some recent developments in topos theory.
\newblock In \emph{Proceedings of the International Congress of Mathematicians
  (1978, Helsinki)}, pp. 331--337, Acad. Sci. Fennica, Helsinki  (1980).

\bibitem{johnstone:topos-theory}
P.~T. Johnstone, \emph{Topos Theory}. vol.~10, \emph{L.M.S. Monographs},
  Academic Press  (1977).

\bibitem{henry:question-barr}
S.~Henry.
\newblock {Barr's theorem and constructivity? (question on MathOverflow)}.
\newblock URL \url{https://mathoverflow.net/questions/142217/}  (2013).

\bibitem{goldblatt:topoi}
R.~Goldblatt, \emph{Topoi: The Categorical Analysis of Logic}. vol.~98,
  \emph{Stud. Logic Found. Math.}, Elsevier  (1984).

\bibitem{caramello:tst}
O.~Caramello, \emph{Theories, Sites, Toposes: Relating and studying
  mathematical theories through topos-theoretic `bridges'}. Oxford University
  Press  (2018).

\bibitem{dyckhoff-negri:geometrisation}
R.~Dyckhoff and S.~Negri, Geometrisation of first-order logic, \emph{Bull.
  Symbolic Logic}. {\bf 21}\penalty0 (2), \penalty0 123--163  (2015).

\bibitem{rathjen-toppel:reduction}
M.~Rathjen and M.~Toppel.
\newblock On relating theories: Proof-theoretical reduction.
\newblock In eds. S.~Centrone, S.~Negri, D.~Sarikaya, and P.~Schuster,
  \emph{Mathesis Universalis, Computability and Proof}, vol. 412,
  \emph{Synthese Library}, pp. 311--331. Springer  (2019).

\bibitem{tarizadeh:flat}
A.~Tarizadeh.
\newblock The flat topology and its duality aspects.
\newblock URL \url{https://arxiv.org/abs/1503.04299v9}  (2015).

\bibitem{johnstone:rings-fields-and-spectra}
P.~T. Johnstone, Rings, fields, and spectra, \emph{J. Algebra}. {\bf
  49}\penalty0 (1), \penalty0 238--260  (1977).

\bibitem{blechschmidt:custom-tailored}
I.~Blechschmidt.
\newblock Exploring mathematical objects from custom-tailored mathematical
  universes.
\newblock In eds. G.~Oliveri, S.~Boscolo, and C.~Ternullo, \emph{Philosophy of
  mathematics. Objects, Structures, and Logics (forthcoming)}. Springer
  (2020).
\newblock URL
  \url{https://rawgit.com/iblech/internal-methods/master/paper-filmat.pdf}.

\bibitem{mulvey:repr}
C.~Mulvey.
\newblock Intuitionistic algebra and representations of rings.
\newblock In eds. K.~H. Hofmann and J.~R. Liukkonen, \emph{Recent Advances in
  the Representation Theory of Rings and $C^*$-algebras by Continuous
  Sections}, vol. 148, \emph{Mem. Amer. Math. Soc.}, pp. 3--57, American
  Mathematical Society  (1974).

\bibitem{caramello:preliminaries}
O.~Caramello.
\newblock Topos-theoretic background.
\newblock URL
  \url{https://www.oliviacaramello.com/Unification/ToposTheoreticPreliminariesOliviaCaramello.pdf}
   (2014).

\bibitem{streicher:ctcl}
T.~Streicher.
\newblock Introduction to category theory and categorical logic.
\newblock URL \url{https://www.mathematik.tu-darmstadt.de/~streicher/CTCL.pdf}
  (2004).

\bibitem{borceux:handbook3}
F.~Borceux, \emph{Handbook of Categorical Algebra: Volume 3, Sheaf Theory}.
  Encyclopedia Math. Appl., Cambridge University Press  (1994).

\bibitem{shulman:categorical-logic}
M.~Shulman.
\newblock {Categorical logic from a categorical point of view (draft for AARMS
  Summer School 2016)}.
\newblock URL \url{http://mikeshulman.github.io/catlog/catlog.pdf}  (2016).

\bibitem{hyland:effective-topos}
M.~Hyland.
\newblock The ef{}fective topos.
\newblock In eds. A.~S. Troelstra and D.~van Dalen, \emph{The L. E. J. Brouwer
  Centenary Symposium}, pp. 165--216, North-Holland  (1982).

\bibitem{phoa:effective}
W.~Phoa.
\newblock An introduction to fibrations, topos theory, the ef{}fective topos
  and modest sets.
\newblock Technical report, University of Edinburgh  (1992).
\newblock URL \url{http://www.lfcs.inf.ed.ac.uk/reports/92/ECS-LFCS-92-208/}.
\newblock ECS-LFCS-92-208.

\bibitem{bauer:c2c}
A.~Bauer.
\newblock Realizability as the connection between computable and constructive
  mathematics.
\newblock URL \url{http://math.andrej.com/asset/data/c2c.pdf}  (2005).

\bibitem{scedrov:forcing}
A.~Ščedrov, \emph{Forcing and Classifying Topoi}. vol. 295, \emph{Mem. Amer.
  Math. Soc.}, American Mathematical Society  (1984).

\bibitem{hamkins-lewis:ittm}
J.~Hamkins and A.~Lewis, Infinite time turing machines, \emph{J. Symbolic
  Logic}. {\bf 65}\penalty0 (2), \penalty0 567--604  (2000).

\bibitem{bauer:injection}
A.~Bauer, An injection from the baire space to natural numbers, \emph{Math.
  Structures Comput. Sci.} {\bf 25}\penalty0 (7), \penalty0 1484--1489  (2015).

\bibitem{normann-sanders:uncountable}
D.~Normann and S.~Sanders, On the mathematical and foundational significance of
  the uncountable, \emph{J. Math. Log.} {\bf 18}\penalty0 (2)  (2018).

\bibitem{richman:fta}
F.~Richman, The fundamental theorem of algebra: a constructive development
  without choice, \emph{Pac. J. Math.} {\bf 196}\penalty0 (1), \penalty0
  213--230  (2000).

\bibitem{ruitenburg:roots}
W.~Ruitenburg.
\newblock Constructing roots of polynomials over the complex numbers.
\newblock In ed. A.~Cohen, \emph{Computational Aspects of Lie Group
  Representations and Related Topics, Proc. of the 1990 Computer Algebra
  Seminar held in Amsterdam}, vol.~84, \emph{CWI Tract}, pp. 107--128, Centrum
  voor Wiskunde en Informatica, Amsterdam  (1991).

\bibitem{bridges-richman-schuster:wcc}
D.~Bridges, F.~Richman, and P.~Schuster, A weak countable choice principle,
  \emph{Proc. Amer. Math. Soc.} {\bf 128}\penalty0 (9)  (2000).

\bibitem{blechschmidt:phd}
I.~Blechschmidt.
\newblock \emph{Using the internal language of toposes in algebraic geometry}.
\newblock PhD thesis, University of Augsburg  (2017).
\newblock URL
  \url{https://rawgit.com/iblech/internal-methods/master/notes.pdf}.

\bibitem{richman:noetherian}
F.~Richman, The ascending tree condition: constructive algebra without choice,
  \emph{Comm. Algebra}. {\bf 31}\penalty0 (4), \penalty0 1993--2002  (2003).

\bibitem{perdry:noetherian}
H.~Perdry, Strongly noetherian rings and constructive ideal theory, \emph{J.
  Symbolic Comput.} {\bf 37}\penalty0 (4), \penalty0 511--535  (2004).

\bibitem{perdry:lazy}
H.~Perdry, {Lazy bases: a minimalist constructive theory of Noetherian rings},
  \emph{MLQ Math. Log. Q.} {\bf 54}\penalty0 (1), \penalty0 70--82  (2008).

\bibitem{perdry-schuster:noetherian}
H.~Perdry and P.~Schuster, Noetherian orders, \emph{Math. Structures in Comput.
  Sci.} {\bf 21}\penalty0 (1), \penalty0 111--124  (2011).

\bibitem{tennenbaum:hilbert}
J.~Tennenbaum.
\newblock \emph{A constructive version of Hilbert's basis theorem}.
\newblock PhD thesis, University of California  (1973).

\bibitem{kraus-escardo-coquand-altenkirch:anonymous}
N.~Kraus, M.~Escardó, T.~Coquand, and T.~Altenkirch, {Notions of anonymous
  existence in Martin--Löf type theory}, \emph{Log. Methods Comput. Sci}. {\bf
  13}\penalty0 (1)  (2017).

\bibitem{atiyah-macdonald:commutative-algebra}
M.~Atiyah and I.~Macdonald, \emph{Introduction to Commutative Algebra}.
  Addison--Wesley  (1969).

\bibitem{palmgren-vickers:partial-horn}
E.~Palmgren and S.~Vickers, Partial horn logic and cartesian categories,
  \emph{Ann. Pure Appl. Logic}. {\bf 145}\penalty0 (3), \penalty0 314--353
  (2007).

\bibitem{coste:sheaf-representation}
M.~Coste.
\newblock Localisation, spectra and sheaf representation.
\newblock In eds. M.~Fourman, C.~Mulvey, and D.~Scott, \emph{Applications of
  sheaves}, vol. 753, \emph{Lecture Notes in Math.}, pp. 212--238, Springer
  (1979).

\bibitem{cole:spectra}
J.~Cole, The bicategory of topoi and spectra, \emph{Repr. Theory Appl. Categ.}
  {\bf 25}, \penalty0 1--16  (2016).

\bibitem{crosilla:predicativity}
L.~Crosilla, \emph{Exploring predicativity}, In eds. K.~Mainzer, P.~Schuster,
  and H.~Schwichtenberg, \emph{Proof and Computation}, pp. 83--108.
\newblock World Scientific  (2018).

\bibitem{aczel-rathjen:cst}
P.~Aczel and M.~Rathjen.
\newblock Constructive set theory (book draft).
\newblock URL \url{https://www1.maths.leeds.ac.uk/~rathjen/book.pdf}  (2010).

\bibitem{ega-4-2}
J.~Dieudonné and A.~Grothendieck, {Éléments de géométrie algébrique: IV.
  Étude locale des schémas et des morphismes de schémas, Seconde partie},
  \emph{Publ. Math. Inst. Hautes Études Sci.} {\bf 24}  (1965).

\bibitem{matsumura:commutative-ring-theory}
H.~Matsumura, \emph{Commutative Ring Theory}. vol.~8, \emph{Cambridge Stud.
  Adv. Math.}, Cambridge University Press  (1987).

\bibitem{eisenbud:commutative-algebra}
D.~Eisenbud, \emph{Commutative Algebra with a View Toward Algebraic Geometry}.
  vol. 150, \emph{Grad. Texts in Math.}, Springer  (1995).

\bibitem{staats:generic-freeness}
C.~Staats.
\newblock Elementary proof of generic freeness.
\newblock URL
  \url{https://math.uchicago.edu/~cstaats/Charles_Staats_III/Notes_and_papers_files/generic_freeness.pdf}
   (2011).

\bibitem{stacks-project}
{The Stacks Project Authors}.
\newblock {Stacks Project}.
\newblock URL \url{https://stacks.math.columbia.edu/} .

\bibitem{blechschmidt:generic-freeness}
I.~Blechschmidt.
\newblock {An elementary and constructive proof of Grothendieck's generic
  freeness lemma}.
\newblock URL \url{https://arxiv.org/abs/1807.01231}  (2018).

\bibitem{blocki:mccoy}
Z.~Błocki, {An elementary proof of the McCoy theorem}, \emph{Univ. Iagel. Acta
  Math.} {\bf 30}, \penalty0 215--218  (1993).

\bibitem{lam:exercises}
T.~Lam, \emph{Exercises in Modules and Rings}. Problem Books in Math., Springer
   (2007).

\bibitem{bourbaki:alg}
N.~Bourbaki, \emph{Algebra I, Chapters 1--3}. Addison-Wesley  (1973).

\bibitem{coquand-quitte:constructive-finite-free-resolutions}
T.~Coquand and C.~Quitté, Constructive finite free resolutions,
  \emph{Manuscripta Math.} {\bf 137}, \penalty0 331--345  (2012).

\bibitem{coquand:local-rings}
T.~Coquand.
\newblock A remark about the theory of local rings.
\newblock URL \url{http://www.cse.chalmers.se/~coquand/local.pdf}  (2008).

\bibitem{coquand-ducos-lombardi-quitte:krull-integral}
T.~Coquand, L.~Ducos, H.~Lombardi, and C.~Quitté, {Constructive Krull
  dimension I: integral extensions}, \emph{J. Algebra Appl.} {\bf 8}\penalty0
  (1), \penalty0 129--138  (2009).

\bibitem{coquand-lombardi-roy:char-krull}
T.~Coquand, H.~Lombardi, and M.-F. Roy, \emph{An elementary characterisation of
  Krull dimension}, In \emph{From Sets and Types to Analysis and Topology:
  Towards Practicable Foundations for Constructive Mathematics}, pp. 239--244.
\newblock Oxford Univ. Press  (2005).

\bibitem{coquand-lombardi:krull-dim-polynomial-ring}
T.~Coquand and H.~Lombardi, {A short proof for the Krull dimension of a
  polynomial ring}, \emph{Amer. Math. Monthly}. {\bf 112}\penalty0 (9),
  \penalty0 826--829  (2005).

\bibitem{blechschmidt:kaplansky}
I.~Blechschmidt.
\newblock Vector bundles on affine schemes (short note).
\newblock URL \url{https://www.ingo-blechschmidt.eu/kaplansky-en.pdf}  (2015).

\end{thebibliography}

\end{document}